\documentclass[a4paper,12pt]{preprint}
\usepackage[margin=1.3in]{geometry}  % set the margins to 1in on all sides
\usepackage[full]{textcomp}
\usepackage[osf]{newtxtext}

\usepackage{mathalfa}
\usepackage{microtype}

\usepackage{booktabs}

\usepackage{enumitem}
\usepackage{amsmath}
\usepackage{amsfonts} 
\usepackage{amssymb}             % for blackboard bold, etc

\usepackage{comment}
\usepackage{breakurl}

\usepackage{mathrsfs}
\usepackage{booktabs}
\usepackage{mathtools}

\usepackage{tikz}
\usepackage{amsthm}                % better theorem environment
\usepackage{mhequ}
\usepackage{hyperref}

\hypersetup{pdfstartview=}

\addtocontents{toc}{\protect\hypersetup{hidelinks}}

\DeclareSymbolFont{sfoperators}{OT1}{ptm}{m}{n}
\DeclareSymbolFontAlphabet{\mathsf}{sfoperators}

\makeatletter
\def\operator@font{\mathgroup\symsfoperators}
\makeatother
\newcommand{\eqdef}{\stackrel{\mbox{\tiny def}}{=}}

\numberwithin{equation}{section}

\newtheorem{thm}{Theorem}[section]
\newtheorem{defn}[thm]{Definition}
\newtheorem{lem}[thm]{Lemma}
\newtheorem{prop}[thm]{Proposition}

\newtheorem{assumption}[thm]{Assumption}
\theoremstyle{remark}
\newtheorem{remark}[thm]{Remark}
\newtheorem{rmk}[thm]{Remark}

\makeatletter
\def\th@newremark{\th@remark\thm@headfont{\bfseries}}

\def\bdiamond{\mathop{\mathpalette\bdi@mond\relax}}
\newcommand\bdi@mond[2]{%
	\vcenter{\hbox{\m@th
			\scalebox{\ifx#1\displaystyle 2.6\else1.8\fi}{$#1\diamond$}%
	}}%
}

\def\bDiamond{\mathop{\mathpalette\bDi@mond\relax}}
\newcommand\bDi@mond[2]{%
	\vcenter{\hbox{\m@th
			\scalebox{\ifx#1\displaystyle 2.6\else1.2\fi}{$#1\Diamond$}%
	}}%
}
\makeatletter

\definecolor{darkgreen}{rgb}{0.1,0.7,0.1}
\definecolor{darkred}{rgb}{0.7,0.1,0.1}
\definecolor{darkblue}{rgb}{0,0,0.7}
\newcommand{\EE}{\mathbb{E}}     %mathematical expectation

\newcommand{\VV}{\mathbb{V}}

\newcommand{\aA}{\mathcal{A}}
\newcommand{\bB}{\mathcal{B}}
\newcommand{\cC}{\mathcal{C}}
\newcommand{\dD}{\mathcal{D}}
\newcommand{\eE}{\mathcal{E}}
\newcommand{\fF}{\mathcal{F}}
\newcommand{\gG}{\mathcal{G}}
\newcommand{\hH}{\mathcal{H}}
\newcommand{\iI}{\mathcal{I}}
\newcommand{\jJ}{\mathcal{J}}

\newcommand{\lL}{\mathcal{L}}
\newcommand{\mM}{\mathcal{M}}
\newcommand{\nN}{\mathcal{N}}
\newcommand{\oO}{\mathcal{O}}
\newcommand{\pP}{\mathcal{P}}

\newcommand{\rR}{\mathcal{R}}
\newcommand{\sS}{\mathcal{S}}
\newcommand{\tT}{\mathcal{T}}
\newcommand{\uU}{\mathcal{U}}
\newcommand{\vV}{\mathcal{V}}
\newcommand{\wW}{\mathcal{W}}

\newcommand{\zZ}{\mathcal{Z}}

\newcommand{\fR}{\mathfrak{R}}

\newcommand{\fx}{\mathfrak{x}}
\newcommand{\fy}{\mathfrak{y}}
\newcommand{\fz}{\mathfrak{z}}

\makeatletter % Guarantees same font as sin, cos, etc. Don't know why they don't show up in Times though...
\newcommand{\cov}{{\operator@font cov}}
\newcommand{\var}{{\operator@font var}}
\newcommand{\corr}{{\operator@font corr}}
\newcommand{\diam}{{\operator@font diam}}
\newcommand{\Av}{{\operator@font Av}}
\newcommand{\trig}{{\operator@font trig}}
\newcommand{\Enh}{{\operator@font Enh}}
\newcommand{\EEnh}{\overline {\operator@font Enh}}
\makeatother

\newcommand{\KPZ}{\text{\tiny KPZ}}

\newcommand{\E}{\mathbf{E}}
\newcommand{\K}{\mathbf{K}}
\newcommand{\N}{\mathbf{N}}
\newcommand{\R}{\mathbf{R}}
\newcommand{\T}{\mathbf{T}}

\newcommand{\X}{\mathbf{X}}

\newcommand{\Z}{\mathbf{Z}}
\newcommand{\0}{\mathbf{0}}
\newcommand{\1}{{\color{symbols}{\mathbf{1}}}}
\newcommand{\one}{\mathbf{1}}

\renewcommand{\k}{\mathbf{k}}
\newcommand{\m}{\mathbf{m}}
\newcommand{\n}{\mathbf{n}}
\newcommand{\p}{\mathbf{p}}
\newcommand{\q}{\mathbf{q}}
\renewcommand{\r}{\mathbf{r}}
\newcommand{\s}{\mathbf{a}}
\renewcommand{\u}{\mathbf{u}}

\newcommand{\x}{\mathbf{x}}
\newcommand{\y}{\mathbf{y}}
\newcommand{\z}{\mathbf{z}}
\newcommand{\sA}{\mathscr{A}}
\newcommand{\sB}{\mathscr{B}}
\newcommand{\sD}{\mathscr{D}}

\newcommand{\sM}{\mathscr{M}}
\newcommand{\sR}{\mathscr{R}}
\newcommand{\sT}{\mathscr{T}}
\newcommand{\sU}{\mathscr{U}}

\newcommand{\PPi}{\boldsymbol{\Pi}}
\newcommand*\Bell{\ensuremath{\boldsymbol\ell}}
\newcommand*\Bbeta{\ensuremath{\boldsymbol\beta}}
\newcommand*\Btheta{\ensuremath{\boldsymbol\theta}}

\newcommand*\Bvarphi{\ensuremath{\boldsymbol{\varphi}}}
\newcommand*{\Bphi}{\ensuremath{\boldsymbol{\phi}}}

\def\set{{\mathfrak{u}}}

\newcommand{\ha}{\widehat{a}}

\newcommand{\hF}{\widehat{F}}
\newcommand{\hP}{\widehat{P}}

\newcommand{\hPi}{\widehat{\Pi}}
\newcommand{\hPPi}{\widehat{\PPi}}
\newcommand{\hxi}{\widehat{\xi}}

\def\Ind{\mathscr{I}}
\def\Clus{\mathscr{C}}
\def\ddeg{\mathbf{deg}}

\newcommand{\eps}{\varepsilon}
\newcommand{\Ups}{\Upsilon}

\colorlet{symbols}{blue!90!black}
\colorlet{testcolor}{green!60!black}

\def\${|\!|\!|}
\def\ex{{\mathrm{ex}}}

\usetikzlibrary{shapes.misc}
\usetikzlibrary{shapes.symbols}
\usetikzlibrary{snakes}
\usetikzlibrary{decorations}
\usetikzlibrary{decorations.markings}

\def\drawx{\draw[-,solid] (-3pt,-3pt) -- (3pt,3pt);\draw[-,solid] (-3pt,3pt) -- (3pt,-3pt);}
\tikzset{
	root/.style={circle,fill=testcolor,inner sep=0pt, minimum size=2mm},
	dot/.style={circle,fill=black,inner sep=0pt, minimum size=1mm},
	edot/.style={circle,fill=black,inner sep=0pt, minimum size=1mm},
	odot/.style={circle,draw=black,inner sep=0pt, minimum size=1mm},
	var/.style={circle,fill=black!10,draw=black,inner sep=0pt, minimum size=
	2mm},
    svar/.style={circle,fill=black!10,draw=black,inner sep=0pt, minimum size=
	1.5mm},
    noise0/.style={rectangle,draw=symbols,fill=white,inner sep=0pt, minimum size=1.5mm},
    noise1/.style={circle,draw=symbols,fill=white,inner sep=0pt, minimum size=1.5mm},
    noise2/.style={circle,draw=symbols,fill=symbols,inner sep=0pt, minimum size=1.5mm},
	dotred/.style={circle,fill=symbols!50,inner sep=0pt, minimum size=2mm},
	generic/.style={semithick,shorten >=1pt,shorten <=1pt},
	ageneric/.style={semithick},
	dist/.style={ultra thick,draw=testcolor,shorten >=1pt,shorten <=1pt},
	testfcn/.style={ultra thick,testcolor,shorten >=1pt,shorten <=1pt,<-},
	testfcnx/.style={ultra thick,testcolor,shorten >=1pt,shorten <=1pt,<-,
		postaction={decorate,decoration={markings,mark=at position 0.6 with {\drawx}}}},
	kepsilon/.style={semithick,shorten >=1pt,shorten <=1pt,densely dashed,->},
	kprimex/.style={semithick,shorten >=1pt,shorten <=1pt,densely dashed,->,
		postaction={decorate,decoration={markings,mark=at position 0.4 with {\drawx}}}},
	kernel/.style={semithick,shorten >=1pt,shorten <=1pt,->},
	akernel/.style={semithick,->},
	multx/.style={shorten >=1pt,shorten <=1pt,
		postaction={decorate,decoration={markings,mark=at position 0.5 with {\drawx}}}},
	kernelx/.style={semithick,shorten >=1pt,shorten <=1pt,->,
		postaction={decorate,decoration={markings,mark=at position 0.4 with {\drawx}}}},
	kernel1/.style={->,semithick,shorten >=1pt,shorten <=1pt,postaction={decorate,decoration={markings,mark=at position 0.45 with {\draw[-] (0,-0.1) -- (0,0.1);}}}},
	kernel2/.style={->,semithick,shorten >=1pt,shorten <=1pt,postaction={decorate,decoration={markings,mark=at position 0.45 with {\draw[-] (0.05,-0.1) -- (0.05,0.1);\draw[-] (-0.05,-0.1) -- (-0.05,0.1);}}}},
	kernelBig/.style={semithick,shorten >=1pt,shorten <=1pt,decorate, decoration={zigzag,amplitude=1.5pt,segment length = 3pt,pre length=2pt,post length=2pt}},
	gepsilon/.style={dotted,semithick,shorten >=1pt,shorten <=1pt},
	renorm/.style={shape=circle,fill=white,inner sep=1pt},
	labl/.style={shape=rectangle,fill=white,inner sep=1pt},
	xi/.style={circle,fill=symbols!10,draw=symbols,inner sep=0pt,minimum size=1.2mm},
	xix/.style={crosscircle,fill=symbols!10,draw=symbols,inner sep=0pt,minimum size=1.2mm},
	xib/.style={circle,fill=symbols!10,draw=symbols,inner sep=0pt,minimum size=1.6mm},
	xibx/.style={crosscircle,fill=symbols!10,draw=symbols,inner sep=0pt,minimum size=1.6mm},
	not/.style={circle,fill=symbols,draw=symbols,inner sep=0pt,minimum size=0.5mm},
	>=stealth,
  	highlight/.style={line width=7pt,blue,draw opacity=0.2,line cap=round,line join=round},
  	cover/.style={line width=7pt,blue,line cap=round,line join=round},
	smalldot/.style={circle,fill=symbols,draw=symbols, solid,inner sep=0pt,minimum size=0.5mm},
	}

\makeatletter
\def\DeclareSymbol#1#2#3{\expandafter\gdef\csname MH@symb@#1\endcsname{\tikz[baseline=#2,scale=0.15,draw=symbols]{#3}}\expandafter\gdef\csname MH@symb@#1s\endcsname{\scalebox{0.5}{\tikz[baseline=#2,scale=0.15,draw=symbols]{#3}}}}
\def\<#1>{\csname MH@symb@#1\endcsname}
\makeatother

\DeclareSymbol{Xi22}{0.5}{\draw (0,0) node[xi] {} -- (-1,1) node[not] {} -- (0,2) node[xi] {}; }

\DeclareSymbol{0'}{0}{\node at (0,0.7) [noise0] {}; }
\DeclareSymbol{1'}{0}{\node at (0,0.7) [noise1] {}; }
\DeclareSymbol{2'}{0}{\node at (0,0.7) [noise2] {}; }

\DeclareSymbol{0'0}{0}{\draw (0,-0.5) node[not] {} -- (0,1.5) node[noise0] {};}
\DeclareSymbol{1'0}{0}{\draw (0,-0.5) node[not] {} -- (0,1.5) node[noise1] {};}
\DeclareSymbol{2'0}{0}{\draw (0,-0.5) node[not] {} -- (0,1.5) node[noise2] {};}
\DeclareSymbol{1'1'}{0}{\draw (-1,1.5) node[noise1] {}  -- (0.5,0) node[noise1] {};}
\DeclareSymbol{2'1'}{0}{\draw (-1,1.5) node[noise2] {}  -- (0.5,0) node[noise1] {};}
\DeclareSymbol{2'1'0}{0}{\draw (1.5,2) node[noise2] {}  -- (0,0.5) node[noise1] {} -- (1.5,-1) node[not] {};}
\DeclareSymbol{2'2'0}{-2}{\draw (-1,1) node[noise2] {}  -- (0,-0.5) node[smalldot] {} -- (1,1) node[noise2] {};}
\DeclareSymbol{2'1'1'}{-1}{\draw (0,1.4) node[noise2] {}  -- (1.5,0.5) node[noise1] {} -- (0,-0.4) node[noise1] {};}
\DeclareSymbol{2'0'}{0}{\draw (-1,1.5) node[noise2] {}  -- (0.5,0) node[noise0] {};}
\DeclareSymbol{2'2'0'}{0}{\draw (-1,1.5) node[noise2] {}  -- (0,0) node[noise0] {} -- (1,1.5) node[noise2] {};}

\DeclareSymbol{0}{0}{\draw[white] (-.4,0) -- (.4,0); \draw (0,0)  -- (0,1.2) node[smalldot] {};}
\DeclareSymbol{1}{0}{\draw[white] (-.4,0) -- (.4,0); \draw (0,0)  -- (0,1.2) node[smalldot] {};}
\DeclareSymbol{2}{0}{\draw (-0.5,1.2) node[smalldot] {} -- (0,0) -- (0.5,1.2) node[smalldot] {};}
\DeclareSymbol{11}{0}{\draw (0,1.8) node[smalldot] {} -- (-0.7,0.9) -- (0,0) -- (0.7,1) node[smalldot] {};}
\DeclareSymbol{10}{0}{\draw (0,1.8) node[smalldot] {} -- (-0.8,0.9) -- (0,0);}
\DeclareSymbol{21}{0.7}{\draw (-1,1.8) node[smalldot] {} -- (-0.5,0.9); \draw (0,1.8) node[smalldot] {} -- (-0.5,0.9) -- (0,0) -- (0.5,0.9) node[smalldot] {};}
\DeclareSymbol{20}{0.7}{\draw (-1,1.8) node[smalldot] {} -- (-0.5,0.9); \draw (0,1.8) node[smalldot] {} -- (-0.5,0.9) -- (-0.5,0);}
\DeclareSymbol{210}{1.1}{\draw (-0.5,2.4) node[smalldot] {} -- (-1,1.6);\draw (-1.5,2.4) node[smalldot] {} -- (-0.5,0.8); \draw (0,1.6) node[smalldot] {} -- (-0.5,0.8) -- (-0.5,0);}
\DeclareSymbol{211}{1.1}{\draw (-0.5,2.4) node[smalldot] {} -- (-1,1.6);\draw (-1.5,2.4) node[smalldot] {} -- (-0.5,0.8); \draw (0,1.6) node[smalldot] {} -- (-0.5,0.8) -- (0,0) -- (0.5,0.8) node[smalldot] {};}
\DeclareSymbol{22}{0.7}{\draw (-1.5,1.8) node[smalldot] {} -- (-1,0.9) -- (0,0) -- (1,0.9) -- (1.5,1.8) node[smalldot] {}; \draw (-0.5,1.8) node[smalldot] {} -- (-1,0.9);\draw (0.5,1.8) node[smalldot] {} -- (1,0.9);}

\setlist[itemize]{topsep=3pt,itemsep=1.5pt,parsep=0pt}

\def\scal#1{\langle#1\rangle}
\def\cent#1{\mathopen{{\langle\kern-0.3em\rangle}}#1\mathclose{{\langle\kern-0.3em\rangle}}}
\def\bscal#1{\big\langle#1\big\rangle}

\def\d{\partial}

\begin{document}

\title{Large-scale limit of interface fluctuation models}
\author{Martin Hairer$^1$ and Weijun Xu$^2$}
\institute{Imperial College London, UK, \email{m.hairer@imperial.ac.uk}
\and University of Warwick, UK / NYU Shanghai, China, \email{weijunx@gmail.com}}

\maketitle

\begin{abstract}
	We extend the weak universality of KPZ in \cite{HQ} to weakly asymmetric interface models with general growth mechanisms beyond polynomials. A key new ingredient is a pointwise bound on correlations of trigonometric functions of Gaussians in terms of their polynomial counterparts. This enables us to reduce the problem of a general nonlinearity with sufficient regularity to that of a polynomial. 
\end{abstract}

\setcounter{tocdepth}{2}
\microtypesetup{protrusion=false}
\tableofcontents
\microtypesetup{protrusion=true}

\def\k{\mathbf{k}}

\section{Introduction}

\subsection{Weak universality of KPZ}

The weak universality conjecture for KPZ states that any ``reasonable'' weakly asymmetric interface 
fluctuation model in $1+1$ dimensions should rescale to the KPZ equation \cite{KPZ86}, formally given by
\begin{equ} \label{e:KPZ}
\d_{t} h = \d_{x}^{2} h + a (\d_x h)^2 + \xi. 
\end{equ}
Here, $\xi$ denotes space-time white noise on the one-dimensional torus $\T$, and $a \in \R$ is a coupling constant describing the strength of the asymmetry. The term ``reasonable'' refers to the following three features in the microscopic model: 
\begin{itemize}
	\item There is a smoothing mechanism which erodes high peaks and fills deep valleys.
	\item The fluctuation mechanism depends on the slope of the interface in a nontrivial way. 
	\item The system is influenced by a random fluctuation with short range correlations. 
\end{itemize}
These features are clearly visible in the macroscopic equation \eqref{e:KPZ}, represented by the Laplacian, the nonlinearity $(\partial_x h)^{2}$, and the noise $\xi$ respectively. The additional requirement, namely that of the microscopic model being ``weakly asymmetric'', is also essential for its large scale limit to be given by the KPZ equation. 
It refers to that the strength of the growth at microscopic level should be very weak, and is tuned according to the scale at which one looks at the system. The presence of this weak parameter features the so-called ``crossover regime''. 

In fact, for completely symmetric models (that is, the strength of the growth is $0$), it is widely believed 
that their scaling limit is described by the stochastic heat equation, so that they belong to the 
Edwards--Wilkinson universality class \cite{EW}, which exhibits Gaussian fluctuations at large scales. 
On the other hand, if the strength of asymmetry is of order $1$, then it is believed that the large scale behaviour of 
such models is described by the ``KPZ fixed point'' and they are said to 
belong to the KPZ universality class. Very recently, the 
breakthrough by \cite{KPZ_fixed_pt} established the convergence of TASEP to the KPZ fixed point, which in particular 
yields a complete characterisation of the latter. There has also been substantial progress in the understanding 
of various statistics for other models in this class 
(see for example \cite{ACQ,Macdonald,fluc_expo} and references therein). So far however, all 
these results strongly rely on the presence of a suitable ``integrable structure'' in the model, 
so that the underlying reason for the universal behaviour is still unclear from a mathematical perspective. 

A natural related question is to investigate the crossover regime -- that is, the object(s) that lie in between these two 
universality classes. This is where the KPZ equation comes in. The equation was first derived in \cite{KPZ86} for 
height functions of droplets, and is expected to be a universal object for weakly asymmetric growth models at large 
scales. What is usually referred to as the ``weak KPZ universality conjecture'' can be interpreted as saying
that the KPZ equation \eqref{e:KPZ} is the only object that interpolates between the Edwards--Wilkinson and 
KPZ universality classes in the sense that solutions to the KPZ equations are expected to be the only
stationary space-time Markov process that converges to the stochastic heat equation when ``zooming in'' and the
KPZ fixed point when ``zooming out''. Unfortunately, such a characterisation of the KPZ equation appears to
be far out of reach of current techniques. One possible step in this direction which has been investigated in
recent years is to exhibit a class of models that is as large as possible and depends on a parameter $\eps$ 
tuning their asymmetry in the sense that models with $\eps = 0$ belong to the EW universality class and
models with $\eps \neq 0$ (are expected to) belong to the KPZ universality class. The weak universality conjecture 
then suggests that if one simultaneously sends $\eps$ to $0$ and considers such a system at a large scale 
$L(\eps)$, then there exists a specific choice of $L$ such that one observes convergence to the KPZ equation
as $\eps \to 0$.

One mathematical obstacle towards the understanding of such a claim is that \eqref{e:KPZ} is ill-posed. 
Nevertheless much progress has been made towards this conjecture over the past years. The first rigorous statement 
was in the seminal work \cite{BG97}, where the authors showed that the height function of weakly asymmetric 
simple exclusion processes (WASEP) rescale to \eqref{e:KPZ}. The proof uses the Hopf--Cole transform of the KPZ equation, which amounts to 
showing that the exponential of the system converges to the multiplicative stochastic heat equation. There have been other recent works in this direction, see for example \cite{CS,CST,Corwin_Tsai,Dembo_Tsai,Cyril_KPZ}. The systems considered in these works are all related to WASEP and rely on its various structures. Since most particle systems do not have the same structures, and they are usually not very well behaved  under exponentiation, it is not clear how the methods employed in these works would generalise to other not so closely related situations. 

In \cite{Energy}, the authors introduced the notion of an ``energy solution'' to the KPZ equation in equilibrium. A 
slightly stronger notion of solution was subsequently introduced in \cite{Energy_unique} and shown to be unique and 
coincide with the Hopf--Cole solution. This characterisation of the solutions to the KPZ equation was
very fruitful in showing ``weak universality'' results as discussed above for a variety of models. 
In particular, this includes a large class of particle systems generalising WASEP but requiring much
less structure \cite{Energy}, systems of interacting Brownian motion \cite{DGP}, 
as well as stochastic PDE models \cite{HQ_stationary}. 
The only drawback of this technique is that it requires a priori knowledge of the invariant measure
of the microscopic model. Furthermore, it only appears to cover Markovian models with the additional property that
the decomposition $L = S+A$ of the generator into a symmetric and an antisymmetric part is such that 
both $A$ and $S$ are local operators (or at least ``almost local'').

In \cite{HairerKPZ,Hai14a}, a pathwise notion of solution was developed based on the theory of rough 
paths / regularity structures. (See also \cite{GIP12,Reloaded} for a pathwise approach using paracontrolled calculus
that is quite different in its technical implementation but very similar in spirit.) This approach avoids 
using the Hopf--Cole transform as well as the use of 
the invariant measure, so that it covers non-Markovian models in particular. 
It also allows to show the stability of quite general discrete approximations \cite{general_discrete}, including some 
standard ones which had been shown to be stable earlier \cite{Etienne,KPZCLT}. In addition, a universality 
result was shown in \cite{HQ} for quite a large class of continuous random PDE models with polynomial growth 
mechanism (see below for details). The present article pursues this line of investigation and shows that 
the polynomial growth mechanism that was essential in the proof of  \cite{HQ} can be replaced by an arbitrary
sufficiently smooth function with subpolynomial growth at infinity.

\subsection{Main result}

In \cite{HQ}, the authors explored continuous microscopic growth models in weak asymmetry regime of the type
\begin{equ} \label{e:micro_model}
\partial_{t} h = \partial_{x}^{2} h + \sqrt{\eps} F(\partial_{x} h) + \hxi, 
\end{equ}
where $F$ is an even polynomial and $\hxi$ is a space-time Gaussian random field with smooth short range covariance that integrates to $1$. They showed that the large scale behaviour of $h$, when properly rescaled and re-centered, is described by the KPZ equation. One surprising fact revealed by this result is that the coupling constant $a$ depends on all coefficients of the polynomial $F$. The aim of this article is to remove the polynomial requirement on the nonlinearity. Our precise assumption on $F$ is the following.

\begin{assumption} \label{as:F}
	The function $F:\R \rightarrow \R$ is symmetric. Furthermore, there exist $\alpha \in (0,1)$ and constants $C,M > 0$ such that $F \in \cC^{7+\alpha}$ and satisfies the bounds
	\begin{equ}
	\sup_{0 \leq \ell \leq 7} |F^{(\ell)}(u)| \leq C(1+|u|)^{M}, \quad \sup_{|h|<1} \frac{|F^{(7)}(u+h) - F^{(7)}(u)|}{|h|^{\alpha}} \leq C (1+|u|)^{M}
	\end{equ}
	for all $u \in \R$. 
\end{assumption}

\begin{rmk}
The above assumption on $F$ implies certain decay of the local distributional norm of its Fourier transform, as stated in Proposition~\ref{pr:F_decomp}. In fact, the bounds \eqref{e:F_smooth} and \eqref{e:F_remainder} in that proposition are the only two properties of $F$ we will use. 
\end{rmk}

One specific example of function $F$ that satisfies our assumptions is given by $F(u) = \sqrt{1+u^2}$.

\begin{assumption} \label{as:noise}
There exists a compactly supported smooth function $\rho \colon \R^2 \to \R$
which is symmetric in its space variable $x$ and integrates to $1$ (in space-time), and such that 
\begin{equ} \label{e:mollified_noise}
\hxi \stackrel{\text{\tiny law}}{=} \xi * \rho\;, 
\end{equ}
where $\xi$ denotes space-time white noise.
\end{assumption}

\begin{rmk}
The assumption on the symmetry of $\rho$ in its space variable is mainly for technical convenience. Without the symmetry assumption, one also needs to shift the frame horizontally in order to establish the convergence. But that will not change the limiting equation, and we keep this symmetry assumption here for technical simplicity. 
\end{rmk}

We now give our main result. Let $\widehat{\Psi} = P' * \hxi$, where $P'$ denotes the spatial derivative of the heat kernel $P$, and $*$ is the convolution in space-time. Then, $\widehat{\Psi}$ is stationary, and we use $\mu$ to denote its law at any space-time point. We define the constant $a$ to be
\begin{equ} \label{e:coupling_const}
a := \frac{1}{2} \E F''(\widehat{\Psi}) = \frac{1}{2} \int_{\R} F''(x) \mu(dx). 
\end{equ}
Let $h$ be the process \eqref{e:micro_model}, and define
\begin{equ}
h_{\eps}(t,x) := \eps^{\frac{1}{2}} h(t/\eps^{2}, x/\eps) - C_{\eps} t. 
\end{equ}
Then, $h_{\eps}$ solves the equation
\begin{equ} \label{e:macro_eq_appro}
\partial_{t} h_{\eps} = \partial_{x}^{2} h_{\eps} + \eps^{-1} F(\eps^{\frac{1}{2}} \partial_{x}h_{\eps}) + \xi_{\eps} - C_{\eps}, 
\end{equ}
where $\xi_{\eps}(t,x) = \eps^{-\frac{3}{2}} \hxi(t/\eps^{2}, x/\eps)$ is an approximation (in law) to the space-time white noise $\xi$ at scale $\eps$. There is a slight abuse of notation here since $h$ itself also depends on $\eps$, but since we are mostly working with the rescaled process $h_{\eps}$, the omission of $\eps$ in the microscopic process should not create confusion. 

From now, rather than considering \eqref{e:macro_eq_appro} on all of $\R$, we will assume that the space variable
takes values in the one-dimensional torus $\T$ of size $1$.
This means that the original equation \eqref{e:micro_model} and hence the noise $\hxi$ are actually defined 
on the torus of size $\eps^{-1}$ and Assumption~\ref{as:noise} should be interpreted accordingly. 
The constant $a$ defined in \eqref{e:coupling_const} however still uses the process defined on the whole space.
 
Returning to \eqref{e:macro_eq_appro}, if $F$ has sufficient regularity one can expand the nonlinearity as
\begin{equ}
\eps^{-1} F(\eps^{\frac{1}{2}} \d_{x} h_{\eps}) = \frac{a_0}{\eps} + a_{1} (\d_{x} h_{\eps})^{2} + a_{2} \eps (\d_{x} h_{\eps})^{4} + \cdots. 
\end{equ}
The term $\frac{a_0}{\eps}$ can be killed by the choice of the large constant $C_{\eps}$ in \eqref{e:macro_eq_appro}. Since each of the higher powers of $\d_{x} h_{\eps}$ ($\geq 4$) is multiplied by a positive power of $\eps$, it seems that only the quadratic term would survive in the limit $\eps \rightarrow 0$, and one might expect that $h_{\eps}$ converges to the KPZ equation with coupling constant $a_1$. However, as already shown in \cite{HQ} for polynomial $F$, this is simply \textit{not} the case. The main result of this article is an extension of \cite{HQ} to nonlinearities $F$ that satisfy Assumption~\ref{as:F}. We state it as follows.

\begin{thm} \label{th:main_intro}
Let $h_{\eps}$ be the solution to \eqref{e:macro_eq_appro} with initial data $h_{0}^{(\eps)}$, where the 
nonlinearity $F$ satisfies Assumption~\ref{as:F}. Suppose there exists 
$\eta \in (\frac{1}{2}-\frac{1}{M+4}, \frac{1}{2})$ and $h_0 \in \cC^{\eta}(\T)$ such that for some 
$\gamma \in (\frac{3}{2}, \frac{5}{3})$, $\|h_{0}^{(\eps)}; h_{0}\|_{\gamma,\eta;\eps} \rightarrow 0$ in the 
sense of \eqref{e:s_space_funct}. Then, there exists $C_{\eps} \rightarrow +\infty$ such that for every $T>0$, 
$h_{\eps}$ converges in probability in $\cC^{\eta}([0,T], \T)$ to the Hopf--Cole solution to the KPZ equation 
with coupling constant $a$ given by \eqref{e:coupling_const}. 
\end{thm}

\begin{rmk}
	The large constant $C_{\eps}$ is of the form $C_{\eps} = \frac{\ha}{\eps} + \oO(1)$, where $\ha = \E F(\widehat{\Psi})$. This can be easily deduced by combining the definition of the renormalisation constants \eqref{e:renorm_const} and their behaviours in Section~\ref{sec:renorm_const}, the form of $C_{\eps}$ in Theorem~\ref{th:renorm_eq}, and the convergence of renormalised models in Theorem~\ref{th:main_converge}. 
\end{rmk}

In the case of polynomial $F$, one can see from \eqref{e:coupling_const} that $a$ is a linear combination of all its coefficients (except the zero-th term). This suggests that all the higher powers have contributions to the limit rather than simply vanishing. For the $F$ we consider in this article, even though it is in general not infinitely differentiable to have a power series expansion, the combined effects of all these ``higher powers'' still exist, and is given explicitly by \eqref{e:coupling_const}. 

Similar universality questions have been studied in the context of $\Phi^4_3$ equation. In \cite{Phi4_poly}, the authors considered the $3$D microscopic phase coexistence models of the type
\begin{equ}
\d_{t} u = \Delta u - \eps V_{\theta}'(u) + \widehat{\xi}
\end{equ}
for an even polynomial potential $V$ near a critical point $\theta = 0$ and a smooth Gaussian field $\widehat{\xi}$. It was shown that the large scale behaviour of the field $u$ was described by the $\Phi^4_3$ equation. This result was extended to non-Gaussian noise \cite{Phi4_non_Gaussian} and general even potential with Gaussian noise \cite{Phi4_general}. 

The difficulties of extending from polynomial to general nonlinearities are essentially the same in both situations: 
one needs to control arbitrary high moments of a general function of a Gaussian field (or a more general random 
field). The methods developed in this article and the ones in \cite{Phi4_general} however are very different. 
We include a brief discussion on this towards the end of Section~\ref{sec:strategy}. 

\subsection{Possible generalisations}

We discuss two possible generalisations of the result in this article. 

\begin{enumerate}
\item \textbf{Regularity of $F$.} Our assumption of $7+$ differentiability of $F$ comes from the form of the bound we develop in this article and it is clear that this requirement is not optimal. In fact, the expression 
\eqref{e:coupling_const} suggests that it may be possible to take any $F$ that is Lipschitz continuous
and not growing too fast at infinity. 
	
One very interesting example would be $F(x) = |x|$, but it is not clear at this stage whether our technique could
be sharpened to include that case. Since there is a big gap between Assumption~\ref{as:F} and Lipschitzness, we expect that new ideas are needed to treat low regularity functions. 
	
\item \textbf{Non-Gaussian noise.} The essential part of the paper where we use Gaussianity of the noise is in 
Section~\ref{sec:general_bound}, where we develop a pointwise bound for correlations of trigonometric functions of 
Gaussian fields. If one is able obtain a similar bound for a class of non-Gaussian random fields, then the arguments 
in the rest part of the paper can be used in exactly same way. However it is not clear at this moment how such bounds 
can be obtained for non-Gaussian noises. 
\end{enumerate}

\subsection{Notations}

Throughout the article, we use $M$ to denote the growth of the derivatives of $F$ (in Assumption~\ref{as:F}). For every random variable $X$, we write $\cent{X} = X - \E X$ as its re-centered version. We use $\diamond$ to denote Wick products between Gaussians, for example $X^{\diamond k_1} \diamond Y^{\diamond k_2} \diamond Z^{\diamond k_3}$. We use $A \sqcup B$ to denote the union of two \textit{disjoint} sets $A$ and $B$. Finally, we define the Fourier transform of $G$ such that
\begin{equ} \label{e:FT}
G(x) = \int_{\R} \widehat{G}(\theta) e^{i \theta x} d \theta. 
\end{equ}
In this article, $G$ will always be the nonlinearity $F$, its derivatives, or their mollifications. 

\subsection{Organisation of the article}

The rest of the article is organised as follows. In Section~\ref{sec:strategy}, we briefly explain the difficulties and give an outline of the strategy. In Section~\ref{sec:framework}, we construct the regularity structures for \eqref{e:macro_eq_appro}, and solve the corresponding abstract equation in a suitable modeled distribution space associated to the regularity structures. Section~\ref{sec:preliminary} contains a few preliminary lemmas and bounds. In Section~\ref{sec:convergence}, we prove the convergence of the rescaled processes to the KPZ equation. The main new ingredient of the proof is a general pointwise bound for correlation functions, which we develop in Section~\ref{sec:general_bound}. This bound enables us to reduce the problem of a general nonlinearity $F$ to that of a polynomial, which has already been treated in \cite{HQ}.

\subsection*{Acknowledgements}

{\small MH gratefully acknowledges financial support from the Leverhulme trust through a leadership award and the European Research Council through a consolidator grant, project 615897. WX has been supported by the Engineering and Physical Sciences Research Council through the fellowship EP/N021568/1. The majority of the work was conducted when both authors were at the University of Warwick, which provided an ideal environment for mathematical research.}

\section{Strategy}
\label{sec:strategy}

In order to prove Theorem~\ref{th:main_intro}, we use the theory of regularity structures developed in \cite{Hai14a} and its adaption to the case of polynomial $F$ in \cite{HQ}. Let
\begin{equ} \label{e:linear_sol}
Z_{\eps} = P * \xi_{\eps}, \qquad \Psi_{\eps} = P' * \xi_{\eps}, 
\end{equ}
where $P$ is the heat kernel on the torus $\T$, and $*$ denotes space-time convolution. If $h_{\eps}$ satisfies \eqref{e:macro_eq_appro}, then the remainder $u_{\eps} = h_{\eps} - Z_{\eps}$ solves the equation
\begin{equ} \label{e:remainder}
\partial_{t} u_{\eps} = \partial_x^2 u_{\eps} + \eps^{-1} F \big(\eps^{\frac{1}{2}} (\Psi_{\eps} + \partial_{x} u_{\eps}) \big) - C_{\eps}, \qquad u_{0}^{(\eps)} = h_{0}^{(\eps)} - Z_{0}^{(\eps)}, 
\end{equ}
where $Z_{0}^{(\eps)} = \int_{-\infty}^{0} P_{t-s} * \xi_{\eps}(s) ds$. Since $\eps^{\frac{1}{2}} \Psi_{\eps} \sim \nN(0, \sigma^{2})$ and, by analogy with \cite{HairerKPZ}, $\eps^{\frac{1}{2}} \partial_{x} u_{\eps}$ is expected to have size of almost $\eps^{\frac{1}{2}}$, we can therefore expand $F$ near $\eps^{\frac{1}{2}} \Psi_{\eps}$, formally yielding
\begin{equs} \label{e:taylor_F}
	\begin{split}
\eps^{-1} F(\eps^{\frac{1}{2}} \Psi_{\eps} + \eps^{\frac{1}{2}} \partial_{x} u_{\eps}) &= \eps^{-1} F(\eps^{\frac{1}{2}} \Psi_{\eps}) + \eps^{-\frac{1}{2}} F'(\eps^{\frac{1}{2}} \Psi_{\eps}) \cdot (\partial_{x} u_{\eps})\\
&\quad + \frac{1}{2} F''(\eps^{\frac{1}{2}} \Psi_{\eps}) \cdot (\partial_x u_{\eps})^{2} + \oO(\eps^{\frac{1}{2}-}). 
\end{split}
\end{equs}
The natural step next is to characterise as $\eps \rightarrow 0$ the limiting objects $F''(\eps^{\frac{1}{2}} \Psi_{\eps})$, $\eps^{-1} F(\eps^{\frac{1}{2}} \Psi_{\eps})$,  $\eps^{-\frac{1}{2}} F'(\eps^{\frac{1}{2}} \Psi_{\eps})$, as well as the products between them and $\d_{x} u_{\eps}$ and $(\d_{x} u_{\eps})^{2}$. If one expands $F''$, $\eps^{-\frac{1}{2}}F'$ and $\eps^{-1}F$ (all with the argument $\eps^{\frac{1}{2}}\Psi_{\eps}$) into Wiener chaos, then it is easy to see that arbitrary moments of all the higher order chaos vanish term-wise as $\eps \rightarrow 0$, and one should expect that the limiting objects are given by constant multiples of $1$, the free field $\Psi$ (limit of $\Psi_{\eps}$), and its Wick square $\Psi^{\diamond 2}$ (after re-centering) respectively. The behaviour of product of these objects as $\eps \rightarrow 0$ can be seen in a similar way. 

In fact, this is the procedure taken in \cite{HQ} for polynomial $F$. However, the main obstacle of implementing the same procedure for general $F$ is that it gives an \textit{infinite chaos series}. In order for the term-wise $L^{p}$ moments to be summable for every $p$, even for fixed $\eps$, one needs extremely fast decay of the coefficients in the chaos expansion. This translates into the condition that the Fourier transform of $F$ should decay faster than every Gaussian, which is clearly too restrictive for the main statement to be interesting and widely applicable. 

This way of direct chaos expansion is of course far from being optimal when $F$ is not a polynomial. In fact, different homogeneous chaos are highly correlated in $L^p$ for large $p$, and the sum of term-wise $L^p$ norms is simply a bad upper bound for the $L^p$ norm of the whole series -- it does not capture large cancellations between different terms from the chaos expansion. On the other hand, in order to get a bound that is uniform in $\eps$, one still needs to chaos expand the object at some point so that the negative power of $\eps$ in front of $F$ or $F'$ can be balanced out by the positive powers carried by Wick powers in the expansion. 

It is at this point that our approach starts to deviate from that in \cite{HQ}. The main idea is to write the 
nonlinear function of the free fields in terms of its Fourier transform, and to use clustering arguments and 
trigonometric identities to encode cancellations \textit{before} chaos expanding them. In this way, we obtain a 
pointwise bound on correlations of trigonometric functions of $\Psi_{\eps}$ in terms of their polynomial counterparts. 
This bound is uniform in $\eps$ and polynomial in the frequency of the trigonometric function. Hence, as long as $F$ 
is sufficiently regular, it essentially enables us to reduce the problem of a general $F$ to that of a polynomial. A systematic 
procedure to obtain such a pointwise bound in general situations, which applies to all the objects appearing in the 
expansion of \eqref{e:taylor_F}, is developed in Section~\ref{sec:general_bound}. The same technique applies 
straight away to all objects arising in the study of the dynamical $\Phi_3^4$ model and would in principle allow to 
recover the results of \cite{Phi4_general} where techniques from Malliavin calculus are employed to bound
these objects. 
Conversely, it appears that these techniques may also in principle be able to treat the KPZ case with 
Gaussian noise. 

One advantage of our present approach is that it isolates the reliance on the Gaussianity of the noise 
into the bound of Section~\ref{sec:general_bound}, while the rest of the argument is essentially independent of it.

\section{Regularity structure and the abstract equation}
\label{sec:framework}

The aim of this section is to construct the regularity structure that will enable us to solve the equation \eqref{e:remainder}. As long as we can solve for $u_{\eps}$, the process $h_{\eps}$ is just $u_{\eps} + Z_{\eps}$. From now on, we focus on the remainder equation \eqref{e:remainder}. 

\subsection{The regularity structure}

\label{sec:rs}

We start by introducing the collection of symbols in the regularity structure. Let $X^k$ denote the abstract polynomials, where $k = (k_0, k_1)$ is a two dimensional index with $k_i \in \N$. In particular, we use the special symbol $\1$ for the case $k = \0$. We also use the symbols \<0'>, \<1'> and \<2'> to describe recentered and rescaled versions of $F''(\eps^{\frac{1}{2}} \Psi_{\eps})$, $\eps^{-\frac{1}{2}} F'(\eps^{\frac{1}{2}} \Psi_{\eps})$ and $\eps^{-1} F(\eps^{\frac{1}{2}} \Psi_{\eps})$ respectively. At the level of the regularity structure, we impose $\<2'> = \<1'>^{2}$ even though the canonical lift used later on in this article
does \textit{not} satisfy the corresponding identity. The reason, as we shall see later, is that if we choose properly the constant multiples and renormalisations in the model, one converges to the ``Wick square'' of the other as $\eps \rightarrow 0$.

Let $\iI$ and $\iI'$ denote the abstract 
integration maps with respect to the heat kernel and its spatial derivative. We then generate and add new symbols 
(basis vectors) to the regularity structure by applying $\iI$ and $\iI'$ and pairwise multiplication, as dictated 
by the structure of the equation \eqref{e:taylor_F} where we ignore for the moment the $\oO(\eps^{{1\over 2}-})$ error term. 
For convenience of notations, we use graphical notations  analogous to those
used in \cite{HairerKPZ,Etienne} to denote the newly generated symbols by setting
\begin{equ}
\<1'0> = \iI'(\<1'>)\;, \quad \<2'2'0> = \iI'(\<2'>)^2\;, \quad  \<2'1'1'> = \<1'>\cdot \iI'\bigl(\<1'> \cdot\iI'(\<2'>)\bigr) \;,\quad \text{etc.} 
\end{equ}
Note that solid lines denote $\iI'$ and not $\iI$ since $\iI$ does not play much of a role in our analysis. 
We now associate to every symbol $\tau$ a homogeneity $|\tau| \in \R$. For the Taylor polynomial $X^k$, we 
let $|X^k| = |k| = 2k_0 + k_1$ if $k = (k_0, k_1)$. Let $\kappa>0$ be small. We set
\begin{equ}
|\<0'>| = -\kappa, \qquad |\<1'>| = -\frac{1}{2} - \kappa, 
\end{equ}
and define recursively
\begin{equ}
|\tau \bar{\tau}| = |\tau| + |\bar{\tau}|, \quad |\iI(\tau)| = |\tau|+2, \quad |\iI'(\tau)| = |\tau| + 1. 
\end{equ}
Note that since we decreed that $\<2'> = \<1'>^{2}$, we have in particular $|\<2'>| = -1-2\kappa$. The following
is a list of all the symbols of negative homogeneity appearing in the regularity structure obtained in this way:
\begin{equ}\label{e:symbols}
	\begin{tabular}{lll}\toprule
		\<2'>\;, \<2'1'>\;, \<1'>\;, \<2'2'0'>\;, \<2'2'0>\;, \<2'1'1'>\;, \<2'0'>\;, \<1'1'>\;, \<2'0>\;, \<0'> \\
%		\;, \<2'1'0>\;, \<1'0>\\
		\bottomrule
	\end{tabular}
\end{equ}
The corresponding structure group can be defined in the same way as in \cite{Hai14a,HQ,rs_algebraic}, so we omit the details here.

\subsection{The models}

Recall the definition of $a$ in \eqref{e:coupling_const}, and that $\Psi_{\eps} = P' * \xi_{\eps}$. For every $\eps>0$, we 
define a (random) representation $\PPi^\eps$ of the regularity structure by
\begin{equs} \label{e:model_new}
\begin{split}
(\PPi^{\eps} \<0'>)(z) &= {1\over 2a} F''\big(\eps^{\frac{1}{2}} \Psi_\eps(z)\big) - 1, \quad
(\PPi^{\eps} \<1'>)(z) = {1\over 2a\sqrt{\eps}} F'\big(\eps^{\frac{1}{2}} \Psi_\eps(z)\big), \\
(\PPi^{\eps} \<2'>)(z) &= {1\over a \eps} F\big(\eps^{\frac{1}{2}} \Psi_\eps(z)\big) - C_{\<2's>}^{(\eps)},  
\end{split}
\end{equs}
where the constant $C_{\<2's>}^{(\eps)}$ is chosen in such a way that 
$\E(\PPi^{\eps} \<2'>)(z) = 0$.
This is extended canonically to the whole regularity structure by postulating that $\iI$ and $\iI'$
correspond to convolution by $K$ and $\d_x K$ respectively, where $K$ is a suitable truncation of the 
heat kernel $P$ and equals $P$ in a domain containing the origin, and
by setting $(\PPi^{\eps} \tau \bar{\tau})(z) = (\PPi^{\eps}\tau)(z) \cdot (\PPi^{\eps}\bar{\tau})(z)$. 

The reason why we normalise and subtract constants in the way specified in \eqref{e:model_new} is that, 
with this choice of normalisation, we will be able to show that the action of $\PPi^{\eps}$ 
on these three symbols converges to $0$, $\Psi$, and $\Psi^{\diamond 2}$ respectively, with these
limits being independent of $F$ and of the covariance of the noise. This also justifies the relation $\<2'>=\<1'>^2$ imposed in our regularity structure. 

The canonical (random) model $\lL_{\eps} = (\Pi^{\eps},\Gamma^\eps) = \zZ(\PPi^\eps)$ is then defined
as in \cite{rs_algebraic}\footnote{For the purpose of
applying the results of \cite{rs_algebraic}, we consider \<1's> and \<2's> as unrelated `noise types'. The relation
$\<2's>=\<1's>^2$ is really only useful for comparing the limiting model to existing results on the KPZ
equation and plays no role here. In particular, it does not restrict the space of admissible models for our
regularity structure.}, see \cite[Def.~6.7]{rs_algebraic} and the discussion preceding \cite[Def.~6.23]{rs_algebraic}. The fact that the canonical model is indeed a model (i.e.\ it satisfies the corresponding
analytical bounds) was shown in \cite[Prop.~6.11]{rs_algebraic} (see also \cite[Prop.~8.27]{Hai14a} 
for an essentially equivalent statement).

Given the canonical model $\lL_{\eps}$, we define a renormalised model $\widehat{\lL}_{\eps} = (\hPi^{\eps},\widehat\Gamma^\eps)$ by stipulating that $\widehat{\lL}_{\eps}$ is the model obtained from
$\lL_\eps$ by BPHZ renormalisation, as defined in \cite[Thm~6.17]{rs_algebraic}.
In our case, this renormalisation procedure is easy to describe explicitly:
we set $\hPPi^{\eps} \tau = \PPi^{\eps} \tau$ for $\tau \in \{\<0'>,\<1'>,\<2'>\}$ and keep the canonical actions of $\iI$, $\iI'$ and products intact, except that we set
\begin{equs} \label{e:model_2}
\begin{split}
(\hPPi^{\eps} \<2'2'0>)(z) &= (\hPPi^{\eps} \<2'0>)^{2}(z) - C_{\<2'2'0s>}^{(\eps)}\;, \\
(\hPPi^{\eps} \<2'1'1'>)(z) &= (\hPPi^{\eps} \<2'1'0>)(z) \cdot (\hPPi^{\eps} \<1'>)(z) - C_{\<2'1'1's>}^{(\eps)}\;, \\
(\hPPi^{\eps} \<2'2'0'>)(z) &= (\hPPi^{\eps} \<0'>)(z) \cdot (\hPPi^{\eps} \<2'2'0>)(z) - C_{\<2'2'0's>}^{(\eps)}\;,
\end{split}
\end{equs}
where the values of the constants $C_{\tau}^{(\eps)}$ are chosen in such a way that $\E (\hPPi^{\eps} \tau)(0) = 0$
for all symbols $\tau$ appearing in \eqref{e:model_2}. We then set 
$\widehat{\lL}_{\eps} = \zZ(\hPPi^\eps)$, which is again a model by 
\cite[Thm~6.28]{rs_algebraic}.
Using the Gaussianity of $\Psi_\eps$, the fact that $F$ is even by assumption, and our assumption
that the covariance of $\xi_\eps$ is spatially symmetric,
one can verify that this then implies that $\E (\hPPi^{\eps} \tau)(0) = 0$ for all $\tau$ with $|\tau| < 0$,
which implies that this is indeed the BPHZ model by the uniqueness statement in \cite[Thm~6.17]{rs_algebraic}.

It was shown in \cite{rs_analytic} that, for a rather large class of noises and stochastic PDEs, the 
corresponding BPHZ renormalised model is well-defined and stable under perturbations.
Unfortunately, in order to apply this result to our situation, we would need sharp cumulant bounds 
of all orders on the three stochastic processes appearing in \eqref{e:model_2}. In our particular
example, this does not appear to be any easier than showing the full convergence of the models, so
we will avoid using these results.

\subsection{Abstract equation}

The aim of this subsection is to formulate and solve an abstract fixed point problem in a suitable modeled distribution space so that it can be reconstructed back to \eqref{e:remainder}. The form of the equation \eqref{e:remainder}, the Taylor expansion of $F$ in \eqref{e:taylor_F} and the definition of the model in \eqref{e:model_new} and \eqref{e:model_2} suggest that we may want to consider the fixed point equation
\begin{equ}[e:basicFP]
U = \pP \one_+ \big( a (\<1'> + \sD U)^{2} + a\,\<0'>\, (\sD U)^{2} + \eps^{-1} G(\eps^{\frac{1}{2}} \Psi_\eps, \eps^{\frac{1}{2}} \rR_\eps \sD U) \cdot \1 \big) + \hP u_{0}\;,
\end{equ}
(here we do use the interpretation $\<2'>=\<1'>^2$ when expanding the square) where
\begin{equ}
G(x,y) = F(x+y) - F(x) - F'(x)y - \frac{1}{2} F''(x) y^{2}\;, 
\end{equ}
and $\sD$ is the abstract differentiation operator in the regularity structure. Using \eqref{e:model_new} and the 
fact that we are considering the canonical model, one verifies that solving \eqref{e:basicFP} as a fixed 
point problem in some $\dD^\gamma$ space based on 
the canonical model $\lL_\eps$ yields a solution $U$ which is such that $u_\eps = \rR_\eps U$
solves \eqref{e:remainder} with $C_{\eps} = aC_{\<2's>}^{(\eps)}$. (Here, $\rR_\eps$ is the reconstruction
operator associated to $\lL_\eps$.)
On the other hand, it turns out that the solution associated to the model $\widehat{\lL}_{\eps}$ defined 
using \eqref{e:model_2} (in which case we replace of course the reconstruction operator $\rR_\eps$
appearing in \eqref{e:basicFP} by the reconstruction operator $\hat \rR_\eps$ associated to $\widehat{\lL}_{\eps}$) yields solutions $u_\eps$  to \eqref{e:remainder} with
\begin{equ}[e:constantEps]
C_{\eps} = a C_{\<2's>}^{(\eps)} + a^{3} \big( C_{\<2'2'0s>}^{(\eps)} + 4 C_{\<2'1'1's>}^{(\eps)}  + C_{\<2'2'0's>}^{(\eps)}\big)\;. 
\end{equ}
This will be justified in Theorem~\ref{th:renorm_eq} below. 

In order to solve the abstract fixed point problem \eqref{e:basicFP}, we show that the right hand side yields 
a contraction in a suitable space for small enough time $T$. This will give the existence of solution for short 
time, and then we can continue it to maximal time. Since $\Psi_{\eps}$ is the stationary free field, this 
requires us in particular to be able to treat initial data for the equation for $u$ that have regularity just below $\cC^{\frac{1}{2}}$.

There is a technical issue in carrying out this procedure. Since we cannot expect the solution at very short times to 
behave better than the solution to the heat equation, we can at best hope for a bound of order $t^{-\frac{1-\eta}{2}}$ on 
$\rR \sD U$ as $t \to 0$ if we start with a generic initial condition in $\cC^{\eta}$. In particular, for any positive $\eps$, we would not expect the term $G(\eps^{\frac{1}{2}} \psi, \eps^{\frac{1}{2}} \rR \sD U)$ to be integrable if $F(x)$ grows faster than $|x|^{\frac{2}{1-\eta}}$ as $|x| \rightarrow \infty$. This would require $F$ to have less-than-quartic growth in order to start from initial data below $\cC^{\frac{1}{2}}$. Fortunately, what saves us is that for any fixed $\eps$, the solution is actually smooth, as long as we consider scales smaller than $\eps$. 

In order to quantify this, we proceed as in \cite{HQ} and introduce $\eps$-dependent spaces of functions and models. For $\eta \in (0,1)$ and $\gamma \in (1,2)$, we let $\cC^{\gamma,\eta}_{\eps}$ be the space of functions that are $\cC^{\eta}$ at large scales (larger than $\eps$) and $\cC^{\gamma}$ at smaller scales. More precisely, we define the norm $\|\cdot\|_{\gamma,\eta;\eps}$ by
\begin{equ}
\|u\|_{\gamma,\eta;\eps} := \|u\|_{\cC^{\eta}} + \frac{\|u'\|_{\infty}}{\eps^{\alpha-1}} + \|u'\|_{\infty} + \sup_{\stackrel{x \neq y}{|x-y| \leq \eps}} \frac{|u'(x)-u'(y)|}{\eps^{\eta-\gamma} |x-y|^{\gamma-1}}. 
\end{equ}
Note that $\cC^{\gamma,\eta}_{0}$ is the same as $\cC^{\eta}$. We can also compare two functions $u^{(\eps)} \in \cC^{\gamma,\eta}_{\eps}$ and $u \in \cC^{\eta}$ by
\begin{equ} \label{e:s_space_funct}
\|u^{(\eps)};u\|_{\gamma,\eta;\eps} = \|u^{(\eps)} - u\|_{\cC^{\eta}} + \frac{\|(u^{(\eps)})'\|_{\infty}}{\eps^{\alpha-1}} + \sup_{\stackrel{x \neq y}{|x-y| \leq \eps}} \frac{|(u^{(\eps)})'(x)-(u^{(\eps)})'(y)|}{\eps^{\eta-\gamma} |x-y|^{\gamma-1}}. 
\end{equ}
For a continuous function $\varphi: \R^+ \times \T \rightarrow \R$, $z \in \R^+ \times \T$ and $\lambda>0$, we let $\varphi_z^{\lambda}$ be the function
\begin{equ}
\varphi_{z}^{\lambda}(z') = \lambda^{-3} \varphi\big((z'-z)/\lambda\big). 
\end{equ}
We let $\bB$ denote the set of smooth functions which are compactly supported in the ball of radius one, and whose derivatives up to second order (including the function itself) are uniformly bounded by $1$. We let $\bB_0$ denote the class of functions $\phi \in \bB$ such that $\int \phi(z) dz = 0$. Let $\sM_{\eps}$ be the space of admissible models\footnote{Admissible models are the ones that act canonically on abstract Taylor polynomials and for which the abstract integration maps do represent convolution by $K$, see \cite{Hai14a}. All models considered in this article are indeed admissible. } $(\Pi, \Gamma)$ that further satisfy the bound
\begin{equ}
|(\Pi_z \tau)(\phi_{z}^{\lambda})| \lesssim \lambda^{\bar{\gamma}} \eps^{|\tau|-\bar{\gamma}}, \qquad \tau \in \uU', \quad \bar{\gamma} = 1 - \kappa
\end{equ}
for all test functions $\phi \in \bB_0$. Let
\begin{equ}
\|\Pi\|_{\eps} := \sup_{z} \sup_{\stackrel{\tau \in \uU'}{|\tau|<\bar{\gamma}}} \sup_{\lambda \leq \eps} \sup_{\phi \in \bB_0} \lambda^{-\bar{\gamma}} \eps^{\bar{\gamma}-|\tau|} |(\Pi_{z} \tau)(\phi_{z}^{\lambda})|. 
\end{equ}
We then define the family of ``norms'' on $\sM_{\eps}$ by
\begin{equ}
|\!|\!| \Pi  |\!|\!|_{\eps} := |\!|\!| \Pi  |\!|\!| + \|\Pi\|_{\eps}, 
\end{equ}
where $|\!|\!| \cdot |\!|\!|$ is the usual norm on space of models as in \cite{Hai14a}. Note that $\sM_{\eps}$ consists of the same collection of models for all $\eps > 0$, but their ``norms'' behave very differently as $\eps \rightarrow 0$. We compare a model $\Pi^{\eps} \in \sM_{\eps}$ and a model $\Pi \in \sM$ by
\begin{equ}
|\!|\!| \Pi^{\eps}; \Pi  |\!|\!|_{\eps,0} = |\!|\!| \Pi^{\eps}; \Pi  |\!|\!| + \|\Pi^{\eps}\|_{\eps}. 
\end{equ}
We have only $\Pi^{\eps}$ under the norm $\|\cdot\|_{\eps}$ above since $\|\Pi\|_{\eps}$ may be infinity for positive $\eps$. 

We also introduce $\eps$-dependent spaces of modelled distributions as given in \cite[Def.~2.17]{HQ}. Given a model $(\Pi, \Gamma) \in \sM_{\eps}$, define the $\dD^{\gamma,\eta}_{\eps}$ space to be the modelled distributions $U$ with the norm
\begin{equ}
\|U\|_{\gamma,\eta;\eps} := \|U\|_{\gamma,\eta} + \sup_{z} \sup_{\alpha > \gamma} \frac{|U(z)|_{\alpha}}{\eps^{\eta-\alpha}} + \sup_{\stackrel{|z-z'| \leq \sqrt{|t| \wedge |t'|}}{|z-z'|\leq \eps}} \frac{|U(z) - \Gamma_{z,z'}U(z')|}{|z-z'|^{\gamma-\alpha} \eps^{\eta-\gamma}}. 
\end{equ}
Similarly, we compare two functions $U^{(\eps)} \in \dD^{\gamma,\eta}_{\eps}$ and $U \in \dD^{\gamma,\eta}$ by
\begin{equ}
\|U^{(\eps)}; U\|_{\gamma,\eta;\eps} = \|U^{(\eps)}; U\|_{\gamma,\eta} + \sup_{z} \sup_{\alpha > \gamma} \frac{|U^{(\eps)}(z)|_{\alpha}}{\eps^{\eta-\alpha}} + \sup_{\stackrel{|z-z'| \leq \sqrt{|t| \wedge |t'|}}{|z-z'|\leq \eps}} \frac{|U^{(\eps)}(z) - \Gamma_{z,z'}U^{(\eps)}(z')|}{|z-z'|^{\gamma-\alpha} \eps^{\eta-\gamma}}. 
\end{equ}
The reason that $U$ does not appear in the latter two terms on the right hand side above is the same as before -- these two supremum may be infinity for general $U \in \dD^{\gamma,\eta}_{0}$. Also here $\eta$ is allowed to be any real number less than $\gamma$ (not necessarily positive). 

\begin{rmk}
	The readers may have noticed that we have the abuse of notation $\|\cdot\|_{\gamma,\eta;\eps}$ to denote both $\cC^{\gamma,\eta}_{\eps}$ and $\dD^{\gamma,\eta}_{\eps}$ norms. But since the precise function space we are referring to should be clear in relevant contexts, we keep this same notation for both for simplicity. 
\end{rmk}

\begin{prop} \label{pr:lift}
	Let $\gamma \in (1,2)$ and $\eta \in (0,1)$. Let $u \in \cC^{\gamma,\eta}_{\eps}$, and $\hP u$ be the harmonic extension of $u$. Then, $\hP u \in \dD^{\gamma,\eta}_{\eps}$ and
$	\|\hP u\|_{\gamma,\eta;\eps} \lesssim \|u\|_{\gamma,\eta;\eps}$. 
	Furthermore, if $\bar{u} \in \cC^{\eta}$, then one has
$	\|\hP u; \hP \bar{u}\|_{\gamma,\eta;\eps} \lesssim \|u; \bar{u}\|_{\gamma,\eta;\eps}$. 
\end{prop}
\begin{proof}
	Same as \cite[Prop.~4.7]{HQ}. 
\end{proof}

\begin{prop} \label{pr:restart}
	Let $\eta \leq 1 - \kappa$, $\gamma = \frac{3}{2} + 2M \kappa$, and $T > 0$. Let $U_{\eps} \in \dD^{\gamma,\eta}$ be based on some model $\Pi^{\eps} \in \sM_{\eps}$. Then, for every $t>0$ such that $[t-\eps^{2}, t+\eps^{2}] \subset [0,T]$, the function $u_{t}^{(\eps)} = (\rR^{\eps} U_{\eps})(t, \cdot)$ belongs to $\cC^{\gamma,\eta}_{\eps}$ with the bound
	\begin{equ}
	\|u_{t}^{(\eps)}\|_{\gamma,\eta;\eps} \lesssim \|U_{\eps}\|_{\gamma,\eta} |\!|\!| \Pi^{\eps} |\!|\!|_{\eps}, 
	\end{equ}
	where the proportionality constant is independent of $\eps$. Furthermore, if $U \in \dD^{\gamma,\eta}$ based on some model $\Pi \in \sM$, then $u_{t} = (\rR U)(t, \cdot) \in \cC^{\eta}$ and one has the bound
	\begin{equ}
	\| u_{t}^{(\eps)}; u_{t} \|_{\gamma,\eta;\eps} \lesssim \|U_{\eps}; U\|_{\gamma,\eta} \big( |\!|\!| \Pi^{\eps} |\!|\!|_{\eps} + |\!|\!| \Pi  |\!|\!| \big) + |\!|\!| \Pi^{\eps}; \Pi  |\!|\!|_{\eps,0} \big( \|U_{\eps}\|_{\gamma,\eta} + \|U\|_{\gamma,\eta} \big). 
	\end{equ}
\end{prop}
\begin{proof}
	Same as \cite[Prop.~4.8]{HQ}. 
\end{proof}

\begin{prop} \label{pr:differentiation}
	Let $U \in \dD^{\gamma,\eta}_{\eps}$ for some $\gamma>1$ and $\eta \in \R$. Then, $\sD U \in \dD^{\gamma-1,\eta-1}_{\eps}$ with the bound
	\begin{equ}
	\|\sD U\|_{\gamma-1,\eta-1;\eps} \leq C \|U\|_{\gamma,\eta;\eps}. 
	\end{equ}
	Furthermore, if $\bar{U} \in \dD^{\gamma,\eta}_{\eps}$, then one has
	\begin{equ}
	\|\sD U;\sD \bar{U}\|_{\gamma-1,\eta-1;\eps} \leq \|U;\bar{U}\|_{\gamma,\eta;\eps}. 
	\end{equ}
\end{prop}
\begin{proof}
	Same as \cite[Prop.~4.9]{HQ}. 
\end{proof}

\begin{prop} \label{pr:multiplication}
	For $i=1,2$, let $U_{i} \in \dD^{\gamma_i, \eta_i}_{\eps}(V^{(i)})$, where $V^{(i)}$ is a sector with regularity $\alpha_i$. Then, $U = U_1 U_2 \in \dD^{\gamma,\eta}_{\eps}$ with
	\begin{equ}
	\gamma = (\gamma_1 + \alpha_2) \wedge (\gamma_2 + \alpha_1), \qquad \eta = (\eta_1 + \alpha_2) \wedge (\eta_2 + \alpha_1) \wedge (\eta_1 + \eta_2), 
	\end{equ}
	and one has the bound
	\begin{equ}
	\|U\|_{\gamma,\eta;\eps} < C \|U_1\|_{\gamma_1,\eta_1;\eps} \|U_2\|_{\gamma_2,\eta_2;\eps} \big(1 + |\!|\!| \Pi |\!|\!|_{\eps}\big)^{2}. 
	\end{equ}
	Further more, if $\bar{U}_i \in \dD^{\gamma_i,\eta_i}_{\eps}$ with the same parameters as above, and $\bar{U} = \bar{U}_1 \bar{U}_2$, then one has the bound
	\begin{equ}
	\|U;\bar{U}\|_{\gamma,\eta;\eps} < C \big( \|U_1; \bar{U}_1\|_{\gamma_1,\eta_1;\eps} + \|U_2;\bar{U}_2\|_{\gamma_2,\eta_2;\eps} \big) \big( 1 + |\!|\!| \Pi; \bar{\Pi} |\!|\!|_{\eps} \big). 
	\end{equ}
\end{prop}
\begin{proof}
	Same as \cite[Prop.~4.10]{HQ}. 
\end{proof}

\begin{prop} \label{pr:integration}
	Let $V$ be a sector of regularity $\alpha$, and let $U \in \dD^{\gamma,\eta}_{\eps}(V)$ with $-2 < \eta < \gamma \wedge \alpha$. Let $\one_+$ be the restriction of time variables to be positive, and $T$ denote the length of the interval where the abstract integration takes place. Then, provided that $\gamma$ and $\eta$ are not integers, there exists $\theta > 0$ such that
	\begin{equ}
	\|\pP \one_+ U\|_{\gamma+2,\eta+2;\eps} \leq C (T+\eps)^{\theta} \big( \|U\|_{\gamma,\eta;\eps} + |\!|\!| \Pi |\!|\!|_{\eps} \big). 
	\end{equ}
	Furthermore, if $\bar{U} \in \dD^{\gamma,\eta}_{\eps}$, then one has
	\begin{equ}
	\| \pP U; \pP \bar{U} \|_{\gamma+2,\eta+2;\eps} \leq C (T + \eps)^{\theta} \big( \|U;\bar{U}\|_{\gamma,\eta;\eps} + |\!|\!| \Pi; \bar{\Pi} |\!|\!|_{\eps} \big). 
	\end{equ}
\end{prop}
\begin{proof}
	Same as \cite[Thm~7.1, Lem.~7.3]{Hai14a} and \cite[Prop.~4.13]{HQ}. 
\end{proof}

We are now ready to prove the main theorem on the existence of the solution to the abstract equation. We restrict our regularity structure to the space spanned by polynomials and the basis elements listed in \eqref{e:symbols}. 

\begin{thm} \label{th:abstract_fpt}
	Let $\gamma \in (\frac{3}{2}, \frac{5}{3})$ and $\eta \in (\frac{1}{2} - \frac{1}{M+4},\frac{1}{2})$, where $M$ is the same as in Assumption~\ref{as:F}. Let $u_{0}^{(\eps)} \in \cC^{\gamma,\eta}_{\eps}$, and let $\psi_{\eps}: \R^{+} \times\ \T \rightarrow \R$ be a family of smooth functions such that
	\begin{equ}
	\sup_{\eps \in (0,1)} \sup_{z \in [0,\bar{T}] \times \T} \eps^{\frac{1}{2} + \kappa} |\psi_{\eps}(z)| < +\infty
	\end{equ}
	for every $\bar{T}>0$, and for some sufficiently small $\kappa > 0$. Consider the fixed point problem
	\begin{equ} \label{e:fixed_pt}
	U = \pP \one_+ \Big( a (\<1'> + \sD U)^{2} + \<0'> \cdot (\sD U)^{2} + \eps^{-1} G \big( \eps^{\frac{1}{2}} \psi_{\eps}, \eps^{\frac{1}{2}} \rR \sD U \big) \cdot \1 \Big) + \hP u_{0}^{(\eps)}, 
	\end{equ}
	where $\hP$ is the harmonic extension operator, and 
	\begin{equ}
	G(x,y) = F(x+y) - F(x) - F'(x) y - \frac{1}{2} F''(x) y^{2}. 
	\end{equ}
	Then, there exists $\eps_{0} > 0$ such that for every $\eps \in [0, \eps_{0}]$ and every $\Pi^{\eps} \in \sM_{\eps}$, \eqref{e:fixed_pt} has a unique solution $U^{(\eps)} \in \dD^{\gamma,\eta}_{\eps}$ to up to time $T$. Furthermore, the terminal time $T$ can be chosen uniformly jointly over bounded sets of initial conditions in $\cC^{\gamma,\eta}_{\eps}$ and bounded sets in $\sM_{\eps}$. 
	
	Now, suppose $\{u_{0}^{(\eps)}\}$ is a sequence in $\cC^{\gamma,\eta}_{\eps}$ such that $\|u_{0}^{(\eps)}; u_{0}\|_{\gamma,\eta;\eps} \rightarrow 0$ for some function $u_{0} \in \cC^{\eta}$, and $\{\Pi^{\eps}\}$ is a sequence of models in $\sM_{\eps}$ such that $|\!|\!| \Pi^{\eps}; \Pi |\!|\!|_{\eps,0} \rightarrow 0$ for some model $\Pi \in \sM$. Let $T>0$ be fixed and suppose $U \in \dD^{\gamma,\eta}$ solves the fixed point problem
	\begin{equ} \label{e:fixed_pt_0}
	U = \pP \one_+ \Big( a (\<1'> + \sD U)^{2} + \<0'> \cdot (\sD U)^{2} \Big) + \hP u_{0}, 
	\end{equ}
	with the model $\Pi$ and initial data $u_{0}$ up to time $T$. Then, for every $\eps > 0$ sufficiently small, there exists a unique solution $U^{(\eps)} \in \dD^{\gamma,\eta}_{\eps}$ to \eqref{e:fixed_pt} with the model $\Pi^{\eps}$ and initial condition $u_{0}^{(\eps)}$ up to the same terminal time $T$. Furthermore, we have
	\begin{equ}
	\lim_{\eps \rightarrow 0} \| U^{(\eps)}; U \|_{\gamma,\eta;\eps} = 0. 
	\end{equ}
\end{thm}
\begin{proof}
	We first show that the fixed point problem \eqref{e:fixed_pt} can be solved in $\dD^{\gamma,\eta}_{\eps}$ with local existence time uniform in $\eps$. Consider the map
	\begin{equ} \label{e:fpt_epsilon}
	\mM^{(\eps)}_{T} (U) = \pP \one_+ \Big( a (\<1'> + \sD U)^{2} + \<0'> \cdot (\sD U)^{2} + \eps^{-1} G(\eps^{\frac{1}{2}}\psi_{\eps}, \eps^{\frac{1}{2}} \rR^{\eps} \sD U) \cdot \1 \Big) + \hP u_{0}^{(\eps)}, 
	\end{equ}
	where $\rR^{\eps}$ is the reconstruction operator associated to $\Pi^{\eps} \in \sM_{\eps}$, and $T$ is the length of the time interval on which the modelled distribution $U$ is defined. We want to show that for sufficiently small $T$ and $\eps$, $\mM_{T}^{(\eps)}$ is a contraction map from $\dD^{\gamma,\eta}_{\eps}$ to itself, and this $T$ can be chosen uniformly over all small $\eps$. For convenience of notations, we omit the subscript $T$ below and simply denote the map by $\mM^{(\eps)}$. 
	
	For the term with the initial condition, by Proposition~\ref{pr:lift}, we have
	\begin{equ} \label{e:initial}
	\| \hP u_{0}^{(\eps)} \|_{\gamma,\eta;\eps} < C \|u_{0}^{(\eps)}\|_{\gamma,\eta;\eps}. 
	\end{equ}
	Note that the left hand side above is the norm for $\dD^{\gamma,\eta}_{\eps}$, while the right hand side is for $\cC^{\gamma,\eta}_{\eps}$. We now treat the terms $(\<1'> + \sD U)^{2}$ and $\<0'> \cdot (\sD U)^{2}$. 
	
	Since $\sD U \in \dD^{\gamma-1,\eta-1}_{\eps}$ with a sector of regularity $-2\kappa$, and the two ``constant terms'' $\<1'>$ and $\<0'>$ are both in $\dD^{\infty,\infty}_{\eps}$ with sector regularities $-\frac{1}{2}-\kappa$ and $-\kappa$ respectively, it follows from Proposition~\ref{pr:multiplication} that
	\begin{equs}
	\|(\<1'> + \sD U)^{2}\|_{\gamma_1,\eta_1;\eps} \lesssim (1 + |\!|\!| \Pi^{\eps} |\!|\!|_{\eps})^{2} (1 + \|U\|_{\gamma,\eta;\eps})^{2},\;  &\gamma_1 = \gamma-\frac{3}{2}-\kappa, \eta_1 = 2\eta-2, \\
	\|\<0'> \cdot (\sD U)^{2}\|_{\gamma_2,\eta_2;\eps} \lesssim (1 + |\!|\!| \Pi^{\eps} |\!|\!|_{\eps})^{3} (1 + \|U\|_{\gamma,\eta;\eps})^{2},\; &\gamma_2=\gamma-1-3\kappa, \eta_2=2 \eta-2-\kappa. 
	\end{equs}
	Thus, the modelled distribution $a(\<1'>+ \sD U)^{2} + \<0'> \cdot (\sD U)^{2}$ belongs to $\dD^{\bar{\gamma},\bar{\eta}}_{\eps}$ with a sector of regularity $\bar{\alpha}$ with
	\begin{equ}
	\bar{\gamma} = \gamma-\frac{3}{2}-\kappa, \qquad \bar{\eta} = 2 \eta-2-\kappa, \qquad \bar{\alpha} = -1-2\kappa. 
	\end{equ}
	One can check that these parameters satisfy
	\begin{equ}
	-2 < \bar{\eta} < \bar{\gamma} \wedge \bar{\alpha}, \quad \bar{\gamma}+2>\gamma, \quad \bar{\eta}+2>\eta, 
	\end{equ}
	so one can apply Proposition~\ref{pr:integration} and use the natural inclusion of the $\dD^{\gamma,\eta}_{\eps}$ spaces to conclude that there exists $\theta_1 > 0$ such that
	\begin{equ} \label{e:nonlinear}
	\|\pP \one_+ \big( a (\<1'> + \sD U)^{2} + \<0'> \cdot (\sD U)^{2} \big)\|_{\gamma,\eta;\eps} \lesssim (T+\eps)^{\theta_1} (1 + |\!|\!| \Pi^{\eps} |\!|\!|_{\eps})^{3} (1 + \|U\|_{\gamma,\eta;\eps})^{2}. 
	\end{equ}
    As for the ``Taylor remainder'' term $G$, we have
    \begin{equ}
    G(x,y) = \frac{1}{6} F^{(3)}(x+h) y^{3}
    \end{equ}
    for some $h$ between $0$ and $y$. Thus, the growth condition $|F^{(3)}(x+h)| \lesssim (1+ |x+h|)^{M}$ and the fact $|h| \leq |y|$ imply that
    \begin{equs}
    &\phantom{111}\|\eps^{-1} G(\eps^{\frac{1}{2}} \psi_{\eps}(t,\cdot), \eps^{\frac{1}{2}} \rR^{\eps} \sD U(t,\cdot)) \|_{L^{\infty}(\T)}\\
    &\lesssim \eps^{\frac{1}{2}} \|(\rR^{\eps} \sD U)(t,\cdot)\|_{L^{\infty}(\T)}^{3} \Big(1+ \|\eps^{\frac{1}{2}} \psi_{\eps}\|_{\infty} + \|\eps^{\frac{1}{2}} (\rR^{\eps} \sD U)(t,\cdot) \|_{L^{\infty}(\T)} \Big)^{M}. 
    \end{equs}
    Since $\sD U \in \dD^{\gamma-1,\eta-1}_{\eps}$, it follows from the assumption $\Pi^{\eps} \in \sM_{\eps}$ and the reconstruction theorem (\cite[Thm~3.10, Prop.~6.9]{Hai14a}) that
    \begin{equ}
    \|(\rR^{\eps} \sD U)(t,\cdot)\|_{L^{\infty}(\T)} \lesssim |\!|\!| \Pi^{\eps} |\!|\!|_{\eps} \|U\|_{\gamma,\eta;\eps} \cdot (\sqrt{t}+\eps)^{\eta-1}. 
    \end{equ}
    Together with the assumption on $\psi_{\eps}$, we see that for every $\delta > 0$, we have
    \begin{equs}
    \begin{split}
    &\phantom{111}\|\eps^{-1} G(\eps^{\frac{1}{2}} \psi_{\eps}(t,\cdot), \eps^{\frac{1}{2}} (\rR^{\eps} \sD U)(t,\cdot))\|_{L^{\infty}(\T)}\\
    &\lesssim \eps^{\frac{1}{2}} (\sqrt{t}+\eps)^{3(\eta-1)} \big(\eps^{-M\kappa}+(\sqrt{t}+\eps)^{M(\eta-\frac{1}{2})}\big) \cdot \big(1 + |\!|\!| \Pi^{\eps} |\!|\!|_{\eps} \|U\|_{\gamma,\eta;\eps} \big)^{M+3}. 
    \end{split}
    \end{equs}
    One can easily check that for $\eta \in (\frac{1}{2}-\frac{1}{M+4}, \frac{1}{2})$ and sufficiently small $\kappa$, we have
    \begin{equ}
    \eps^{\frac{1}{2}} (\sqrt{t}+\eps)^{3(\eta-1)} \big(\eps^{-M\kappa}+(\sqrt{t}+\eps)^{M(\eta-\frac{1}{2})}\big) \lesssim \eps^{\delta} t^{-1+\delta} 
    \end{equ}
    for some $\delta > 0$. In particular, the singularity in $t$ near the origin is integrable. Hence, there exist $C, \theta_2 > 0$ such that
    \begin{equ} \label{e:Taylor}
    \eps^{-1} \|\pP \one_+ G(\eps^{\frac{1}{2}} \psi_{\eps}, \eps^{\frac{1}{2}} \rR^{\eps} \sD U)\|_{\gamma,\eta;\eps} \leq C (\eps T)^{\theta_2} \big( 1 + |\!|\!| \Pi^{\eps} |\!|\!|_{\eps} \|U\|_{\gamma,\eta;\eps} \big)^{M+3}. 
    \end{equ}
    Combining \eqref{e:initial}, \eqref{e:nonlinear} and \eqref{e:Taylor}, we see that for every $R > 0$ sufficiently large, there exists a final time $T>0$ and $\eps_{0} > 0$ such that for all $\eps \in [0,\eps_{0}]$, the map $\mM_{T}^{(\eps)}$ defined in \eqref{e:fixed_pt} maps a ball of radius $R$ in $\dD^{\gamma,\eta}_{\eps}$ into itself provided $\|u_{0}^{(\eps)}\|_{\gamma,\eta;\eps} < \frac{\Lambda}{3 C}$ for the $C$ in \eqref{e:initial} and
    \begin{equ}
    (T+\eps)^{\theta} < \frac{R}{3 C\big( 1 + |\!|\!| \Pi^{\eps} |\!|\!|_{\eps}\big)^{M+3} \big(1+R\big)^{M+3}}. 
    \end{equ}
    We now fix $\Pi^{\eps} \in \sM_{\eps}$, and we will show that the map $\mM_{T}^{(\eps)}$ is locally Lipschitz from $R$-balls of $\dD^{\gamma,\eta}_{\eps}$ into itself for with Lipschitz constant less than $1$ for small $T$. This will imply that it is a contraction on the $R$-ball in $\dD^{\gamma,\eta}_{\eps}$. In fact, using the relation
	\begin{equ}
	(\sD U)^{2} - (\sD \bar{U})^{2} = (\sD U - \sD \bar{U}) (\sD U + \sD \bar{U}), 
	\end{equ}
	we can use the same argument as above to show that the map
	\begin{equ} \label{e:map_main}
	U \mapsto \pP \one_{+} \Big(a(\<1'>+\sD U)^{2} + \<0'> \cdot (\sD U)^{2} \Big)
	\end{equ}
	is locally Lipschitz with the Lipschitz constant bounded by
	\begin{equ}
	C (T+\eps)^{\theta} (1+R)^{2} (1+ |\!|\!| \Pi^{\eps} |\!|\!|_{\eps})^{2}. 
	\end{equ}
	As for the term $\eps^{-1} G(\eps^{\frac{1}{2}} \psi_{\eps}, \eps^{\frac{1}{2}} \rR^{\eps} \sD U)$, by Taylor's theorem, there exist space-time functions $\phi_{\eps}, \bar{\phi}_{\eps}$ with
	\begin{equ}
	|\phi_{\eps} - \eps^{\frac{1}{2}} \psi_{\eps}| \leq \eps^{\frac{1}{2}} |\rR^{\eps} \sD U|, \qquad |\bar{\phi}_{\eps} - \eps^{\frac{1}{2}} \psi_{\eps}| \leq \eps^{\frac{1}{2}} |\rR^{\eps} \sD \bar{U}|. 
	\end{equ}
	such that
	\begin{equs}
	\eps^{-1} G(\eps^{\frac{1}{2}}\psi_{\eps}, \eps^{\frac{1}{2}} \rR^{\eps} \sD U) &= \frac{\sqrt{\eps}}{6} F^{(3)}(\phi_{\eps}) (\rR^{\eps} \sD U)^{3}, \\
	\eps^{-1} G(\eps^{\frac{1}{2}}\psi_{\eps}, \eps^{\frac{1}{2}} \rR^{\eps} \sD \bar{U}) &= \frac{\sqrt{\eps}}{6} F^{(3)}(\bar{\phi}_{\eps}) (\rR^{\eps} \sD \bar{U})^{3}. 
	\end{equs}
	Using the triangle inequality, we see that
	\begin{equs}
	&\phantom{111}\eps^{-1} | G(\eps^{\frac{1}{2}}\psi_{\eps},\rR^{\eps} \sD U) - G(\eps^{\frac{1}{2}}\psi_{\eps},\rR^{\eps} \sD \bar{U}) |\\
	&\lesssim \eps^{\frac{1}{2}} |F^{(3)}(\phi_{\eps})| \cdot |\rR^{\eps} \sD U - \rR^{\eps} \sD \bar{U}| \cdot \Big( |\rR^{\eps} \sD U|^{2} + |\rR^{\eps} \sD \bar{U}|^{2} \Big)\\
	&\phantom{11}+ \eps^{\frac{1}{2}} |F^{(3)}(\phi_{\eps}) - F^{(3)}(\bar{\phi}_{\eps})| \cdot |\rR^{\eps} \sD \bar{U}|^{3}. 
	\end{equs}
	For the second term on the right hand side above, we have
	\begin{equ}
	|F^{(3)}(\phi_{\eps}) - F^{(3)}(\bar{\phi}_{\eps})| \leq |F^{(4)}(\widetilde{\phi}_{\eps})| \cdot |\phi_{\eps} - \bar{\phi}_{\eps}|
	\end{equ}
	for some $\widetilde{\phi}_{\eps}$ with $|\widetilde{\phi_{\eps}} - \eps^{\frac{1}{2}} \psi_{\eps}| \leq |\eps^{\frac{1}{2}} \rR^{\eps} \sD U| + |\eps^{\frac{1}{2}} \rR^{\eps} \sD \bar{U}|$. Thus, using the assumptions on the uniform boundedness of $\eps^{\frac{1}{2}+\kappa} \psi_{\eps}$, the growth of the derivatives of $F$, and that $|\phi_{\eps} - \bar{\phi}_{\eps}| \leq \eps^{\frac{1}{2}} |\rR^{\eps} (\sD U - \sD \bar{U})|$, we can perform a similar argument as before to obtain
	\begin{equs}
	&\phantom{111}\eps^{-1} \|G(\eps^{\frac{1}{2}}\psi_{\eps},\rR^{\eps} \sD U) - G(\eps^{\frac{1}{2}}\psi_{\eps},\rR^{\eps} \sD \bar{U}) \|_{L^{\infty}(\T)}\\ 
	&\lesssim \eps^{\delta} t^{-1+\delta} (1 + |\!|\!| \Pi^{\eps} |\!|\!|_{\eps})^{M+4} (1+R)^{M+4} \|U-\bar{U}\|_{\gamma,\eta;\eps}
	\end{equs}
	for some $\delta > 0$. This implies that the map
	\begin{equ}
	U \mapsto \pP \one_+ \big( \eps^{-1} G(\eps^{\frac{1}{2}} \Psi_{\eps}, \rR^{\eps} \sD U) \cdot \1 \big)
	\end{equ}
	is locally Lipschitz with constant bounded by
	\begin{equ}
	C (\eps T)^{\theta} (1 + |\!|\!| \Pi^{\eps} |\!|\!|_{\eps})^{M+4} (1 + R)^{M+4}
	\end{equ}
	for some $C, \theta > 0$. Combining it with the local Lipschitzness of the map in \eqref{e:map_main} (and the Lipschitz constant there is bounded by $R^{4} (T+\eps)^{\theta}$), we deduce that for sufficiently small $T$, there is a unique solution in $\dD^{\gamma,\eta}_{\eps}$ to \eqref{e:fixed_pt} up to time $T$. Furthermore, by the expressions for the local Lipschitz constants above, we see that there exists $\eps_0 > 0$ such that this local existence time can be chosen uniformly over $\eps \in (0,\eps_0)$, over bounded sets of $\Pi^{\eps} \in \sM_{\eps}$ and bounded sets of initial data $u_{0}^{(\eps)} \in \cC^{\gamma,\eta}_{\eps}$. 
	
	We now turn to the second part of the theorem, namely the convergence of the solutions $U^{(\eps)}$ up to the time $T$ for which the limiting abstract solution $U$ is defined. Here, $U^{(\eps)} \in \dD^{\gamma,\eta}_{\eps}$ and $U \in \dD^{\gamma,\eta}_{0}$ are local solutions to the fixed point problems \eqref{e:fixed_pt} and \eqref{e:fixed_pt_0} based on $\Pi^{\eps} \in \sM_{\eps}$ and $\Pi \in \sM$ respectively. By the same arguments as above, we know there exists $S<T$ such that \eqref{e:fixed_pt} has a fixed point solution up to time $S$ for all small $\eps$. Let $\mM^{\eps}_{S}$ and $\mM_{S}$ denote the maps associated to the fixed point problem \eqref{e:fixed_pt} for $\eps>0$ and $\eps=0$ respectively. Previous arguments have already shown that
	\begin{equ}
	\|\mM_{S}^{(\eps)}(U^{(\eps)}); \mM_{S}^{(\eps)}(U)\|_{\gamma,\eta;\eps} \lesssim (S+\eps)^{\theta} \|U^{(\eps)}; U\|_{\gamma,\eta;\eps} + |\!|\!| \Pi^{\eps}; \Pi |\!|\!|_{\eps;0} + \|u_{0}^{(\eps)}; u_{0}\|_{\gamma,\eta;\eps}, 
	\end{equ}
	where the proportionality constant is independent of $\eps$. If $U^{(\eps)}$ and $U$ are fixed point solutions, then for small enough $S$, we have
	\begin{equs} \label{e:iterate}
	\|U^{(\eps)}; U\|_{\gamma,\eta;\eps} \lesssim |\!|\!| \Pi^{\eps}; \Pi |\!|\!|_{\eps;0} + \|u_{0}^{(\eps)}; u_{0}\|_{\gamma,\eta;\eps},  
	\end{equs}
	which surely converges to $0$ (up to time $S$) as $\eps \rightarrow 0$. Furthermore, Proposition~\ref{pr:restart} guarantees that we can iterate \eqref{e:iterate} up to time $T$ in finitely many steps, thus completing the proof. 
\end{proof}

\subsection{Renormalised equation}

Let $\widehat{\lL}_{\eps}$ be the model defined in \eqref{e:model_2} and 
let $\rR$ be the reconstruction map associated with $\widehat{\lL}_{\eps}$. We have the following theorem. 

\begin{thm} \label{th:renorm_eq}
	Let $U \in \dD^{\gamma,\eta}_{\eps}$ be the abstract solution to the fixed point problem \eqref{e:fixed_pt} with initial data $u_0 \in \cC^{\gamma,\eta}_{\eps}$. Then, $u = \rR U$ solves the classical PDE
	\begin{equ}
	\partial_{t} u = \partial_{x}^{2} u + \eps^{-1} F(\eps^{\frac{1}{2}} \Psi_\eps + \eps^{\frac{1}{2}} \partial_{x} u) - C_{\eps}
	\end{equ}
	with initial condition $u_{0}$, where $C_\eps$ is given by \eqref{e:constantEps}.
\end{thm}
\begin{proof}
	Applying the reconstruction operator to both sides of \eqref{e:fixed_pt}, we get
	\begin{equ} \label{e:mild_eq}
	u = P * \one_{+} \Big( \rR \big( a (\<1'> + \sD U)^{2} + a\,\<0'> \, (\sD U)^{2} \big) + \eps^{-1} G(\eps^{\frac{1}{2}} \Psi_\eps, \eps^{\frac{1}{2}} \rR \sD U) \Big) + P_{t} u_{0}. 
	\end{equ}
A straightforward calculation shows that, modulo terms of homogeneity strictly greater than ${1\over 2}$, 
any fixed point $U$ to \eqref{e:fixed_pt} has the form
	\begin{equ}
	\sD U = a \, \<2'0> + v' \, \1 + 2 a^{2} \, \<2'1'0> + 2a v' \, \<1'0>\;, 
	\end{equ}
    for some space-time function $v'$. As a consequence, the factor multiplying $\one_+$ in 
    the right hand side of \eqref{e:fixed_pt} is of the form
	\begin{equ}
	a \, \<2'> + 2a^2\,\<2'1'> + 2av'\,\<1'>+ 4a^{3} \, \<2'1'1'> + a^{3} \, \<2'2'0> + 2a^2 v'\,\<2'0> + a^3 \, \<2'2'0'> + 2a^2v'\,\<2'0'> + w \,\1\;, 
	\end{equ}
	for some continuous function $w$, modulo terms of strictly positive homogeneity.
The claim then follows from the definition of the model in \eqref{e:model_new} and \eqref{e:model_2}
in essentially the same way as \cite[Prop.~9.10]{Hai14a} for example.
\end{proof}

\section{Preliminary bounds}
\label{sec:preliminary}

In this section, we give some preliminary lemmas and bounds that will be useful later when we establish a general pointwise bound and prove the main convergence.

\subsection{Bounds on the free field}

Recall from \eqref{e:mollified_noise} that
\begin{equ}
\xi_{\eps} \stackrel{\text{\tiny law}}{=} \xi * \rho_{\eps}
\end{equ}
for some mollifier $\rho$ symmetric in its space variable and rescaled at scale $\eps$. Here, ``*'' denotes the space-time convolution. Also recall that $\Psi_{\eps} = P' * \xi_{\eps}$, where $P'$ is the spatial derivative of the heat kernel $P$ on the torus. We have the following bound on the correlation function of $\Psi_{\eps}$. It is the basis for most of the bounds throughout the article. 

\begin{lem} \label{le:corr_ff}
	The correlation $\E \big( \Psi_{\eps}(x) \Psi_{\eps}(y) \big)$ is a function of $x-y$, and this function is symmetric in its space variable. Furthermore, there exists $\Lambda > 0$ such that
	\begin{equ} \label{e:corr_ff}
	\frac{1}{\Lambda(|x-y|+\eps)} \leq \E \big( \Psi_{\eps}(x) \Psi_{\eps}(y) \big) \leq \frac{\Lambda}{|x-y|+\eps}
	\end{equ}
	for all $x,y \in \R^{+} \times \T$ and $\eps>0$. Here, $|x-y|$ denotes the parabolic distance between the space-time points $x$ and $y$. 
\end{lem}
\begin{proof}
	Since $\Psi_{\eps} = P' * \xi_{\eps} = (P' * \rho_{\eps}) * \xi$, the delta correlation of the space-time white noise gives
	\begin{equ}
	\E \big( \Psi_{\eps}(x) \Psi_{\eps}(y) \big) = ((P')^{\star 2} * (\rho_{\eps}^{\star 2}))(x-y), 
	\end{equ}
	where ``*'' is the usual space-time convolution, and ``$\star$'' is the forward convolution in space-time in the sense that
	\begin{equ}
	(f \star g)(x) = \int f(x+y) g(y) dy. 
	\end{equ}
	Assumption~\ref{as:noise} on the symmetry of covariance implies that $\rho_{\delta}^{\star 2}$ is also symmetric in its space variable. Since $(P')^{\star 2}$ also has the same symmetry, so does $\E \big( \Psi_{\eps}(x) \Psi_{\eps}(y) \big)$. 
	
	For the second claim, it is easy to check that $(P')^{\star 2} = \frac{1}{2} P$ and $\rho_{\eps}^{\star 2} = (\rho^{\star 2})_{\eps}$. Thus, both bounds in \eqref{e:corr_ff} follow from the properties of the heat kernel. 
\end{proof}

By stationarity, we let
\begin{equ}
\varrho_{\eps}(x-y) = \E \big( \Psi_{\eps}(x) \Psi_{\eps}(y) \big)\;. 
\end{equ}
The following lemma, based on Lemma~\ref{le:corr_ff} and the triangle inequality, will also be used throughout Section~\ref{sec:general_bound} to get the general pointwise bound. 

\begin{lem} \label{le:corr_change}
	Let $K \geq 1$. For every $z,z' \in \R^{+} \times \T$ with $|z'| \leq K|z|$ in the parabolic metric, we have
	\begin{equ} \label{e:corr_change}
	\varrho_{\eps}(z) \leq K \Lambda^{2} \varrho_{\eps}(z'). 
	\end{equ}
	\begin{proof}
		By Lemma~\ref{le:corr_ff}, we have
		\begin{equ}
		\varrho_{\eps}(z') \geq \frac{1}{\Lambda(|z'|+\eps)} \geq \frac{1}{K \Lambda^{2}} \cdot \frac{\Lambda}{|z|+\eps} \geq \frac{1}{K \Lambda^{2}} \cdot \varrho_{\eps}(z), 
		\end{equ}
		where in the second inequality we have used the assumption $|z'| \leq K|z|$ as well as $K \geq 1$. 
	\end{proof}
\end{lem}

\subsection{Interchanging the supremum with the expectation}

Fix an integer $d \geq 1$. For each $k=1, \dots, d$, let $J_k = [c_k, C_k]$ be a fixed interval. For every function $\Phi$ on $\jJ = J_1 \times \cdots \times J_d$ and every subset $I = \{i_1, \dots, i_m\} \subset \{1, \dots, d\}$, we write
\begin{equ}
\partial_{I} \Phi = \partial_{i_1} \cdots \partial_{i_m} \Phi. 
\end{equ}
We also let $\Phi_{I}$ be the function of $|I|$ variables such that
\begin{equ}
\Phi_{I}: J_{i_1} \times \cdots \times J_{i_m} \rightarrow \R
\end{equ}
is the restriction of $\Phi$ to variables from $I$ while the remaining variables take the value $c_k$ for $k \notin I$. For convenience, we write
\begin{equ}
\jJ_{I} = J_{i_1} \times \cdots \times J_{i_m}, 
\end{equ}
and use $|\jJ_I|$ to denote the volume of $\jJ_I$. Finally, we let $\Btheta = (\theta_1, \dots, \theta_K)$ and $\Btheta_I = (\theta_{i_1}, \dots, \theta_{i_m})$. With these notations, we have the following lemma. 

\begin{lem} \label{le:interchange}
	Let $\Phi$ be a random smooth function on $\jJ = J_1 \times \cdots \times J_d$. Then, we have
	\begin{equ}
	\E \sup_{\Btheta \in \jJ} |\Phi(\Btheta)|^{2n} \leq 2^{2nd} \sum_{I} |\jJ_{I}|^{2n} \sup_{\Btheta \in \jJ} \E |(\partial_{I} \Phi) (\Btheta)|^{2n}, 
	\end{equ}
	where the sum is taken over all subsets $I$ of $\{1, \dots, d\}$, and $|\jJ_I| = 1$ if $I = \emptyset$. 
\end{lem}
\begin{proof}
	By repeatedly applying the fundamental theorem of calculus, we see that for every $\Btheta \in \jJ$ one
	has the identity
	\begin{equ}
	\Phi(\Btheta) = \sum_{m=0}^{d} \sum_{I = (i_1, \dots, i_m)} \int_{c_{i_1}}^{\theta_{i_1}} \cdots \int_{c_{i_m}}^{\theta_{i_m}} (\partial_{I} \Phi)_{I}(\u_{I}) d \u_{I}, 
	\end{equ}
	where the sum on the right hand side is taken over all size-$m$ subsets $I$ of $\{1, \dots, d\}$, and then over all $m$. We have also used the convention that the term corresponding to $m=0$ (or $I = \emptyset$) is the constant term $\Phi(c_1, \dots, c_d)$. We then have the bound
	\begin{equ}
	\sup_{\Btheta \in \jJ} |\Phi(\Btheta)| \leq \sum_{I} \int_{\jJ_{I}} |(\partial_{I} \Phi)_{I}(\u_{I})| d \u_{I}. 
	\end{equ}
	Since there are $2^{d}$ terms in the sum, raising both sides to the power $2n$, and using H\"{o}lder's inequality for each of the integrals on the right hand side, we obtain
	\begin{equ}
	\sup_{\Btheta \in \jJ} |\Phi(\Btheta)|^{2n} \leq 2^{2nd} \sum_{I} |\jJ_I|^{2n-1} \int_{\jJ_{I}} |(\partial_{I} \Phi)_{I}(\u_{I})|^{2n} d \u_{I}. 
	\end{equ}
	Taking expectation on both sides, and replacing the integrals by the supremum norm, we get
	\begin{equ}
	\E \sup_{\Btheta \in \jJ} |\Phi(\Btheta)|^{2n} \leq 2^{2nd} \sum_{I} |\jJ_{I}|^{2n} \sup_{\Btheta_{I} \in \jJ_{I}} \E |(\partial_{I}\Phi)_{I}(\Btheta_I)|^{2n}. 
	\end{equ}
	Finally, note that the $(\partial_{I} \Phi)_{I}$ is a function of $|I|$ variables, and the other $d-|I|$ ones are fixed at values $c_k$'s. Thus, by allowing taking the supremum of all the variables $\theta_k$'s, we enlarge the upper bound to match the statement of the lemma. This completes the proof. 
\end{proof}

\begin{rmk}
	In the proof of the lemma, one controls the $L^{\infty}$ norm of $\Phi$ by the $L^1$ norms of its derivatives up to order $d$. From Sobolev embedding, one can reduce the number of derivatives by increasing the integrability of them on the right hand side. But it turns out later that in our context, both ways yield the same control for $\big( \E \sup_{\Btheta} |\Phi(\Btheta)|^{2n} \big)^{\frac{1}{2n}}$. 
\end{rmk}

\subsection{Localisation and decomposition of \texorpdfstring{$\hF$}{F}}
\label{sec:FT}

For every $d \geq 1$ and $\Btheta \in \R^{d}$, let $\fR_{\Btheta}$ be the rectangle of side length $2$ in $\R^{d}$ centred at $\Btheta$.The value of $d$ varies in different contexts in this article, and hence so is the dimension of $\fR_{\Btheta}$. We let $|\Btheta| := \sum_{i=1}^{d} |\theta_i|$. For every open bounded set $\Omega \subset \R^{d}$ and $M \in \N$, we define a norm $\|\cdot\|_{\bB_{M}(\Omega)}$ on $\cC_{c}^{\infty}(\Omega)$ functions by
\begin{equ}
\|\Bphi\|_{\bB_{M}(\Omega)} := \sup_{\r: 0 \leq r_i \leq M} \sup_{\x \in \Omega} |(\d^{\r} \Bphi)(\x)|, 
\end{equ}
where $\x=(x_1, \dots, x_d)$, $\r = (r_1, \dots, r_d)$, and $\d^{\r}= \d_{x_1}^{r_1} \cdots \d_{x_d}^{r_d}$. Let $\bB_{M}(\Omega)$ be the closure of $\cC_{c}^{\infty}(\Omega)$ under that norm. The space $\bB_{M}$ is different from the usual H\"{o}lder space $\cC^{M}$ in that it requires every directional derivative of $\Bphi$ up to order $M$, not just the total number of its derivatives. But they do coincide when $d=1$. For every integer $M$ and open set $\Omega \subset \R^{d}$, we define a norm $\|\cdot\|_{M,\Omega}$ on distributions on $\R^{d}$ by
\begin{equ}
\|\Ups\|_{M,\Omega} := \sup_{\Bphi: \|\Bphi\|_{\bB_{M}(\Omega)} \leq 1} |\scal{\Ups,\Bphi}|, 
\end{equ}
where the supremum is taken over all $\phi \in \cC_{c}^{\infty}(\Omega)$ with the $\bB_{M}(\Omega)$ norm bounded by $1$. We have the following lemma.

\begin{lem} \label{le:dist_product}
	Suppose $\Ups$ is a distribution on $\R^{d}$ with the form $\Ups = \otimes_{i=1}^{d} \Ups^{i}$, where each $\Ups^{i}$ is a distribution on $\R$. Then, we have
	\begin{equ}
	\|\Ups\|_{M,\fR_{\K}} \leq \prod_{i=1}^{d} \|\Ups^{i}\|_{M,\fR_{K_i}}
	\end{equ}
	for all $\K = (K_1, \dots, K_d) \in \Z^{d}$. 
\end{lem}
\begin{proof}
	The case $d=1$ is immediate. For $d \geq 2$, we have
	\begin{equs}
	\big|\scal{\Ups, \Bphi} \big| &= \big|\bscal{\Ups^{d}, \scal{\otimes_{i=1}^{d-1}\Ups^i, \Bphi(\cdot, \dots, \cdot, x_d)}} \big|\\
	&\leq \|\Ups^{d}\|_{M,\fR_{K_d}} \cdot \sup_{r_d \leq M} \sup_{x_d \in \fR_{K_d}} \big| \bscal{\otimes_{i=1}^{d-1} \Ups^{i}, (\d_{x_d}^{r_d} \phi)(\cdot, \dots, \cdot, x_d) } \big|, 
	\end{equs}
	and the claim follows by induction. 
\end{proof}

\begin{prop} \label{pr:dist_decom}
	Let $\rho$ be a mollifier on $\R$. For every $\delta>0$, let $\rho_{\delta}(\cdot) = \delta^{-1} \rho(\cdot/\delta)$. For $i=1,\dots, d$, let $G^i$ be a Schwartz distribution on $\R$ and write $G^i_{\delta} = G^i * \rho_{\delta}$. Let
	\begin{equ}
	\Ups = \bigotimes_{i=1}^{d} \widehat{G^i}, \qquad \Ups_{\delta} = \bigotimes_{i=1}^{d} \widehat{G^i_{\delta}}, 
	\end{equ}
	where $\widehat{G^i}$ and $\widehat{G^i_{\delta}}$ are Fourier transforms of $G^i$ and $G^{i}_{\delta}$. Then for every $M,N \in \Z$, there exists $C = C(d,M,N) > 0$ such that
	\begin{equ} \label{e:T_smooth}
	\|\Ups_{\delta}\|_{M, \fR_{\K}} \leq C \delta^{-N} (1+|\K|)^{-N} \prod_{i=1}^{d} \|\widehat{G^i}\|_{M, \fR_{K_i}} 
	\end{equ}
	for all $\K = (K_1, \dots, K_d) \in \Z^{d}$ and all $\delta \in (0,1)$. Furthermore, there exists $C = C(d,M) > 0$ such that
	\begin{equ} \label{e:T_remainder}
	\|\Ups-\Ups_{\delta}\|_{M, \fR_{\K}} \leq C \delta^{\beta} \prod_{i=1}^{d} \Big[ (1+|K_i|)^{\beta} \|\widehat{G^i}\|_{M, \fR_{K_i}} \Big]
	\end{equ}
	for all $\K \in \Z^{d}$, all $\beta \in (0,1)$ and all $\delta \in (0,1)$. 
\end{prop}
\begin{proof}
	We first prove the bound for $\Ups_{\delta}$. By Lemma~\ref{le:dist_product}, we have
	\begin{equ} \label{e:dist_prod}
	\|\Ups_{\delta}\|_{M,\fR_{\K}} \leq \prod_{i=1}^{d} \|\widehat{G^i_{\delta}}\|_{M, \fR_{K_i}}. 
	\end{equ}
	For every $i$, we have $\scal{\widehat{G^i_{\delta}}, \phi} = \scal{\widehat{G^i}, \widehat{\rho}(\delta \cdot) \phi}$. Since for every $\phi \in \bB_{M}(\fR_{K_i})$, the function $\widehat{\rho}(\delta \cdot) \phi$ also belongs to $\bB_{M}(\fR_{K_i})$ with the bound
	\begin{equ}
	\|\widehat{\rho}(\delta \cdot) \phi\|_{\bB_{M}(\fR_{K_i})} \leq C_{M,N} (1 + \delta |K_i|)^{-N} \|\phi\|_{\bB_{M}(\fR_{K_i})}, 
	\end{equ}
    we can deduce that
    \begin{equ} \label{e:G_smooth}
    \|\widehat{G^i_{\delta}}\|_{M,\fR_{K_i}} \leq C_{M,N} (1+\delta |K_i|)^{-N} \|\widehat{G^i}\|_{M,\fR_{K_i}}. 
    \end{equ}
    Plugging \eqref{e:G_smooth} into \eqref{e:dist_prod} and noticing that
	\begin{equ}
	\prod_{i}(1+\delta |K_i|)^{-1} \leq (1 + \delta |\K|)^{-1} \leq \delta^{-1} (1+|\K|)^{-1}
	\end{equ}
	for all $\delta < 1$, we obtain the bound \eqref{e:T_smooth} for $T_{\delta}$. As for the difference $\Ups-\Ups_{\delta}$, it can be expressed by
	\begin{equ}
	\Ups - \Ups_{\delta} = \sum_{j=1}^{d} \Big[\bigotimes_{i=1}^{j-1} \widehat{G^i_{\delta}} \otimes (\widehat{G^{j}} - \widehat{G^j_{\delta}}) \otimes \bigotimes_{i=j+1}^{d} \widehat{G^i} \Big]. 
	\end{equ}
	By Lemma~\ref{le:dist_product} and the bound \eqref{e:G_smooth} with $N=0$, we have
	\begin{equ}
	\|\Ups-\Ups_{\delta}\|_{M,\fR_{\K}} \leq C_{M} \sum_{j=1}^{d} \Big[ \|\widehat{G^j} - \widehat{G^j_{\delta}}\|_{M,\fR_{K_j}} \prod_{i \neq j} \|\widehat{G^i}\|_{M,\fR_{K_i}} \Big]. 
	\end{equ}
	It then suffices to bound $\|\widehat{G^j} - \widehat{G^j_{\delta}}\|_{M,\fR_{K_j}}$ for each $j$. In fact, since $\scal{\widehat{G^j} - \widehat{G^j_{\delta}}, \phi} = \scal{\widehat{G^j}, (1-\widehat{\rho}(\delta \cdot)) \phi}$, and
	\begin{equ}
	\| \big(1-\widehat{\rho}(\delta \cdot) \big) \phi \|_{\bB_{M}(\fR_{K_j})} \leq C_{M} \delta^{\beta} (1+|K_j|)^{\beta} \|\phi\|_{\bB_{M}(\fR_{K_j})}
	\end{equ}
	for every $\beta \in (0,1)$, we deduce immediately from the definition of $\| \cdot \|_{M, \fR_{K_j}}$ that
	\begin{equ}
	\|\widehat{G^j} - \widehat{G^j_{\delta}}\|_{M,\fR_{K_j}} \leq C_{M} \delta^{\beta} (1+|K_j|)^{\beta} \|\widehat{G^j}\|_{M,\fR_{K_j}}. 
	\end{equ}
	The bound \eqref{e:T_remainder} then follows. 
\end{proof}

\begin{lem} \label{le:dist_derivative}
	Let $G$ be a Schwartz distribution on $\R$. Then, for every $\ell, M \in \N$, there exists a constant $C>0$ such that
	\begin{equ}
	\|\widehat{G^{(\ell)}}\|_{M, \fR_K} \leq C (1+|K|)^{\ell} \|\widehat{G}\|_{M, \fR_K}
	\end{equ}
	for all $K \in \Z$. 
\end{lem}
\begin{proof}
	We have $\scal{\widehat{G^{(\ell)}}, \phi} = \scal{\widehat{G}, \tilde{\phi}}$, where $\tilde{\phi}(\theta) = (i\theta)^{\ell} \phi(\theta)$. For $\phi \in \cC_{c}^{\infty}(\fR_K)$, $\tilde{\phi}$ also belongs to $\cC_{c}^{\infty}(\fR_K)$ with the bound
	\begin{equ}
	\|\tilde{\phi}\|_{\bB_{M}(\fR_K)} \lesssim_{\ell,M} (1+|K|)^{\ell} \|\phi\|_{\bB_{M}(\fR_K)}. 
	\end{equ}
	The claim then follows immediately. 
\end{proof}

\begin{lem} \label{le:FT_decay}
	Suppose $F \in \cC^{k+\alpha}$ for some integer $k$ and $\alpha \in [0,1)$, and its derivatives satisfy the bounds
	\begin{equ}
	\sup_{0 \leq \ell \leq k} |F^{(\ell)}(x)| \leq C(1+|x|)^{M}, \quad \sup_{|h|<1} \frac{|F^{(k)}(x+h) - F^{(k)}(x)|}{|h|^{\alpha}} \leq C (1+|x|)^{M}
	\end{equ}
	for all $x \in \R$. Then, there exists $C = C(k,M) > 0$ such that
	\begin{equ}
	\|\hF\|_{M+2,\fR_{K}} \leq C (1+|K|)^{-k-\alpha}
	\end{equ}
	for every $K \in \Z$. 
\end{lem}
\begin{proof}
	For $\phi: \R \to \R$, let $\phi_{K}(\cdot) = \phi(\cdot-K)$. By definition, we have
	\begin{equ} \label{e:F_hat_norm}
	\|\hF\|_{M+2, \fR_{K}} = \sup_{\phi} |\scal{\hF, \phi_K}| = \sup_{\phi} \Big| \int_{\R} F(x) \widehat{\phi}(x) e^{-iKx} dx \Big|, 
	\end{equ}
	where the supremum is taken over all $\phi \in \cC_{c}^{\infty}(\fR_{0})$ such that $\|\phi\|_{\cC^{M+2}(\fR_{0})} \leq 1$. Here, the $B_{M}$ and $C^{M}$ norms are equivalent since we are in dimension one. We need to look at $\|\cdot\|_{M+2;\fR_{K}}$ norm of $\hF$ since $F$ has growth of order $M$, and the $(M+2)$-differentiability of $\phi$ implies the derivatives of its Fourier transform has the decay
	\begin{equ}
	|\widehat{\phi}^{(\ell)}(x)| \lesssim_{\ell} (1+|x|)^{-M-2}
	\end{equ}
	for all $\ell \geq 1$, so the right hand side of \eqref{e:F_hat_norm} is integrable. If $k=0$, we can write the right hand side of \eqref{e:F_hat_norm} as
	\begin{equ}
	\int_{\R} F(x) \widehat{\phi}(x) e^{-iKx} dx = \frac{1}{2} \int_{\R} \Big( F(x) \widehat{\phi}(x) - F(x-\frac{\pi}{K}) \widehat{\phi}(x-\frac{\pi}{K}) \Big) e^{-iKx} dx. 
	\end{equ}
	Then, by the H\"{o}lder continuity of $F$ and the decay of $\widehat{\phi}$ and its derivatives, we have
	\begin{equ}
	\|\hF\|_{M+2,\fR_{K}} \leq C (1+|K|)^{-\alpha}. 
	\end{equ}
	If $k \geq 1$, integrating by parts $k$ times, we have (for $|K| \neq 0$)
	\begin{equ}
	\Big|\int_{\R} F(x) \widehat{\phi}(x) e^{-iKx} dx\Big| = |K|^{-k} \Big| \sum_{\ell=0}^{k} \begin{pmatrix} k \\ \ell \end{pmatrix} \int_{\R}  F^{(\ell)}(x) \widehat{\phi}^{(k-\ell)}(x) e^{-iKx} dx \Big|. 
	\end{equ}
	For the terms with $\ell \leq k-1$ in the sum above, one can further integrate by parts to get the decay of order $\oO(|K|^{-k-1})$. For $\ell=k$, since $F^{(k)} \in \cC^{\alpha}$, the decay $\oO(|K|^{-k-\alpha})$ follows from the pre-factor $|K|^{-k}$ and the previous case for $\cC^{\alpha}$ functions. This completes the proof. 
\end{proof}

The next proposition is the main property of $F$ we will use in this article. 

\begin{prop} \label{pr:F_decomp}
	Let $F\colon \R \to \R$ satisfy Assumption~\ref{as:F}. Let $\Bell = (\ell_1, \dots, \ell_d) \in \N^{d}$, and
	\begin{equ}
	\Ups = \bigotimes_{i=1}^{d} \widehat{F^{(\ell_i)}}, \qquad \Ups_{\delta} = \bigotimes_{i=1}^{d} \widehat{F^{(\ell_i)}_{\delta}}. 
	\end{equ}
	Then, for every $N>0$, there exists a constant $C$ depending on $N$ and $\ell$ such that
	\begin{equ} \label{e:F_smooth}
	\|\Ups_{\delta}\|_{M+2,\fR_{\K}} \leq C \delta^{-N} (1+ |\K|)^{-N}
	\end{equ}
	for all $\K \in \Z^d$ and $\delta \in (0,1)$. For the difference $T-T_{\delta}$, we have the bound
	\begin{equ} \label{e:F_remainder}
	\|\Ups-\Ups_{\delta}\|_{M+2, \fR_{\K}} \leq C \delta^{\beta} \prod_{i=1}^{d} (1+|K_i|)^{-7-\alpha+\ell_i+\beta}
	\end{equ}
	for all $\K \in \Z^d$, $\delta \in (0,1)$ and $\beta \in (0,1)$. Here, $\alpha$ and $M$ are the same as in Assumption~\ref{as:F}. 
\end{prop}
\begin{proof}
	This follows directly from Assumption~\ref{as:F}, Lemmas~\ref{le:dist_derivative},~\ref{le:FT_decay} and Proposition~\ref{pr:dist_decom} by setting $G^i = F^{(\ell_i)}$. 
\end{proof}

We also need the following proposition later.

\begin{prop} \label{pr:localisation}
	Let $\Ups$ be a distribution on $\R^{d}$ and $\Phi \in \cC^{\infty}(\R^{d})$, which is allowed to grow at infinity. Then, we have
	\begin{equ} \label{e:dist_local}
	|\scal{\Ups, \Phi}| \lesssim_{M} \sum_{\K \in \Z^{d}} \|\Ups\|_{M+2,\fR_{\K}} \sup_{\r: r_j \leq M+2} \sup_{\Btheta \in \fR_{\K}} |\partial^{\r} \Phi(\Btheta)|. 
	\end{equ}
\end{prop}
\begin{proof}
	Multiply $\Phi$ by a partition of unity and estimate the terms separately.
\end{proof}

\subsection{Behaviour of the coefficients in the chaos expansion}

In \cite{HQ}, after Wick renormalisation, each homogeneous chaos component of a polynomial model has logarithmically divergent, finite and vanishing parts as $\eps \rightarrow 0$. The same is true in our context with a general $F$. Thus, in order to identify the limits of our models with that of the KPZ equation, we need to distinguish these three parts for each of our objects. The situation here is more complicated as the models are constructed from a general function $F$ rather than a polynomial, so we need to get fine control on the coefficients of the terms in the chaos expansion. The aim of this section is to give a few lemmas that provide the necessary controls for the object $\tau = \<2'1'1'>$. This is the most complicated object among those in Table \eqref{e:symbols}, and will be treated in detail in Section~\ref{sec:convergence}. 

Throughout, we fix the mollifier $\rho$ on $\R$. For every $\delta>0$, write $\rho_{\delta}(\cdot) = \delta^{-1} \rho(\cdot/\delta)$ and $F_{\delta} = F * \rho_{\delta}$. Let $a_{\delta}$ be defined as the same as the coupling constant $a$ in \eqref{e:coupling_const} except that one replaces $F''$ by $F_{\delta}''$ in \eqref{e:coupling_const}. It is easy to check that $|a - a_{\delta}| \lesssim \delta$. For space-time points $x$, $y$, $z$, we will use the notation
\begin{equ}[e:notation]
X = \eps^{\frac{1}{2}} \Psi_{\eps}(x), \qquad Y = \eps^{\frac{1}{2}} \Psi_{\eps}(y), \qquad Z = \eps^{\frac{1}{2}} \Psi_{\eps}(z). 
\end{equ}
We omit the dependence of these random fields on $\eps$ for notational simplicity. Our aim is to have control on the coefficients that is uniform over all space-time points and all $\eps > 0$. Also recall the re-centering notation $\cent{W} = W - \E W$. We have the following lemmas. 

\begin{lem} \label{le:101}
	With the notation \eqref{e:notation}, we have the decomposition
	\begin{equs} \label{e:101_decomposition}
	&\phantom{11}\E \big( F''(X) F'(Y) F'(Z) \big) - 8a^3 \E YZ\\
	&= \E \big( \cent{F''(X)} F'(Y) F'(Z) \big) + 2a \E \big(F'(Y) F'(Z) -4a^2 YZ \big). 
	\end{equs}
	For the first term on the right hand side, there exist $N, \beta > 0$ such that
	\begin{equ} \label{e:101_1}
	|\E \big(\cent{F''(X)}F'(Y)F'(Z) \big)| \lesssim \delta^{-N} (\E XY) (\E XZ) + \delta^{\beta} \E YZ, 
	\end{equ}
	where the proportionality constant is uniform over $\delta \in (0,1)$. For the second term, there exists $\beta' > 0$ such that
	\begin{equ} \label{e:101_2}
	|\E \big(F'(Y) F'(Z) -4a^2 YZ \big)| \lesssim (\E YZ)^{1+\beta'}. 
	\end{equ}
	Both bounds above are uniform over all $\eps \in (0,1)$ and all locations of the points $x,y,z$. 
\end{lem}
\begin{proof}
	The identity \eqref{e:101_decomposition} is immediate since $\E F''(X) = 2a$. To see the bound \eqref{e:101_1}, we let $\Ups = \widehat{F''} \otimes \widehat{F'} \otimes \widehat{F'}$, and $T_{\delta}$ be the distribution that replaces every appearance of $F^{(\ell)}$ in $\Ups$ by $F_{\delta}^{(\ell)}$. We then write
	\begin{equ}
	\E \big(\cent{F''(X)}F'(Y)F'(Z) \big) = \scal{\Ups_{\delta}, \E \Phi_{\eps}} + \scal{\Ups-\Ups_{\delta}, \E \Phi_{\eps}}, 
	\end{equ}
	where
	\begin{equ}
	\Phi_{\eps}(\Btheta;x,y,z) = \cent{\cos(\theta_{\fx} X)} \sin(\theta_{\fy}Y) \sin(\theta_{\fz} Z), 
	\end{equ}
	is viewed as a test function in the variable $\Btheta = (\theta_{\fx}, \theta_{\fy}, \theta_{\fz})$. Following the notation in Section~\ref{sec:general_statement}, the type space here is $\tilde{\tT} = \{\fx, \fy, \fz\}$ with $\tilde{\oO} = \{\fy, \fz\}$ and $\tilde{\eE} = \{\fx\}$. For the term $\scal{\Ups_{\delta}, \E \Phi_{\eps}}$, we have by Proposition~\ref{pr:localisation}
	\begin{equ} \label{e:101_localise}
	|\scal{\Ups_{\delta}, \E \Phi_{\eps}}| \leq \sum_{\K \in \Z^{3}} \|\Ups_{\delta}\|_{M+2;\fR_{\K}} \sup_{|\r|_{\infty} \leq M+2} \sup_{\Btheta \in \fR_{\K}} |\d_{\Btheta}^{\r} \E \Phi_{\eps}|, 
	\end{equ}
	where $|\r|_{\infty}$ denotes the largest component of the multi-index $\r$. Interchanging differentiation and expectation and applying the bound \eqref{e:special1}, we get
	\begin{equ}
	\sup_{\Btheta \in \fR_{\K}}|\d_{\Btheta}^{\r} \E \Phi_{\eps}| = \sup_{\Btheta \in \fR_{\K}} |\E (\d_{\Btheta}^{\r} \Phi_{\eps})| \lesssim (1+|\K|)^{4} (\E XY) (\E XZ). 
	\end{equ}
	Here, the power of $|\K|$ is $4$ since $\tilde{\oO}$ has cardinality $2$ and $\tilde{\eE}$ is a singleton. Plugging the above bound back into \eqref{e:101_localise} and applying \eqref{e:F_smooth} to the term $\|\Ups_{\delta}\|_{M+2;\fR_{\K}}$, we get
	\begin{equ}
	|\scal{\Ups_{\delta}, \E \Phi_{\eps}}| \lesssim \delta^{-N} (\E XY) (\E XZ)
	\end{equ}
	for some $N$. As for the term with $\Ups-\Ups_{\delta}$, we let
	\begin{equ}
	\Bphi_{\eps} = \cos(\theta_{\fx}X) \sin(\theta_{\fy}Y) \sin(\theta_{\fz}Z), 
	\end{equ}
	and write $\Phi_{\eps} = \Bphi_{\eps} + (\Phi_{\eps} - \Bphi_{\eps})$. We estimate the action of $\Ups-\Ups_{\delta}$ on these two terms separately. For the one with $\Bphi_{\eps}$, \eqref{e:special2} implies
	\begin{equ}
	\sup_{\Btheta \in \fR_{\K}} |\d_{\Btheta}^{\r} \E \Bphi_{\eps}| \lesssim_{\r} (1+|\K|)^{2} (\E YZ). 
	\end{equ}
	Here, there are only two powers on $|\K|$ since we counted the cardinality of the set $\tilde{\oO}$ only. We then use Proposition~\ref{pr:localisation} and the bound \eqref{e:F_remainder} to conclude that
	\begin{equ}
	|\scal{\Ups-\Ups_{\delta}, \E \Bphi_{\eps}}| \lesssim \delta^{\beta} \E YZ. 
	\end{equ}
	The same bound holds for $\scal{\Ups-\Ups_{\delta}, \Phi_{\eps}-\Bphi_{\eps}}$. The derivation is only simpler since the field involving $X$ is averaged out. This proves the bound \eqref{e:101_1}. 
	
	To see \eqref{e:101_2}, one can write $F'$ in terms of chaos expansion and use Wick's formula to get the expression
	\begin{equ} \label{e:sum_correlation}
	\E \big( F'(Y) F'(Z) - 4a^2 YZ \big) = \sum_{n \geq 1} \frac{c_n^2}{(2n+1)!} (\E YZ)^{2n+1},
	\end{equ}
	where $c_{n} = \E \big(F^{(2n+2)}(Y) \big)$. Let $\sigma^{2} = \E Y^2 = \E Z^2$. Since
	\begin{equ}
	\E \big(F'(Y)\big)^{2} = \sum_{n \geq 0} \frac{c_n^2 \sigma^{4n+2}}{(2n+1)!} < +\infty, 
	\end{equ}
	and $\E YZ \leq \sigma^{2}$, we can replace the term $(\E YZ)^{2n+1}$ in \eqref{e:sum_correlation} by $\sigma^{2(2n-\beta')} (\E YZ)^{1+\beta'}$ and still get a convergent series. This implies the bound \eqref{e:101_2} and hence completes the proof. 
\end{proof}

\begin{lem} \label{le:112}
	There exist $\beta, N > 0$ such that
	\begin{equ}
	|\E \big( F''(X) F''(Y) F''(Z) \big) - 8a^3| \lesssim \delta^{-N} \big( (\E XY)^2 + (\E YZ)^2 + (\E ZX)^2 \big) + \delta^{\beta}. 
	\end{equ}
\end{lem}
\begin{proof}
	Again, we decompose the quantity into a main smooth part and a small remainder by
	\begin{equ}
	\E \big( F''(X) F''(Y) F''(Z) \big) - 8a^3 =  \E  \big( F_{\delta}''(X) F_{\delta}''(Y) F_{\delta}''(Z) \big) - 8a_{\delta}^3  + \text{remainder}, 
	\end{equ}
	where $a_{\delta}$ is given at the beginning of this subsection. The remainder is easily seen to be bounded by $\delta^{\beta}$ for some $\beta > 0$. As for the main smooth part, we write it as
	\begin{equs}
	\begin{split}
	&\phantom{111}\E \big( F_{\delta}''(X) F_{\delta}''(Y) F_{\delta}''(Z) \big) - 8a_{\delta}^3 = \E \big( \cent{F_{\delta}''(X)} \cent{F_{\delta}''(Y)} \cent{F_{\delta}''(Z)} \big)\\
	&+ 2 a_{\delta} \E \big( \cent{F_{\delta}''(X)} \cent{F_{\delta}''(Y)} +  \cent{F_{\delta}''(Y)} \cent{F_{\delta}''(Z)} + \cent{F_{\delta}''(Z)} \cent{F_{\delta}''(X)} \big)\\
	&+ 4 a_{\delta}^2 \E \big( \cent{F_{\delta}''(X)} + \cent{F_{\delta}''(Y)} + \cent{F_{\delta}''(Z)} \big). 
	\end{split}
	\end{equs}
	The last term above vanishes by definition of $\cent{\cdot}$. For the second term, the same argument via chaos expansion as in the proof for \eqref{e:101_2} gives the bound $(\E XY)^2 + (\E YZ)^2 + (\E ZX)^2$ uniformly in $\delta$. As for the first one, we repeat the same argument as for the term $\scal{T_{\delta}, \E \Phi_{\eps}}$ in Lemma~\ref{le:101} and apply \eqref{e:special1} to obtain the bound
	\begin{equs}
	\E \big( \cent{F_{\delta}''(X)} \cent{F_{\delta}''(Y)} \cent{F_{\delta}''(Z)} \big) &\lesssim \delta^{-N} (\E XY) (\E YZ) (\E ZX)\\ 
	&\lesssim \delta^{-N} \big( (\E XY)^2 + (\E YZ)^2 + (\E ZX)^2 \big). 
	\end{equs}
	This completes the proof of the lemma. 
\end{proof}

\begin{lem} \label{le:000}
	We have the decomposition
	\begin{equ} \label{e:000_decomposition}
	\E \big( F'(X) F'(Y) \cent{F(Z)} - 4a^3 XYZ^{\diamond 2} \big) = H^{(1)} + H^{(2)}, 
	\end{equ}
	where
	\begin{equ}
	H^{(1)} = \E \Big( \sT_{(1)}(F'(X)) F'(Y) \cent{F(Z)} \Big), \quad H^{(2)} = 2a \E \Big( X \sT_{(1)}(F'(Y)) \cent{F(Z)} \Big), 
	\end{equ}
	and we use the notation $\sT_{(j)}(\cdot)$ from Section~\ref{sec:general_statement} to denote the random variable $\cdot$ obtained by removing its components belonging to the homogeneous Wiener chaoses of order up to (and including) $j$. Furthermore, $H^{(1)}$ and $H^{(2)}$ satisfy the following bounds and properties. For $H^{(1)}$, there exist $\beta, N>0$ such that
	\begin{equ} \label{e:000_1}
	|H^{(1)}| \lesssim \delta^{-N} (\E XY) (\E XZ)^2 + \delta^{\beta} (\E XZ) (\E YZ)
	\end{equ}
	for every $\delta > 0$. $H^{(2)}$ has the further decomposition
	\begin{equ} \label{e:sub_decomp}
	H^{(2)} = (\E XY) \E \big( \cent{F(Y)} \cent{F(Z)}\big) + (\E XZ) \E \big( \sT_{(1)}(F'(Y)) \sT_{(1)}(F'(Z)) \big), 
	\end{equ}
	and there exists $\beta' > 0$ such that
	\begin{equ} \label{e:000_2}
	|H^{(2)}| \lesssim (\E XY + \E XZ) (\E YZ)^{1+\beta'}.  
	\end{equ}
\end{lem}
\begin{proof}
	With the above definition of $H^{(j)}$, the difference between the left and right hand sides of \eqref{e:000_decomposition} is $4a^2 \E \big( XY \sT_{(2)}(F(Z)) \big)$, which is $0$ since $\sT_{(2)} \big( F(Z) \big)$ has components only in fourth and higher order chaos, and hence is orthogonal to $XY$. This verifies the decomposition \eqref{e:000_decomposition}. 
	
	The decomposition \eqref{e:sub_decomp} can be verified by chaos expanding $F'(Y)$ and $F(Z)$ and applying Wick's formula. The bound for $H^{(2)}$ then follows in the same way as the argument for proving \eqref{e:101_2}. 
	
	As for $H^{(1)}$, we let $\Ups = \widehat{F'} \otimes \widehat{F'} \otimes \widehat{F}$, and $\Ups_{\delta}$ be the distribution that replaces every $F^{(\ell)}$ in $\Ups$ by $F_{\delta}^{(\ell)}$. Similar as before, we write
	\begin{equ}
	H^{(1)} = \scal{\Ups_{\delta}, \E \Bphi_{\eps}} + \scal{\Ups-\Ups_{\delta}, \E \Bphi_{\eps}},  
	\end{equ}
	where
	\begin{equ}
	\Bphi_{\eps}(\Btheta;x,y,z) = \sT_{(1)}\big( \sin(\theta_{\fx}X)\big) \sin(\theta_{\fy}Y) \cent{\cos(\theta_{\fz}Z)}. 
	\end{equ}
	For the term with $\Ups_{\delta}$, the function $\Bphi_{\eps}$ does not fall exactly into the assumption of the general bounds in Section~\ref{sec:general_bound}, but one can still get the bound 
	\begin{equ}
	\sup_{\Btheta \in \fR_{\K}} |\d_{\Btheta}^{\r} \E \Bphi_{\eps}| \lesssim (1+|\K|)^{N'} (\E XY) (\E XZ)^{2}
	\end{equ}
	for some $N' > 0$. This can be achieved either by performing the same clustering argument as in Section~\ref{sec:general_bound} or by exact computation of $\E \Bphi_{\eps}$ using trigonometric identities (this is possible since there are only three terms in the product). Combining this bound with \eqref{e:F_smooth}, we get
	\begin{equ}
	|\scal{\Ups_{\delta}, \E \Bphi_{\eps}}| \lesssim \delta^{-N} (\E XY) (\E XZ)^{2}. 
	\end{equ}
	As for the term $\scal{\Ups-\Ups_{\delta}, \E \Bphi_{\eps}}$, we let 
	\begin{equ}
	\Phi_{\eps} = \sin(\theta_{\fx}X) \sin(\theta_{\fy}Y) \cent{\cos(\theta_{\fz}Z)}, 
	\end{equ}
	and write
	\begin{equ}
	\Bphi_{\eps} = \Phi_{\eps} + (\Bphi_{\eps} - \Phi_{\eps}). 
	\end{equ}
	Note that $\Phi_{\eps}$ precisely falls within the assumption of Theorem~\ref{th:special_bound}, and $\Phi_{\eps} - \Bphi_{\eps}$ has the same form as the corresponding term appearing in $H^{(2)}$. We can then estimate the two terms separately. More precisely, applying Proposition~\ref{pr:localisation} together with Theorem~\ref{th:special_bound} and then \eqref{e:F_remainder}, we get
	\begin{equ}
	|\scal{\Ups-\Ups_{\delta}, \Phi_{\eps}}| \lesssim \delta^{\beta} (\E XZ) (\E YZ). 
	\end{equ}
	Using the same procedure for $H^{(2)}$ and applying \eqref{e:F_remainder}, the other term can be bounded by $\delta^{\beta} (\E XZ) (\E YZ)^{2}$. Since we always have $\E YZ \lesssim 1$, it gives the same bound as $|\scal{\Ups-\Ups_{\delta}, \Phi_{\eps}}|$. We have thus completed the proof of the lemma. 
\end{proof}

\section{The main convergence theorem}
\label{sec:convergence}

In this section, we will assume the bounds stated in Theorems~\ref{th:general_bound} and~\ref{th:special_bound}, and use them to establish the convergence of our models to the limiting model which describes the KPZ equation. By stationarity of the input $\Psi_{\eps}$, we can use $0$ as the base point of our models without loss of generality. Again, $X$, $Y$ and $Z$ denote the value of the field $\eps^{\frac{1}{2}} \Psi_{\eps}(\cdot)$ at the point $x$, $y$ and $z$ respectively.

\subsection{Convergence of the models}

Recall that the renormalised model defined in \eqref{e:model_new} and \eqref{e:model_2} is built from $\Psi_{\eps}$ with the renormalisation constants
\begin{equs} [e:renorm_const]
C_{\<2's>}^{(\eps)} &= \frac{1}{a \eps} \E F(\eps^{\frac{1}{2}} \Psi_{\eps}) = \frac{\ha}{\eps}, \\
C_{\<2'2'0s>}^{(\eps)} &= \frac{1}{a^{2} \eps^{2}} \int K(x-y) K(x-z) \E \big( \cent{F(Y)} \cent{F(Z)} \big) \,dy\,dz, \\
C_{\<2'1'1's>}^{(\eps)} &= \frac{1}{4 a^{3} \eps^{2}} \int K(x-y) K(y-z) \E \big(F'(X) F'(Y) \cent{F(Z)} \big) \,dy\,dz, \\
C_{\<2'2'0's>}^{(\eps)} &= \frac{1}{2a^3 \eps^2} \int K(x-y) K(x-z) \E \big( \cent{F''(X)} \cent{F(Y)} \cent{F(Z)} \big) \,dy\,dz. 
\end{equs}
Here, $C_{\tau}^{(\eps)}$ is precisely the mean of $\PPi^{\eps} \tau$.
By stationarity of $\Psi_\eps$,
 the above definitions of these constants are all independent of the location of $x$. 

\begin{rmk}
	Note that since $\E \cent{F''(X)} = 0$, we also have
	\begin{equ}
	C_{\<2'2'0's>}^{(\eps)} = \frac{1}{2a^3 \eps^2} \int K(x-y) K(x-z) \E \big( \cent{F''(X)} \big(\cent{\cent{F(Y)} \cent{F(Z)}}\big) \big) \,dy\,dz, 
	\end{equ}
	which does indeed coincide with the mean of $\hPi^{\eps} \<2'2'0> \cdot \hPi^{\eps}\<0'>$, as
	mentioned just after \eqref{e:model_2}.
	
We also note that the symbols \<2'1'> and \<2'0'> are both of negative homogeneities. The reason why 
they do not exhibit renormalisations is because of the spatial anti-symmetry of $P'$, which leads the
corresponding expectation to vanish.
\end{rmk}

With the above choice of the constants $C_{\tau}^{(\eps)}$, we can show the convergence of our models to the KPZ model. This is the content of the following theorem.

\begin{thm} \label{th:main_converge}
	Let $\Psi_{\eps} = P' * \xi_{\eps}$, and $\hPi^{\eps} = \widehat{\lL}_{\eps}(\Psi_{\eps})$ be the renormalised model defined in \eqref{e:model_new} and \eqref{e:model_2} with input $\psi = \Psi_{\eps}$ and constants $C_{j}^{(\eps)}$ defined in \eqref{e:renorm_const}. Let $\Pi^{\KPZ}$ be the KPZ model described the Appendix. Then, there exists $\zeta>0$ such that for every $\tau$ in \eqref{e:symbols} with $|\tau|<0$, we have
	\begin{equ} \label{e:main_model_bound}
	\big( \E |\scal{\hPi^{\eps}_{z}\tau - \Pi^{\KPZ}_{z}\tau, \varphi_{z}^{\lambda}}|^{2n}\big)^{\frac{1}{2n}} \lesssim_{n} \eps^{\zeta} \lambda^{|\tau|+\zeta}, 
	\end{equ}
	where the bound holds uniformly over all $\eps \in (0,1)$, all $\lambda \in (0,1)$ and all space-time points $z$ in compact sets. As a consequence, we have $|\!|\!| \widehat\Pi^{\eps}; \Pi^{\KPZ} |\!|\!|_{\eps,0} \rightarrow 0$ in probability as $\eps \rightarrow 0$. 
\end{thm}

As long as the bound \eqref{e:main_model_bound} holds for all $\tau$ with negative homogeneities, we can proceed as in \cite[Prop.~6.3]{HQ} to conclude the convergence of $\hPi^{\eps}$ to the limiting KPZ model. From now on, we will focus on proving \eqref{e:main_model_bound} for symbols with $|\tau|<0$.

\subsection{Proof of Theorem~\ref{th:main_converge}}

According to Table \eqref{e:symbols}, there are ten basis elements with negative homogeneities, and hence we need to check the bound \eqref{e:main_model_bound} for all of them. The bounds on \<2'0> are a consequence of those for \<2'>. Thus, there are nine essentially different ones to check. 

For the sake of conciseness of the presentation, we provide details for three of them: \<1'>, \<2'1'1'> and \<2'2'0'>. The element \<1'> is simple but still illustrative enough to explain the general procedure. On the other hand, the symbol \<2'1'1'> is much more complicated, but contains all the subtleties that appear when dealing with the other symbols. Finally, \<2'2'0'> is the one whose convergence (to $0$) requires the strongest differentiability assumption on $F$ ($\cC^{7+}$), so we also include details for it. 

Before we start giving details for the above three symbols, we list the requirement on the decay of local norm of $\|\hF\|$ that guarantees the convergence for each symbol: 
\begin{equs} \label{e:decay_table}
\begin{array}{ *{10}{c} }
\toprule
\|\hF\|_{M+2,\fR_K} & $3$ & $3$ & $3$ & $3$ & $4$ & $5$ & $6$ & $7$ & $7$
\\
\midrule
\tau & \<0'> & \<1'> & \<2'> & \<1'1'> & \<2'1'> & \<2'2'0> & \<2'1'1'> & \<2'0'> & \<2'2'0'>
\\
\bottomrule
\end{array}
\end{equs}
Here, the number $n$ in the first row indicates that we need the decay $\|\hF\|_{M+2,\fR_K} = \oO(|K|^{-n-})$ in order for the corresponding process to converge. We will see that for the three symbols which we give detailed arguments below, the requirements listed above are indeed sufficient. The decay requirement for the other symbols can also be easily deduced in the same way. In the rest of the section, we use the notations
\begin{equ}
\x = (x_1, \dots, x_{2n}), \qquad \Bvarphi^{\lambda}(\x) = \prod_{k=1}^{2n} \varphi^{\lambda}(x_k). 
\end{equ}
Also, $F_{\delta}^{(\ell)}$ is $F^{(\ell)}$ regularised by a symmetric mollifier at scale $\delta$. Hence, $F_{\delta}^{(\ell)}$ is even if $\ell$ is, and it is odd if $\ell$ is odd. Finally, by translation invariance, it suffices to check the bound \eqref{e:main_model_bound} with $z=0$. 

\subsubsection[Term Psi]{The case \texorpdfstring{$\tau = \<1'>$}{Psi}}

We start with $\tau = \<1'>$. In what follows, we will always write $\tau_{\eps} = \hPi^{\eps}_{0} \tau$ 
for simplicity. Recall that $X = \eps^{\frac{1}{2}} \Psi_{\eps}(x)$. By \eqref{e:model_new}, we have
\begin{equ}
\tau_{\eps}(x) = \frac{1}{2a \sqrt{\eps}} F'(X) = \frac{1}{2a \sqrt{\eps}} \sT_{(1)}\big( F'(X) \big) + \Psi_{\eps}(x) =: \tau_{\eps}^{(1)}(x) + \tau_{\eps}^{(2)}(x), 
\end{equ}
where $\sT_{(1)}\big( F'(X) \big)$ is $F'(X)$ with the first chaos removed. Since $\Psi_{\eps} \rightarrow \Psi$ in probability in $\cC^{-\frac{1}{2}-\kappa}$, it suffices to check $\tau_{\eps}^{(1)}$ vanishes in the same topology as $\eps \rightarrow 0$. 

Following the notations in Section~\ref{sec:general_bound}, we let $\Phi_{\eps}(\theta,x) = \sin(\theta X)$. The dependence of $\Phi$ on $\eps$ is via $X = \eps^{\frac{1}{2}} \Psi_{\eps}(x)$. Also recall that for every test function $\varphi$ on $\R^{+} \times \T$ and $\lambda>0$, we write $\varphi^{\lambda}(x) = \lambda^{-3} \varphi(x/\lambda)$. For every function $\fF: \R \times (\R^{+} \times \T) \mapsto \R$, let
\begin{equ}
(\aA_{\eps,\lambda}^{\<1'>} \fF)(\theta) := \frac{1}{2a \sqrt{\eps}} \int_{\R^{+} \times \T} \fF(\theta,x) \varphi^{\lambda}(x) dx, 
\end{equ}
where we have omitted in notation the dependence of $\aA$ on $\varphi$ for simplicity. We split $F'$ into a regular part $F_{\delta}'$ and a small remainder $F'-F_{\delta}'$. Since $F'$ and $F_{\delta}'$ are both odd, by the definition of the Fourier transform in \eqref{e:FT}, we have
\begin{equ} \label{e:decomp_1'}
\scal{\tau_{\eps}^{(1)}, \varphi^{\lambda}} = i \scal{\widehat{F_{\delta}'}, \aA_{\eps,\lambda}^{\<1'>} \sT_{(1)}(\Phi_{\eps})}_{\theta} + i \scal{\widehat{F'} - \widehat{F_{\delta}'}, \aA_{\eps,\lambda}^{\<1'>} \sT_{(1)}(\Phi_{\eps})}_{\theta}, 
\end{equ}
where we have used Fubini to change the order of integration on the right hand side, and the notation $\scal{\cdot,\cdot}_{\theta}$ refers to integration in the $\theta$ variable. We deal with the two terms in \eqref{e:decomp_1'} separately. For the first one, by Proposition~\ref{pr:localisation}, we have
\begin{equ} \label{e:1'_smooth}
|\scal{\widehat{F_{\delta}'}, \aA_{\eps,\lambda}^{\<1'>} \sT_{(1)}(\Phi_{\eps})}_{\theta}| \lesssim \sum_{K \in \Z} \|\widehat{F_{\delta}'}\|_{M+2, \fR_K} \cdot \sup_{r \leq M+2} \sup_{\theta \in \fR_K}  |\big(\aA_{\eps,\lambda}^{\<1'>} \sT_{(1)}(\Phi_{\eps}) \big)^{(r)}(\theta)|. 
\end{equ}
Taking the $2n$th moment on both sides and using Lemma~\ref{le:interchange} to interchange the supremum and expectation for each term on the right hand side, we get
\begin{equ}
\|\scal{\widehat{F_{\delta}'}, \aA_{\eps,\lambda}^{\<1'>} \sT_{(1)}(\Phi_{\eps})}_{\theta}\|_{2n} \lesssim \sum_{K \in \Z} \|\widehat{F_{\delta}'}\|_{M+2, \fR_K} \cdot \sup_{r \leq M+3} \sup_{\theta \in \fR_K}  \|\bigl(\aA_{\eps,\lambda}^{\<1'>} \sT_{(1)}(\d_{\theta}^{r} \Phi_{\eps}) \bigr)(\theta)\|_{2n}, 
\end{equ}
where we have used the notation $\|\cdot\|_{2n} = \big( \E |\cdot|^{2n} \big)^{\frac{1}{2n}}$. The maximum in the number of derivatives is now taken over the range $r \leq M+3$ since one pays one more derivative from Lemma~\ref{le:interchange}. The term inside the supremum above, when raised to the $2n$-th power, has the expression
\begin{equ} \label{e:1'_smooth_corr}
\E \big|\bigl(\aA_{\eps,\lambda}^{\<1'>} \sT_{(1)}(\d_{\theta}^{r} \Phi_{\eps})\bigr)(\theta)\big|^{2n} = (4a^2 \eps)^{-n} \int \Bvarphi^{\lambda}(\x) \bigg[ \E \prod_{k=1}^{2n}  \sT_{(1)} \Big( \d_{\theta}^{r} \Phi_{\eps}(\theta,x_k) \Big) \bigg] d \x. 
\end{equ}
Applying Theorem~\ref{th:general_bound} to the object in the bracket above, we get the bound
\begin{equ} \label{e:1'_smooth_corr_bound}
\E \big|\bigl(\aA_{\eps,\lambda}^{\<1'>} \sT_{(1)}(\d_{\theta}^{r} \Phi_{\eps})\bigr)(\theta)\big|^{2n} \lesssim \eps^{-n} (1+|\theta|)^{N} \int |\Bvarphi^{\lambda}(\x)| \Big( \E \prod_{k=1}^{2n} P_{N}(X_k) \Big) d \x
\end{equ}
for some $N \geq 1$, where $P_N(X_k) = \sum_{j=1}^{N} X_{k}^{\diamond (2j+1)}$. Since $\eps^{-\frac{1}{2}} X_{k}^{\diamond (2j+1)} = \eps^{j} \Psi_{\eps}^{\diamond (2j+1)}(x_k)$ and the sum in $P_N$ starts from the third order term, we see that the right hand side is a linear combination of the $2n$-moments of the quantities $\scal{\eps^{j} \Psi_{\eps}^{\diamond (2j+1)}, |\varphi^{\lambda}|}$ with $j = 1, \dots, N$. These are higher order version of $\scal{\Psi_{\eps}, |\varphi^{\lambda}|}$ with additional Wick powers in $\Psi_{\eps}$ and $\eps$ balanced, so we have the bound
\begin{equ}
\|\bigl(\aA_{\eps,\lambda}^{\<1'>} \sT_{(1)}(\d_{\theta}^{r} \Phi_{\eps})\bigr)(\theta)\|_{2n} \lesssim \eps^{\kappa'} \lambda^{-\frac{1}{2}-\kappa'} (1+|\theta|)^{N}
\end{equ}
for some $N$ and all $\kappa' < \kappa$. Note that here we have relaxed the upper bound by replacing the exponent $\frac{N}{2n}$ with $N$, but this would not affect our result. Plugging the above bound into \eqref{e:1'_smooth} and applying \eqref{e:F_smooth} to $\|\widehat{F_{\delta}'}\|_{M+2;\fR_{K}}$ with $N$ replaced by $2N$, we get the bound
\begin{equ} \label{e:1'_smooth_bound}
\|\scal{\widehat{F_{\delta}'}, \aA_{\eps,\lambda}^{\<1'>} \sT_{(1)}(\Phi_{\eps})}_{\theta}\|_{2n} \lesssim \eps^{\kappa'} \delta^{-2N} \lambda^{-\frac{1}{2} - \kappa'}. 
\end{equ}
We now turn to the second term in \eqref{e:decomp_1'}. Write
\begin{equ} \label{e:1'_remainder_separate}
\big(\aA_{\eps,\lambda}^{\<1'>} \sT_{(1)} (\Phi_{\eps})\big)(\theta) = (\aA_{\eps,\lambda}^{\<1'>} \Phi_{\eps})(\theta) - \frac{\theta}{2a} e^{-\frac{\theta^2 \sigma^2}{2}} \int \Psi_{\eps}(x) \varphi^{\lambda}(x) dx, 
\end{equ}
where $\sigma^2 = \E X^2$, and we estimate the action of $\widehat{F'} - \widehat{F_{\delta}'}$ on these two terms separately. For the first one, similar as before, we have
\begin{equ} \label{e:1'_remainder}
\| \scal{\widehat{F'} - \widehat{F_{\delta}'}, \aA_{\eps,\lambda}^{\<1'>} \Phi_{\eps}}_{\theta} \|_{2n} \lesssim \sum_{K \in \Z} \|\widehat{F'} - \widehat{F_{\delta}'}\|_{M+2, \fR_K} \sup_{r\leq M+3} \sup_{\theta \in \fR_K} \|(\aA_{\eps,\lambda}^{\<1'>} \Phi_{\eps})^{(r)}(\theta)\|_{2n}. 
\end{equ}
The quantity $\|\cdot\|_{2n}$ inside the supremum above (raised to the $2n$-th power) has the expression
\begin{equ}
\E |(\aA_{\eps,\lambda}^{\<1'>} \Phi_{\eps})^{(r)}(\theta)|^{2n} = (4a^2 \eps)^{-n} \int \Bvarphi^{\lambda}(\x) \bigg[ \E \prod_{k=1}^{2n} \d_{\theta}^{r} \Phi_{\eps}(\theta,x_k) \bigg] d\x. 
\end{equ}
The difference between here and \eqref{e:1'_smooth_corr} is that the first chaos component of $\d_{\theta}^{r} \Phi_{\eps}$ is not removed. We can then apply the bound \eqref{e:special2} to the expression inside the bracket above with $\tilde{\tT} = \tilde{\oO} = \{1, \dots, 2n\}$ so that
\begin{equ}
\Big| \E \prod_{k=1}^{2n} \d_{\theta}^{r} \Phi_{\eps}(\theta,x_k) \Big| \lesssim (1+|\theta|)^{2n} \Big( \E \prod_{k=1}^{2n} X_k \Big). 
\end{equ}
Here, the power $2n$ on $(1+|\theta|)$ is precisely $|\tilde{\oO}|$ in \eqref{e:special1}. This immediately gives
\begin{equ}
\E |(\aA_{\eps,\lambda}^{\<1'>} \Phi_{\eps})^{(r)}(\theta)|^{2n} \lesssim \eps^{-n} (1+|\theta|)^{2n} \int |\Bvarphi^{\lambda}(\x)| \Big( \E \prod_{k=1}^{2n} X_k \Big) d\x. 
\end{equ}
Note that the right hand side (without $\theta$) is precisely the $2n$-th moment of $\int \Psi_{\eps}(x) |\varphi^{\lambda}(x)| dx$, so we have the bound
\begin{equ}
\|(\aA_{\eps,\lambda}^{\<1'>} \Phi_{\eps})^{(r)}(\theta)\|_{2n} \lesssim (1+|\theta|) \|\scal{\Psi_{\eps}, |\varphi^{\lambda}|}\|_{2n} \lesssim (1+|\theta|) \lambda^{-\frac{1}{2}}. 
\end{equ}
By \eqref{e:F_remainder}, we have
\begin{equ}
\|\widehat{F'}-\widehat{F_{\delta}'}\|_{M+2,\fR_K} \lesssim \delta^{\beta} (1+|K|)^{-6+\beta}
\end{equ}
for every $\beta \in (0,1)$. Plugging this, together with the above bound for $\|(\aA_{\eps,\lambda}^{\<1'>} \Phi_{\eps})^{(r)}(\theta)\|_{2n}$, back into \eqref{e:1'_remainder}, we obtain
\begin{equ}
\| \scal{\widehat{F'} - \widehat{F_{\delta}'}, \aA_{\eps,\lambda}^{\<1'>} \Phi_{\eps}}_{\theta} \|_{2n} \lesssim \delta^{\beta} \lambda^{-\frac{1}{2}}
\end{equ}
for some $\beta \in (0,1)$. The same bound holds for the second term in \eqref{e:1'_remainder_separate} but the procedure is simpler. Thus, the remainder part of $\tau_{\eps}^{(1)}$ satisfies the bound
\begin{equ} \label{e:1'_remainder_bound}
\| \scal{ \widehat{F'}-\widehat{F_{\delta}'}, \aA_{\eps,\lambda}^{\<1'>} \sT_{(1)}(\Phi_{\eps})}_{\theta} \|_{2n} \lesssim \delta^{\beta} \lambda^{-\frac{1}{2}}. 
\end{equ}
Now, choosing $\delta = \eps^{\frac{\kappa'}{4N}}$, applying the two bounds \eqref{e:1'_smooth_bound} and \eqref{e:1'_remainder_bound} back to \eqref{e:decomp_1'}, and recalling that $\kappa'$ can be arbitrarily small, we have thus proved \eqref{e:main_model_bound} for $\tau = \<1'>$. 

\subsubsection[Term Ladder]{The case \texorpdfstring{$\tau = \<2'1'1'>$}{Ladder}}

We now turn to the case $\tau = \<2'1'1'>$. Again, we write $\tau_{\eps} = \hPi^{\eps}_{0} \tau$. By \eqref{e:model_new}, \eqref{e:model_2} and \eqref{e:renorm_const}, we have
\begin{equs}
\tau_{\eps}(x) = &\frac{1}{4 a^3 \eps^2} \int \big( K(x-y)-K(-y) \big) K(y-z) F'(X) F'(Y) \cent{F(Z)} \,dy\,dz\\
&- \frac{1}{4 a^3 \eps^2} \int K(x-y) K(y-z) \E \big( F'(X) F'(Y) \cent{F(Z)} \big) \,dy\,dz. 
\end{equs}
For simplicity, we write $K(x,y) = K(x-y) - K(-y)$. Following the notations in Section~\ref{sec:general_bound}, we let $\tT = \{\fx, \fy, \fz\}$ be the type space with categories $\oO = \{\fx, \fy\}$ and $\eE = \{\fz\}$. These reflect the roles of the fields $X$, $Y$ and $Z$ appearing in $\tau_{\eps}$. 

We write $\Btheta = (\theta_{\fx}, \theta_{\fy}, \theta_{\fz}) \in \R^{3}$. For each multi-index $\m = (m_{\fx}, m_{\fy}, m_{\fz}) \in \N^\tT$, let $C_{\m}$ denote the coefficient of the term $X^{\diamond m_{\fx}} \diamond Y^{\diamond m_{\fy}} \diamond Z^{\diamond m_{\fz}}$ in the chaos expansion of $\hH(X,Y,Z) \eqdef F'(X) F'(Y) \cent{F(Z)}$. Note that these coefficients do depend on $(x,y,z)$ and are given by the formula
\begin{equ}[e:defCM]
C_{\m}(x,y,z) = {1\over \m!} \E (\d^\m \hH)(X,Y,Z)\;.
\end{equ}
Finally, let $\mM \subset \N^\tT$ be
\begin{equ}
\mM = \big\{ (0,0,0), (1,0,1), (0,1,1), (0,0,2), (1,1,2) \big\}, 
\end{equ}
and let $\sT_{\mM}(\hH(X,Y,Z))$ be the chaos expansion of $\hH(X,Y,Z)$ with the components in $\mM$\protect\footnote{Components in $\mM$ are terms of the form $X^{m_{\fx}} \diamond Y^{m_{\fy}} \diamond Z^{m_{\fz}}$ with $\m = (m_{\fx}, m_{\fy}, m_{\fz}) \in \mM$.} removed. With these notations, we now decompose $\tau_{\eps}$ into
\begin{equ}
\tau_{\eps}(x) = \sum_{j=1}^{6} \tau_{\eps}^{(j)}(x), 
\end{equ}
where the six terms are given by
\begin{equs}
\tau_{\eps}^{(1)}(x) &= \frac{1}{4a^3 \eps^2} \int K(x,y) K(y-z) \sT_{\mM} \big(F'(X) F'(Y) \cent{F(Z)} \big) \; \,dy\,dz, \\
\tau_{\eps}^{(2)}(x) &= \frac{1}{4a^3 \eps^2} \int K(x,y) K(y-z) \cdot C_{1,0,1}(x,y,z) \; X \diamond Z \; \,dy\,dz, \\
\tau_{\eps}^{(3)}(x) &= \frac{1}{4a^3 \eps^2} \int K(x,y) K(y-z) \cdot C_{0,1,1}(x,y,z) \; Y \diamond Z \; \,dy\,dz, \\
\tau_{\eps}^{(4)}(x) &= \frac{1}{4a^3 \eps^2} \int K(x,y) K(y-z) \cdot C_{0,0,2}(x,y,z) \; Z^{\diamond 2} \; \,dy\,dz, \\
\tau_{\eps}^{(5)}(x) &= \frac{1}{4a^3 \eps^2} \int K(x,y) K(y-z) \cdot C_{1,1,2}(x,y,z) \; X \diamond Y \diamond Z^{\diamond 2} \; \,dy\,dz, \\
\tau_{\eps}^{(6)}(x) &= - \frac{1}{4a^3 \eps^2} \int K(-y) K(y-z) \cdot C_{0,0,0}(x,y,z) \; \,dy\,dz, 
\end{equs}
with the constants $C_\m$ given by \eqref{e:defCM}. 

In what follows, we will prove the convergence of each $\tau_{\eps}^{(j)}$ to their corresponding limit, which altogether give the limiting KPZ model for \<2'1'1'>. We start with $\tau_{\eps}^{(1)}$. For $\Btheta = (\theta_{\fx}, \theta_{\fy}, \theta_{\fz})$, let
\begin{equ}
\Phi_{\eps}(\Btheta,x,y,z) = \sin(\theta_{\fx}X) \sin(\theta_{\fy}Y) \cent{\cos(\theta_{\fz}Z)}. 
\end{equ}
We also define the operator $\aA_{\eps,\lambda}^{\tau}$ for $\tau = \<2'1'1'>\;$ by
\begin{equ}
(\aA_{\eps,\lambda}^{\tau} \fF)(\Btheta) = \int \varphi^{\lambda}(x) K(x,y) K(y-z) \fF(\Btheta,x,y,z) dz dy dx. 
\end{equ}
Here and below, we omit the symbol $\tau$ in $\aA$ for simplicity. With these notations, we have the expression
\begin{equ} \label{e:decomp_tau1}
\scal{\tau_{\eps}^{(1)}, \varphi^{\lambda}} = \scal{\Ups_{\delta}, \aA_{\eps,\lambda} \sT_{\mM}(\Phi_{\eps})}_{\Btheta} + \scal{\Ups-\Ups_{\delta}, \aA_{\eps,\lambda} \sT_{\mM}(\Phi_{\eps})}_{\Btheta}, 
\end{equ}
where $\Ups = \widehat{F'} \otimes \widehat{F'} \otimes \hF$, $\Ups_{\delta}$ is the distribution that replaces every appearance of $F^{(\ell)}$ in $\Ups$ by $F_{\delta}^{(\ell)}$, and $\sT_{\mM}(\Phi_{\eps})$ is the chaos expansion of $\Phi_{\eps}$ in terms of $X$, $Y$ and $Z$ with components in $\mM$ removed. For the first term in \eqref{e:decomp_tau1}, similar as before, we have
\begin{equ} \label{e:tau1_smooth_sum}
\|\scal{\Ups_{\delta}, \aA_{\eps,\lambda} \sT_{\mM}(\Phi_{\eps})}_{\Btheta}\|_{2n} \lesssim \sum_{\K \in \Z^3} \|\Ups_{\delta}\|_{M+2,\fR_{\K}} \sup_{\r: |\r|_{\infty} \leq M+3} \sup_{\Btheta \in \fR_{\K}} \| \big( \aA_{\eps,\lambda} \sT_{\mM} ( \d_{\Btheta}^{\r} \Phi_{\eps} ) \big)(\Btheta)\|_{2n}. 
\end{equ}
We first control $\|\big( \aA_{\eps,\lambda} \sT_{\mM} ( \d_{\Btheta}^{\r} \Phi_{\eps} ) \big)(\Btheta)\|_{2n}$. Raising it to the $2n$-th power, we have
\begin{equs}
\E |\big(\aA_{\eps,\lambda} \sT_{\mM} ( \d_{\Btheta}^{\r} \Phi_{\eps}) \big)(\Btheta) |^{2n} = &\int \Bvarphi^{\lambda}(\x) \prod_{k=1}^{2n} \big( K(x_k,y_k) \cdot K(y_k-z_k) \big)\\
&\cdot \Big[ \E \prod_{k=1}^{2n} \sT_{\mM} \big( \d_{\Btheta}^{\r} \Phi_{\eps}(\theta,x_k,y_k,z_k) \big) \Big] d \z d \y d \x. 
\end{equs}
It is straightforward to check that the set $\mM$ satisfies $\bB(\n) \cap \mM = \empty$ for all $\n \in \mM^{c}$, so we apply Theorem~\ref{th:general_bound} to the expectation above to get
\begin{equs}
\E | \big(\aA_{\eps,\lambda} \sT_{\mM} \big( \d_{\Btheta}^{\r} \Phi_{\eps}\big) \big)(\Btheta)|^{2n} \lesssim &(1+\theta)^{N} \int |\Bvarphi^{\lambda}(\x)| \prod_{k=1}^{2n} \Big( |K(x_k,y_k) K(y_k-z_k)| \Big)\\
&\Big[ \E \prod_{k=1}^{2n} \sT_{\mM} \big( P_{N}(X_k) P_{N}(Y_k) Q_N(Z_k) \big) \Big] d \z d \y d \x, 
\end{equs}
which holds for some $N \geq 2$, and $P_{N}(X) = \sum_{j=1}^{N} X^{\diamond (2j-1)}$, $Q_{N}(Z) = \sum_{j=1}^{N} Z^{\diamond (2j)}$. It is also straightforward to check that the right hand side above corresponds to a higher order version of the object
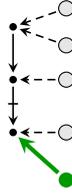
\begin{equ}
\begin{tikzpicture}[scale=0.35,baseline=0.8cm]
\node at (0,-0.8)  [root] (root) {};
\node at (-2,1)  [dot] (left) {};
\node at (-2,3)  [dot] (left1) {};
\node at (-2,5)  [dot] (left2) {};
\node at (0,1) [var] (variable1) {};
\node at (0,3) [var] (variable2) {};
\node at (0,4.3) [var] (variable3) {};
\node at (0,5.7) [var] (variable4) {};
\draw[testfcn] (left) to (root);
\draw[kernel1] (left1) to (left);
\draw[kernel] (left2) to (left1);
\draw[kepsilon] (variable2) to (left1); 
\draw[kepsilon] (variable1) to (left); 
\draw[kepsilon] (variable3) to (left2); 
\draw[kepsilon] (variable4) to (left2);
\end{tikzpicture}\;
\end{equ}
in the sense that every additional two Wick powers of any of these variables are accompanied by one power of $\eps$. It is also strict so that there is at least one additional power of $\eps$. This then implies deduce the bound
\begin{equ}
\|\big( \aA_{\eps,\lambda} \sT_{\mM} ( \d_{\Btheta}^{\r} \Phi_{\eps} ) \big)(\Btheta)\|_{2n} \lesssim (1+\theta)^{N} \cdot \eps^{\kappa'} \lambda^{-\kappa'}
\end{equ}
for all $\kappa' < \kappa$. Plugging this bound into \eqref{e:tau1_smooth_sum} and employing \eqref{e:F_smooth} to control $\|\Ups_{\delta}\|_{M+2,\fR_{\K}}$, we then obtain
\begin{equ} \label{e:tau1_smooth_bound}
\|\scal{\Ups_{\delta}, \aA_{\eps,\lambda} \sT_{\mM}(\Phi_{\eps})}_{\Btheta}\|_{2n} \lesssim \eps^{\kappa'} \delta^{-2N} \lambda^{-\kappa'}. 
\end{equ}
As for the second term in \eqref{e:decomp_tau1}, we write
\begin{equ}
\sT_{\mM}(\Phi_{\eps}) = \Phi_{\eps} - \sum_{\m \in \mM} C_{\m} \cdot X^{\diamond m_{\fx}} \diamond Y^{\diamond m_{\fy}} \diamond Z^{\diamond m_{\fz}}, \quad \m = (m_{\fx}, m_{\fy}, m_{\fz}), 
\end{equ}
and estimate each term above separately. For the action of $\aA_{\eps,\lambda}$ on $\Phi_{\eps}$, we have
\begin{equs}
\E | \big( \aA_{\eps,\lambda} \d_{\Btheta}^{\r} \Phi_{\eps} \big)(\Btheta) |^{2n} = &\int \Bvarphi^{\lambda}(\x) \prod_{k=1}^{2n} \Big( K(x_k,y_k) K(y_k - z_k)  \Big)\\
&\cdot \Big[ \E \prod_{k=1}^{2n} (\d_{\Btheta}^{\r} \Phi_{\eps})(\Btheta,x_k,y_k,z_k) \Big]  d \z d \y d \x.  
\end{equs}
Now, let $\tilde{\tT} = \tT \times \{1, \dots, 2n\}$ where the types of the elements in $\tilde{\tT}$ are the same as their projections onto $\tT$. We then apply \eqref{e:special1} to the right hand side above to get
\begin{equs} [e.1]
\E | \big( \aA_{\eps,\lambda} \d_{\Btheta}^{\r} \Phi_{\eps} \big)(\Btheta) |^{2n} \lesssim \theta^{8n} &\int \big|\Bvarphi^{\lambda}(\x) \big| \prod_{k=1}^{2n} \Big| K(x_k,y_k) K(y_k - z_k) \Big|\\
&\cdot \Big[ \E \prod_{k=1}^{2n} \big( X_k Y_k Z_k^{\diamond 2} \big) \Big] d \z d \y d \x, 
\end{equs}
where $\theta = 1 + |\Btheta|$. Here, the power of $\theta$ is $8n$ since there are $4n$ points in $\tilde{\oO}$ (each contributing $1$ power) and $2n$ points in $\tilde{\eE}$ (each contributing $2$). Note that the right hand side (without $\theta$) is precisely the  $2n$-th moment of the quantity $\scal{\Pi_{0}^{\KPZ;\eps}\<2'1'1'>, \varphi_{0}^{\lambda}}$ without logarithmic renormalisation (that is, one sets $C_{2}^{(\eps)}=0$ in the model), and hence it is bounded by $|\log \eps|^{n} \lambda^{-2n \kappa'}$. Thus, we deduce the bound
\begin{equ}
\| \big( \aA_{\eps,\lambda} \d_{\Btheta}^{\r} \Phi_{\eps} \big)(\Btheta) \|_{2n} \lesssim (1+|\Btheta|)^{4} |\log \eps| \lambda^{- \kappa'}. 
\end{equ}
A similar bound, both in terms of the powers of $\Btheta$ and of the powers of $\eps$ and $\lambda$, holds for the action of $\aA_{\eps,\lambda}$ on $C_{\m}(\Btheta,x,y,z) X^{\diamond m_{\fx}} \diamond Y^{\diamond m_{\fy}} \diamond Z^{\diamond m_{\fz}}$. Since \eqref{e:F_remainder} implies
\begin{equ}
\|\Ups-\Ups_{\delta}\|_{M+2;\fR_{\K}} \lesssim \delta^{\beta} (1+|K_{\fx}|) (1+|K_{\fy}|) (1+|\K|)^{-7-\alpha+\beta}, 
\end{equ}
we then have
\begin{equ} \label{e:tau1_remainder_bound}
\|\scal{\Ups-\Ups_{\delta}, \aA_{\eps,\lambda} \sT_{\mM}(\Phi_{\eps})}_{\Btheta}\|_{2n} \lesssim \delta^{\beta} |\log \eps| \lambda^{-\kappa'}. 
\end{equ}
Now, combining \eqref{e:tau1_smooth_bound} and \eqref{e:tau1_remainder_bound}, taking $\delta = \eps^{\frac{\kappa'}{4N}}$, and noting that both $\kappa'$ and $\beta \in (0,1)$ can be arbitrarily small, we can then deduce that there exists $\zeta > 0$ such that
\begin{equ}
\|\scal{\tau_{\eps}^{(1)}, \varphi^{\lambda}}\|_{2n} \lesssim \eps^{\zeta} \lambda^{|\tau|+\zeta}. 
\end{equ}
We now turn to $\tau_{\eps}^{(2)}$. Let
\begin{equ}
C_{1,0,1}^{(1)} = \E \big( \cent{F''(X)} F'(Y) F'(Z) \big), \qquad C_{1,0,1}^{(2)} = 2a \E \big( F'(Y) F'(Z)-4a^2 YZ \big), 
\end{equ}
and $C_{1,0,1}^{(3)} = 8a^3 \E YZ$. We then write $\tau_{\eps}^{(2)} = \sum_{j=1}^{3} \tau_{\eps}^{(2;j)}$, where
\begin{equ}
\tau_{\eps}^{(2;j)}(x) = \frac{1}{4a^3 \eps} \int K(x,y) K(y-z) \cdot C_{1,0,1}^{(j)}(x,y,z) \Psi_{\eps}(x) \diamond \Psi_{\eps}(z) \,dy\,dz. 
\end{equ}
There is only $\eps^{-1}$ left since the other negative power is combined with $X \diamond Z$ so that we have $\Psi_{\eps}(x) \diamond \Psi_{\eps}(z)$ in the expression. We analyse the three terms separately. For $\tau_{\eps}^{(2;3)}$, we have
\begin{equ} \label{e:tau2_main}
\scal{\tau_{\eps}^{(2;3)}, \varphi_{0}^{\lambda}} = 2\;
\begin{tikzpicture}[scale=0.35,baseline=0.8cm]
\node at (0,-0.8)  [root] (root) {};
\node at (-2,1)  [dot] (left) {};
\node at (-2,3)  [dot] (left1) {};
\node at (-2,5)  [dot] (left2) {};
\node at (0,1) [var] (variable1) {};
\node at (0,5.7) [var] (variable4) {};
\node at (0,4) [dot] (right) {}; 
\draw[testfcn] (left) to (root);
\draw[kernel1] (left1) to (left);
\draw[kernel] (left2) to (left1);
\draw[kepsilon] (variable1) to (left); 
\draw[kepsilon] (variable4) to (left2);
\draw[kepsilon] (right) to (left2); 
\draw[kepsilon] (right) to (left1); 
\end{tikzpicture}\;
= 2\;
\begin{tikzpicture}[scale=0.35,baseline=0.8cm]
\node at (0,-0.8)  [root] (root) {};
\node at (-2,1)  [dot] (left) {};
\node at (-2,3)  [dot] (left1) {};
\node at (-2,5)  [dot] (left2) {};
\node at (0,1) [var] (variable1) {};
\node at (0,5.7) [var] (variable4) {};
\draw[testfcn] (left) to (root);
\draw[kernel1] (left1) to (left);
\draw[kernelBig] (left2) to (left1);
\draw[kepsilon] (variable1) to (left); 
\draw[kepsilon] (variable4) to (left2);
\end{tikzpicture}\;, 
\end{equ}
where we have identified
\begin{equ}
\begin{tikzpicture}[scale=0.35,baseline=-0.2cm]
\node at (-2,-1)  [dot] (left1) {};
\node at (-2,1)  [dot] (left2) {};
\node at (0,0) [dot] (right) {}; 
\draw[kernel] (left2) to (left1);
\draw[kepsilon] (right) to (left2); 
\draw[kepsilon] (right) to (left1); 
\end{tikzpicture}\;
= \; 
\begin{tikzpicture}[scale=0.35,baseline=-0.2cm]
\node at (-2,-1)  [dot] (left1) {};
\node at (-2,1)  [dot] (left2) {};
\draw[kernelBig] (left2) to (left1);. 
\end{tikzpicture}
\end{equ}
This is because the correlation function $\E \bigl( \Psi_{\eps}(x) \Psi_{\eps}(y)\bigr)$, represented by the two dashed arrows above, is symmetric under the reflection of its space variable with respect to the origin. Since the kernel $K$ is anti-symmetric in the space variable, the whole kernel on the left hand side above is anti-symmetric, and hence we can identify it with its renormalised version. The right hand side of \eqref{e:tau2_main} is precisely one of the terms in $\scal{\Pi_{0}^{\KPZ;\eps}\<2'1'1'>, \varphi_{0}^{\lambda}}$ in the Appendix, so it does converge to the desired limit with the error bound \eqref{e:main_model_bound}. For $\scal{\tau_{\eps}^{(2;1)}, \varphi_{0}^{\lambda}}$, according to the bound \eqref{e:101_1}, it can be decomposed into two parts. The first one, corresponding to the first term on the right hand side of \eqref{e:101_1}, can be represented by
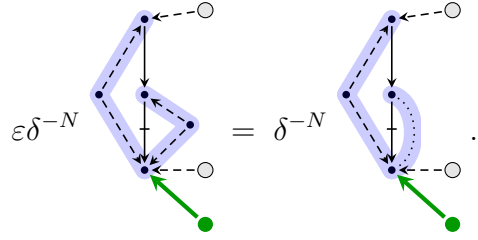
\begin{equ}
	\eps\delta^{-N} \;
	\begin{tikzpicture}[scale=0.4,baseline=0.8cm]
	\node at (0,-0.8)  [root] (root) {};
	\node at (-2,1)  [dot] (left) {};
	\node at (-2,3.5)  [dot] (left1) {};
	\node at (-2,6)  [dot] (left2) {};
	\node at (-3.5,3.5)  [dot] (left3) {}; 
	\node at (-0.5,2.5)  [dot] (right) {}; 
	\node at (0,1) [var] (variable1) {};
	\node at (0,6.3) [var] (variable4) {};
	\draw[highlight] (left2.center) -- (left3.center) -- (left.center) -- (right.center) -- (left1.center);
	\draw[testfcn] (left) to (root);
	\draw[kernel1] (left1) to (left);
	\draw[kernel] (left2) to (left1); 
	\draw[kepsilon] (right) to (left); 
	\draw[kepsilon] (right) to (left1); 
	\draw[kepsilon] (left3) to (left2); 
	\draw[kepsilon] (left3) to (left); 
	\draw[kepsilon] (variable1) to (left); 
	\draw[kepsilon] (variable4) to (left2);
	\end{tikzpicture}
	\;=\;
	\delta^{-N} \;
	\begin{tikzpicture}[scale=0.4,baseline=0.8cm]
	\node at (0,-0.8)  [root] (root) {};
	\node at (-2,1)  [dot] (left) {};
	\node at (-2,3.5)  [dot] (left1) {};
	\node at (-2,6)  [dot] (left2) {};
	\node at (-3.5,3.5)  [dot] (left3) {}; 
	\node at (0,1) [var] (variable1) {};
	\node at (0,6.3) [var] (variable4) {};
	\begin{scope}[transparency group,opacity=0.2]
	\draw[cover] (left2.center) -- (left3.center) -- (left.center);
	\draw[cover,bend right=75] (left.center) to (left1.center);
	\end{scope}
	\draw[testfcn] (left) to (root);
	\draw[kernel1] (left1) to (left);
	\draw[gepsilon, bend left = 60] (left1) to (left);
	\draw[kernel] (left2) to (left1); 
	\draw[kepsilon] (left3) to (left2); 
	\draw[kepsilon] (left3) to (left); 
	\draw[kepsilon] (variable1) to (left); 
	\draw[kepsilon] (variable4) to (left2);
	\end{tikzpicture}\;. 
\end{equ}
Here, the highlights on the edges mean that they are upper bounds for the corresponding parts in the object itself. By the bounds in \cite[Sec.~6.2.6]{HQ}, the $L^{2n}$-th moment of the above graph is controlled by $\eps^{\kappa'} \delta^{-N} \lambda^{-2\kappa'}$. The second one, corresponding to the second term on the right hand side of \eqref{e:101_1}, carries a logarithmic divergence but multiplied by $\delta^{\beta}$. Altogether, we have the bound 
\begin{equ}
\|\scal{\tau_{\eps}^{(2;1)}, \varphi_{0}^{\lambda}}\|_{2n} \lesssim (\eps^{\kappa'} \delta^{-N} + \delta^{\beta} |\log \eps|) \lambda^{- 2\kappa'}. 
\end{equ}
We then choose $\delta = \eps^{\frac{\kappa'}{4N}}$ so that $\|\scal{\tau_{\eps}^{(2;1)}, \varphi_{0}^{\lambda}}\|_{2n} \lesssim \eps^{\kappa'} \lambda^{-2\kappa'}$. Finally, for $\tau_{\eps}^{(2;2)}$, we note that the kernel $K(y-z) \E \big( F'(Y) F'(Z) - 4a^2 YZ \big)$ is anti-symmetric in the space variable, and hence can be identified with its renormalised kernel. Hence, it can be controlled by the graph
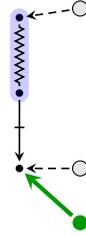
\begin{equ}
\begin{tikzpicture}[scale=0.4,baseline=0.8cm]
\node at (0,-0.8)  [root] (root) {};
\node at (-2,1)  [dot] (left) {};
\node at (-2,3.5)  [dot] (left1) {};
\node at (-2,6)  [dot] (left2) {};
\node at (0,1) [var] (variable1) {};
\node at (0,6.3) [var] (variable4) {};
\begin{scope}[transparency group,opacity=0.2]
\draw[cover] (left1.center) -- (left2.center);
\end{scope}
\draw[testfcn] (left) to (root);
\draw[kernel1] (left1) to (left);
\draw[kernelBig] (left2) to (left1); 
\draw[kepsilon] (variable1) to (left); 
\draw[kepsilon] (variable4) to (left2);
\end{tikzpicture}\;. 
\end{equ}
Here, the highlight on the renormalised kernel means that we have the bound $\|\sR Q_{\eps}\|_{3+\beta'} \lesssim \eps^{\beta'}$ for some $\beta'>0$. This is the same $\beta'$ as in \eqref{e:101_2}. Hence, we have $\|\scal{\tau_{\eps}^{(2;2)}, \varphi_0^{\lambda}}\|_{2n} \lesssim \eps^{\kappa'} \lambda^{-2 \kappa'}$ as well. This completes the proof for $\tau_{\eps}^{(2)}$. 

The situations for $\tau_{\eps}^{(3)}$ and $\tau_{\eps}^{(4)}$ are similar. One can show that their main parts are given by
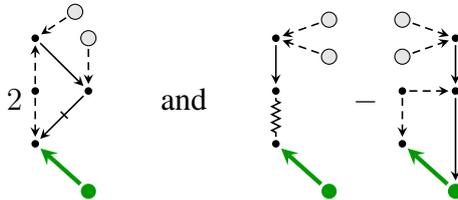
\begin{equ}
2\;
\begin{tikzpicture}[scale=0.35,baseline=0.8cm]
\node at (0,-0.8)  [root] (root) {};
\node at (-2,1)  [dot] (left) {};
\node at (0,3)  [dot] (left1) {};
\node at (-2,5)  [dot] (left2) {};
\node at (0,5) [var] (variable1) {};
\node at (-2,3) [dot] (variable3) {};
\node at (-0.5,6) [var] (variable4) {};
\draw[testfcn] (left) to (root);
\draw[kernel1] (left1) to (left);
\draw[kernel] (left2) to (left1);
\draw[kepsilon] (variable3) to (left); 
\draw[kepsilon] (variable1) to (left1); 
\draw[kepsilon] (variable3) to (left2); 
\draw[kepsilon] (variable4) to (left2);
\end{tikzpicture}
\qquad \text{and} \qquad
\begin{tikzpicture}[scale=0.35,baseline=0.8cm]
\node at (0,-0.8) [root] (root) {};
\node at (-2,1)  [dot] (left) {};
\node at (-2,3)  [dot] (left1) {};
\node at (-2,5)  [dot] (left2) {};
\node at (0,4.3) [var] (variable3) {};
\node at (0,5.7) [var] (variable4) {};
\draw[testfcn] (left) to (root);
\draw[kernelBig] (left1) to (left);
\draw[kernel] (left2) to (left1);
\draw[kepsilon] (variable3) to (left2); 
\draw[kepsilon] (variable4) to (left2);
\end{tikzpicture}
\;-\;
\begin{tikzpicture}[scale=0.35,baseline=0.8cm]
\node at (0,-0.8) [root] (root) {};
\node at (-2,1) [dot] (left) {};
\node at (0,3) [dot] (left1) {};
\node at (0,5) [dot] (left2) {};
\node at (-2,3) [dot] (variable1) {};
\node at (-2,4.3) [var] (variable3) {};
\node at (-2,5.7) [var] (variable4) {};
\draw[testfcn] (left) to (root);
\draw[kernel] (left1) to (root);
\draw[kernel] (left2) to (left1);
\draw[kepsilon] (variable1) to (left1); 
\draw[kepsilon] (variable1) to (left); 
\draw[kepsilon] (variable3) to (left2); 
\draw[kepsilon] (variable4) to (left2);
\end{tikzpicture}
\end{equ}
respectively, and their other parts vanish with the correct order just as in the case for $\tau_{\eps}^{(2;1)}$ and $\tau_{\eps}^{(2;2)}$. This implies that both $\tau_{\eps}^{(3)}$ and $\tau_{\eps}^{(4)}$ converge to the right limit. 

For $\tau_{\eps}^{(5)}$, its main part (when tested against $\varphi_{0}^{\lambda}$) is given by
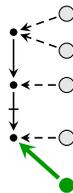
\begin{equ} \label{e:tau5_main}
\begin{tikzpicture}[scale=0.35,baseline=0.8cm]
\node at (0,-0.8)  [root] (root) {};
\node at (-2,1)  [dot] (left) {};
\node at (-2,3)  [dot] (left1) {};
\node at (-2,5)  [dot] (left2) {};
\node at (0,1) [var] (variable1) {};
\node at (0,3) [var] (variable2) {};
\node at (0,4.3) [var] (variable3) {};
\node at (0,5.7) [var] (variable4) {};
\draw[testfcn] (left) to (root);
\draw[kernel1] (left1) to (left);
\draw[kernel] (left2) to (left1);
\draw[kepsilon] (variable2) to (left1); 
\draw[kepsilon] (variable1) to (left); 
\draw[kepsilon] (variable3) to (left2); 
\draw[kepsilon] (variable4) to (left2);
\end{tikzpicture}
\end{equ}
which does converge to the right limit. By Lemma~\ref{le:112}, we see that the remainder can be decomposed into two parts. The first one can be controlled by the graph
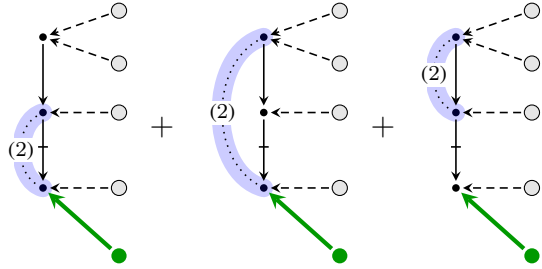
\begin{equ}
\begin{tikzpicture}[scale=0.5,baseline=1.2cm]
\node at (0,-0.8)  [root] (root) {};
\node at (-2,1)  [dot] (left) {};
\node at (-2,3)  [dot] (left1) {};
\node at (-2,5)  [dot] (left2) {};
\node at (0,1) [var] (variable1) {};
\node at (0,3) [var] (variable2) {};
\node at (0,4.3) [var] (variable3) {};
\node at (0,5.7) [var] (variable4) {};
\begin{scope}[transparency group,opacity=0.2]
\draw[cover,bend right=75] (left1.center) to (left.center);
\end{scope}
\draw[testfcn] (left) to (root);
\draw[kernel1] (left1) to (left);
\draw[kernel] (left2) to (left1);
\draw[kepsilon] (variable2) to (left1); 
\draw[kepsilon] (variable1) to (left); 
\draw[kepsilon] (variable3) to (left2); 
\draw[kepsilon] (variable4) to (left2);
\draw[gepsilon, bend right = 60] (left1) to node[labl]{\scriptsize $(2)$} (left);
\end{tikzpicture}
\phantom{1} + \phantom{1}
\begin{tikzpicture}[scale=0.5,baseline=1.2cm]
\node at (0,-0.8)  [root] (root) {};
\node at (-2,1)  [dot] (left) {};
\node at (-2,3)  [dot] (left1) {};
\node at (-2,5)  [dot] (left2) {};
\node at (0,1) [var] (variable1) {};
\node at (0,3) [var] (variable2) {};
\node at (0,4.3) [var] (variable3) {};
\node at (0,5.7) [var] (variable4) {};
\begin{scope}[transparency group,opacity=0.2]
\draw[cover,bend right=75] (left2.center) to (left.center);
\end{scope}
\draw[testfcn] (left) to (root);
\draw[kernel1] (left1) to (left);
\draw[kernel] (left2) to (left1);
\draw[kepsilon] (variable2) to (left1); 
\draw[kepsilon] (variable1) to (left); 
\draw[kepsilon] (variable3) to (left2); 
\draw[kepsilon] (variable4) to (left2);
\draw[gepsilon, bend right = 60] (left2) to node[labl]{\scriptsize $(2)$} (left);
\end{tikzpicture}
\phantom{1} + \phantom{1}
\begin{tikzpicture}[scale=0.5,baseline=1.2cm]
\node at (0,-0.8)  [root] (root) {};
\node at (-2,1)  [dot] (left) {};
\node at (-2,3)  [dot] (left1) {};
\node at (-2,5)  [dot] (left2) {};
\node at (0,1) [var] (variable1) {};
\node at (0,3) [var] (variable2) {};
\node at (0,4.3) [var] (variable3) {};
\node at (0,5.7) [var] (variable4) {};
\begin{scope}[transparency group,opacity=0.2]
\draw[cover,bend right=75] (left2.center) to (left1.center);
\end{scope}
\draw[testfcn] (left) to (root);
\draw[kernel1] (left1) to (left);
\draw[kernel] (left2) to (left1);
\draw[kepsilon] (variable2) to (left1); 
\draw[kepsilon] (variable1) to (left); 
\draw[kepsilon] (variable3) to (left2); 
\draw[kepsilon] (variable4) to (left2);
\draw[gepsilon, bend right = 60] (left2) to node[labl]{\scriptsize $(2)$} (left1);
\end{tikzpicture}
\;
\end{equ}
multiplied by $\delta^{-N}$. The other one from the remainder is the same as \eqref{e:tau5_main} but accompanied by a positive power of $\delta$. Hence, the $\|\cdot\|_{2n}$ norm of the remainder (after testing against $\varphi_0^{\lambda}$) is bounded by $\eps^{\kappa'} (\delta^{-N} + \delta^{\beta}) \lambda^{-2\kappa'}$. One can again choose $\delta$ to be a small positive power of $\eps$ so that the it vanishes in the correct order. 

Finally for $\tau_{\eps}^{(6)}$, according to Lemma~\ref{le:000}, we write
\begin{equ}
C_{0,0,0} = H^{(1)} + H^{(2)} + 4a^3 \E(XYZ^{\diamond 2}). 
\end{equ}
The part with $4a^3 \E(XYZ^{\diamond 2})$ is precisely
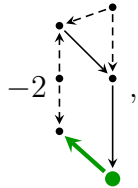
\begin{equ}
-2\;
\begin{tikzpicture}[scale=0.35,baseline=0.8cm]
\node at (0,-0.8)  [root] (root) {};
\node at (-2,1)  [dot] (left) {};
\node at (0,3)  [dot] (left1) {};
\node at (-2,5)  [dot] (left2) {};
\node at (-2,3) [dot] (variable3) {};
\node at (0,5.7) [dot] (variable4) {};
\draw[testfcn] (left) to (root);
\draw[kernel] (left1) to (root);
\draw[kernel] (left2) to (left1);
\draw[kepsilon] (variable3) to (left); 
\draw[kepsilon] (variable4) to (left1); 
\draw[kepsilon] (variable3) to (left2); 
\draw[kepsilon] (variable4) to (left2);
\end{tikzpicture}\;, 
\end{equ}
while the other two parts can be shown (with a similar argument as for $\tau_{\eps}^{(2)}$) to vanish with the correct order in $\lambda$. This finishes the proof for $\tau = \<2'1'1'>$.

\subsubsection[Extra Term]{The case \texorpdfstring{$\tau = \<2'2'0'>$}{extra}}
\label{sec:extra_term}

We now turn to the last symbol $\tau = \<2'2'0'>$. By \eqref{e:model_new}, \eqref{e:model_2} and \eqref{e:renorm_const}, we have
\begin{equ}
\tau_{\eps}(x) = \bar{\tau}_{\eps}(x) - \frac{C_{\<2'2'0s>}^{(\eps)}}{2a} \cdot \cent{F''(X)}, 
\end{equ}
where
\begin{equ}
\bar{\tau}_{\eps}(x) = \frac{1}{2 a^3 \eps^2} \int K(x-y) K(x-z) \cent{ \cent{F''(X)} \cent{F(Y)} \cent{F(Z)} } \,dy\,dz. 
\end{equ}
We deal with the second term first. Since $\E \cent{F(Y)} \cent{F(Z)} \lesssim (\E YZ)^{2}$, it is not hard to see that $C_{\<2'2'0s>}^{(\eps)} = \oO (|\log \eps|)$. The same procedure as above also implies $\|\scal{\cent{F''(X)}, \varphi^{\lambda}}\|_{2n} \lesssim \eps^{\kappa'} \lambda^{-\kappa'}$. Thus, the term $\frac{1}{2}C_{\<2'2'0s>}^{(\eps)} \cent{F''(X)}$ vanishes at the correct order. It then suffices to show that $\bar{\tau}_{\eps}$ vanishes and satisfies the bound \eqref{e:main_model_bound}. 

If we brutally implement the above procedure with the general bound \eqref{e:special1}, then because of the two derivatives on $F$ (with $X$), we will end up with requiring $9+$ differentiability of $F$ in order for $\bar{\tau}_{\eps}$ to satisfy the desired bound\footnote{The $9+$ differentiability arises as follows. The product $\cent{\cos(\theta_{\fx} X)} \cent{\cos(\theta_{\fy} Y)} \cent{\cos(\theta_{\fz}Z)}$ gives rise to $\theta^6$, and another two powers of $\theta$ come from $F''$. Thus, if we brutally bound things in this way, we need $\|\hF\|$ to decay faster than $\theta^{-9-}$.}. However, a more careful observation reveals that, up to the subtraction of $C_{\<2'2'0's>}^{(\eps)}$, $\bar{\tau}_{\eps}$ is the product of three almost bounded processes: $\hPi^{\eps} \<0'>$ and $(\hPi^{\eps} \<2'0>)^2$. By writing $F'' = F_{\delta}'' + (F'' - F_{\delta}'')$, we can ignore the bad effect of two derivatives on $F_{\delta}$. As for the small remainder, we can separate the product of the three ``almost bounded'' processes into sub-products with less terms. This will reduce the requirement on the regularity of $F$. 

More precisely, we write $\bar{\tau}_{\eps}(x) = \bar{\tau}_{\eps,\delta}^{(1)}(x) + \bar{\tau}_{\eps,\delta}^{(2)}(x)$, where
\begin{equs}
\bar{\tau}_{\eps,\delta}^{(1)}(x) &= \frac{1}{2 a^3 \eps^2} \int K(x-y) K(x-z) \cent{\cent{F_{\delta}''(X)} \cent{F(Y)} \cent{F(Z)}} \,dy\,dz,\\ 
\bar{\tau}_{\eps,\delta}^{(2)}(x) &= \frac{1}{2 a^3 \eps^2} \int K(x-y) K(x-z) \cent{\cent{F''(X) - F_{\delta}''(X)} \cent{F(Y)} \cent{F(Z)}} \,dy\,dz. 
\end{equs}
For $\bar{\tau}_{\eps,\delta}^{(1)}$, similar as before, we get a power $(1+|\K|)^{6}$ from the norm of ``test functions'' localised in $\fR_{\K}$. As for the norm of distributions, by Assumption~\ref{as:F} and Lemma~\ref{le:dist_product}, we have
\begin{equ}
\|\widehat{F_{\delta}''} \otimes \hF \otimes \hF\|_{M+2,\fR_K} \lesssim_{N} \delta^{-N} (1+|K_{\fx}|)^{-N} \big( (1+ |K_{\fy}|)(1+|K_{\fz}|) \big)^{-7-\alpha}
\end{equ}
for all large $N$. By choosing $\delta = \eps^{\frac{\kappa'}{4N}}$ as before, we obtain the bound
\begin{equ} \label{e:extra_1}
\|\scal{\bar{\tau}_{\eps,\delta}^{(1)}, \varphi^{\lambda}}\|_{2n} \lesssim \eps^{\frac{\kappa'}{2}} \lambda^{-\kappa'}. 
\end{equ}
As for $\tau_{\eps,\delta}^{(2)}$, we have the expression
\begin{equ}
\scal{\tau_{\eps,\delta}^{(2)}, \varphi^{\lambda}} = \int \cent{F''(X) - F_{\delta}''(X)} \cdot \Big( \eps^{-1} \int K(x-y) \cent{F(Y)} dy \Big)^{2} \varphi^{\lambda}(x) dx. 
\end{equ}
We first separate the two terms in the integrand so that
\begin{equ} \label{e:separate_product}
|\scal{\tau_{\eps,\delta}^{(2)}, \varphi^{\lambda}}| \lesssim \sup_{x} |\cent{F''(X) - F_{\delta}''(X)}| \cdot \scal{(\<2'0>_{\eps})^{2}, |\varphi^{\lambda}|}, 
\end{equ}
where $\<2'0>_{\eps} = \hPi^{\eps} \<2'0> = K * \hPi^{\eps} \<2'>$. Considering the $L^{2n}$ norm of both sides and applying H\"{o}lder's inequality, we get
\begin{equ}
\|\scal{\tau_{\eps,\delta}^{(2)}, \varphi^{\lambda}}\|_{2n} \lesssim \big( \E \sup_{x} |\cent{F''(X) - F_{\delta}''(X)}|^{4n} \big)^{\frac{1}{4n}} \cdot \|\scal{(\<2'0>_{\eps})^{2}, |\varphi^{\lambda}|}\|_{4n}. 
\end{equ}
The first term is easily seen to be bounded by some positive power of $\eps$ since exchanging the supremum with the expectation costs an arbitrary small power of $\eps$, but we have $\delta = \eps^{\frac{\kappa'}{4N}}$. For the second term, we have $\|\scal{(\<2'0>_{\eps})^{2}, |\varphi^{\lambda}|}\|_{4n} \lesssim |\log \eps| \cdot \lambda^{-\kappa'}$. This logarithmic factor can be killed by the positive power of $\eps$ from the first term, so that we get
\begin{equ} \label{e:extra_2}
\|\scal{\tau_{\eps,\delta}^{(2)}, \varphi^{\lambda}}\|_{2n} \lesssim \eps^{\kappa''} \lambda^{-\kappa'}. 
\end{equ}
The bound for $\bar{\tau}_{\eps}$ and hence the symbol \<2'2'0'> then follows immediately by combining \eqref{e:extra_1} and \eqref{e:extra_2}.

\begin{rmk}
	It is clear from the above argument that the $7+$ differentiability of $F$ exactly guarantees the bound \eqref{e:extra_1} for $\bar{\tau}_{\eps,\delta}^{(1)}$. The readers may wonder why we only separate $F_{\delta}''$ from $F''$ but not the other two $F$'s. The reason is that in \eqref{e:separate_product} when we apply $L^{\infty}$ and $L^{1}$ bounds to separate the integrands, it is essential to have the stochastic part of the latter integrand positive (like $(\<2'0>_{\eps})^{2}$) so that we can keep the structure of the stochastic objects without adding absolute value to them. 
\end{rmk}

\subsection{Behaviour of renormalisation constants}

\label{sec:renorm_const}

We now explore behaviour of the renormalisation constants $C_{\tau}^{(\eps)}$'s defined in \eqref{e:renorm_const}. It is clear that $C_{\<2's>}^{(\eps)}$ diverges at order $\eps^{-1}$. As for the other constants, we will see that both $C_{\<2'2'0s>}^{(\eps)}$ and $C_{\<2'1'1's>}^{(\eps)}$ diverge logarithmically. However, the two logarithmic divergences actually cancel each other out, as already shown in \cite{HairerKPZ,HQ} for polynomial $F$. Finally, $C_{\<2'2'0's>}^{(\eps)}$ is uniformly bounded in $\eps$. 

We first show the cancellation of the two logarithmic divergences for general $F$. This follows from that for the polynomial $F$ and Lemma~\ref{le:000}. We give it in the following proposition. 

\begin{prop} \label{pr:log_cancel}
	We have
	\begin{equ}
	\sup_{\eps \in (0,1)} \big(C_{\<2'2'0s>}^{(\eps)} + 4 C_{\<2'1'1's>}^{(\eps)} \big) < +\infty. 
	\end{equ}
\end{prop}
\begin{proof}
	Let
	\begin{equs}
	\widetilde{C}_{\<2'2'0s>}^{(\eps)} &= \int K(x-y) K(x-z) \E \big( \Psi_{\eps}^{\diamond 2}(y) \Psi_{\eps}^{\diamond 2}(z) \big) \,dy \,dz, \\
	\widetilde{C}_{\<2'1'1's>}^{(\eps)} &= \int K(x-y) K(y-z) \E \big( \Psi_{\eps}(x) \Psi_{\eps}(y) \Psi_{\eps}^{\diamond 2}(z) \big) \,dy\, dz. 
	\end{equs}
	We first show that these two quantities contain all the logarithmic divergence for $C_{\<2'2'0s>}^{(\eps)}$ and $C_{\<2'1'1's>}^{(\eps)}$ in the sense that both $C_{\<2'2'0s>}^{(\eps)} - \widetilde{C}_{\<2'2'0s>}^{(\eps)}$ and $C_{\<2'1'1's>}^{(\eps)} - \widetilde{C}_{\<2'1'1's>}^{(\eps)}$ are uniformly bounded in $\eps$. To see this, note that
	\begin{equ}
	C_{\<2'2'0s>}^{(\eps)} - \widetilde{C}_{\<2'2'0s>}^{(\eps)} = \frac{1}{a^2 \eps^2} \int K(x-y) K(x-z) \E \big( \cent{F(Y)} \cent{F(Z)} - a^2 Y^{\diamond 2} Z^{\diamond 2} \big) \,dy\, dz. 
	\end{equ}
	One can directly perform a chaos expansion and show as in the proof of Lemma~\ref{le:000} that
	\begin{equ}
	\big|\E \big( \cent{F(Y)} \cent{F(Z)} - a^2 Y^{\diamond 2} Z^{\diamond 2} \big) \big| \lesssim (\E YZ)^{4}. 
	\end{equ}
	This implies the bound
	\begin{equ}
	|C_{\<2'2'0s>}^{(\eps)} - \widetilde{C}_{\<2'2'0s>}^{(\eps)}| \lesssim \eps \int |K(x-y) K(y-z)| \cdot \E \big(\Psi_{\eps}^{\diamond 4}(y) \Psi_{\eps}^{\diamond 4}(z) \big) \,dy\,dz. 
	\end{equ}
	By the bounds for polynomial models in \cite{HQ} (second part of Thm 6.5), the right hand side above converges to a finite limit as $\eps \rightarrow 0$, and hence $C_{\<2'2'0s>}^{(\eps)} - \widetilde{C}_{\<2'2'0s>}^{(\eps)}$ stays finite. As for $C_{\<2'1'1's>}^{(\eps)}$, we have
	\begin{equ}
	C_{\<2'1'1's>}^{(\eps)} - \widetilde{C}_{\<2'1'1's>}^{(\eps)} = \frac{1}{4a^3 \eps^2} \sum_{j=1}^{2} \int K(x-y) K(y-z) H^{(j)}(x,y,z) \,dy\,dz, 
	\end{equ}
	where $H^{(j)}$ are given by Lemma~\ref{le:000}. The bound for $H^{(2)}$ implies that the integral with $H^{(2)}$ stays finite as $\eps \rightarrow 0$. As for $H^{(1)}$, one can choose $\delta$ to be a small positive power of $\eps$ so that the bound in Lemma~\ref{le:000} is sufficient to guarantee that the integral with $H^{(1)}$ vanishes in the limit as $\eps \rightarrow 0$. Hence, $C_{\<2'1'1's>}^{(\eps)} - \widetilde{C}_{\<2'1'1's>}^{(\eps)}$ is also uniformly bounded in $\eps$. 
	
	Since $\widetilde{C}_{\<2'2'0s>}^{(\eps)}$ and $\widetilde{C}_{\<2'1'1's>}^{(\eps)}$ are precisely the logarithmically divergent quantities from the KPZ equation (with $F(u) = u^2$), it then follows from \cite[Thm~6.5]{HQ} that $\widetilde{C}_{\<2'2'0s>}^{(\eps)} + \widetilde{C}_{\<2'1'1's>}^{(\eps)}$ converges to a finite limit. This completes the proof of the proposition. 
\end{proof}

We have the following proposition on the behaviour of $C_{\<2'2'0's>}^{(\eps)}$ as $\eps \rightarrow 0$. 

\begin{prop} \label{pr:finite_renorm}
	The constant $C_{\<2'2'0's>}^{(\eps)}$ defined in \eqref{e:renorm_const} is uniformly bounded in $\eps \in (0,1)$. 
\end{prop}
\begin{proof}
	We would like to have the control
	\begin{equ} \label{e:C3_bound}
	\E \big( \cent{F''(X)} \cent{F(Y)} \cent{F(Z)} \big) \lesssim \E \big( X^{\diamond 2} Y^{\diamond 2} Z^{\diamond 2} \big). 
	\end{equ}
	This will immediately reduce the problem to the polynomial case and gives the claim. However, similar as in the case for the symbol \<2'2'0'>, if we Fourier expand $F''$ and $F$ and brutally apply Theorem~\ref{th:special_bound}, then we would need $F \in \cC^{9+}$ to get the bound \eqref{e:C3_bound}. Fortunately, there are only three terms in the product, and hence we can do exact computations with the trigonometric identities. This will give us the desired bound \eqref{e:C3_bound} under Assumption~\ref{as:F}, and hence the claim follows. 
\end{proof}

\begin{rmk}
	With some extra effort, it is possible to show that both $C_{\<2'2'0s>}^{(\eps)} + 4 C_{\<2'1'1's>}^{(\eps)}$ and $C_{\<2'2'0's>}^{(\eps)}$ converge to a finite limit as $\eps \rightarrow 0$. But the convergence does not really matter here, so we omit those details and only claim uniform boundedness. 
\end{rmk}

\subsection{Identification of the limit}

We now have all the ingredients in place to prove the main result of the article. 

\begin{thm} \label{th:identification}
Let $\gamma \in (\frac{3}{2}, \frac{5}{3})$ and $\eta \in (\frac{1}{2} - \frac{1}{M+4},\frac{1}{2})$. 
Let $\{h_{0}^{(\eps)}\}$ be a sequence of functions in $\cC^{\gamma,\eta}_{\eps}$ such that 
$\|h_{0}^{(\eps)}; h_{0}\|_{\gamma,\eta;\eps} \rightarrow 0$ for some $h_{0} \in \cC^{\eta}$. Then, 
there exists $C_{\eps} \rightarrow +\infty$ such that for every $T>0$, the solution $h_{\eps}$ defined 
in \eqref{e:macro_eq_appro} with initial data $h_{0}^{(\eps)}$ converge in probability in 
$\cC^{\eta}([0,T] \times \T)$ to the solution to the KPZ equation with parameter $a$
and initial condition $h_0$. 
\end{thm}
\begin{proof}
	Let $\Pi^{\KPZ}$ be the standard KPZ model described in the Appendix. Let $\hPi^{\eps}$ be the sequence of models defined in \eqref{e:model_new} and \eqref{e:model_2} with input $\psi = \Psi_{\eps}$ and the constants $C_{j}^{(\eps)}$ specified in \eqref{e:renorm_const}. Let $u_{0}^{(\eps)} = h_{0}^{(\eps)} - Z_{\eps}(0,\cdot)$ where $Z_{\eps} = P * \xi_{\eps}$. By the convergence of $Z_{\eps}$ to $Z = P*\xi$ and the convergence of $h_{0}^{(\eps)}$, we have $\|u_{0}^{(\eps)};u_{0}\|_{\gamma,\eta;\eps} \rightarrow 0$ in probability where $u_{0} = h_{0} - Z(0,\cdot)$.

Now, let $U^{(\eps)} \in \dD^{\gamma,\eta}_{\eps}$ and $U \in \dD^{\gamma,\eta}$ denote the solutions to the fixed point problem with models $\hPi^{\eps}$ and $\Pi^{\KPZ}$ and initial datum $u_{0}^{(\eps)}$ and $u_{0}$ respectively. Let
\begin{equ}
	u_{\eps} := \widehat{\rR}^{\eps} U^{(\eps)}\;, \qquad u := \rR^{\KPZ} U\;, 
\end{equ}
where $\widehat{\rR}^{\eps}$ and $\rR^{\KPZ}$ are the reconstruction maps associated to the corresponding models. By Theorem~\ref{th:renorm_eq}, $u_{\eps}$ solves the remainder equation~\ref{e:remainder} with initial data $u_{0}^{(\eps)}$ and constant $C_{\eps}$ as in \eqref{e:constantEps}, so that $h_{\eps} = u_{\eps} + Z_{\eps}$ solves \eqref{e:macro_eq_appro} with initial condition $h_{0}^{(\eps)}$.
Similarly, combining its definition with \cite[Thm~1.2]{HQ} (see also \cite[Thm~15.1]{Peter}), 
it follows that there exists a constant $c$ such that $h := u + Z - ct$ solves the KPZ equation 
(in the sense of Hopf--Cole) with 
parameter $a$ and initial condition $h_{0}$. Since $Z_{\eps} \rightarrow Z$ in probability in $\cC^{\eta}([0,T] \times \T)$, the convergence of $h_{\eps}$ to $h$ is established if we show 
	that $u_{\eps} \rightarrow u$ in the same space. The additional term $ct$ can of course easily be generated
	for $h_\eps$ as well by adding $c$ to $C_\eps$.
	
By Theorem~\ref{th:main_converge}, we have $\|\hPi^{\eps}; \Pi^{\KPZ}\|_{\eps,0} \rightarrow 0$. Also, one can easily check that
\begin{equ}
	\lim_{\eps \to 0} \E \sup_{z \in [0,T] \times \T} \eps^{\frac{1}{2}+\kappa}|\Psi_{\eps}(z)| = 0\;. 
\end{equ}
Furthermore, we know from \cite[Prop.~7.11]{Hai14a} 
and the fact that the Hopf--Cole solutions to the KPZ equation are global in time almost surely,
that we can find $U$ solving \eqref{e:fixed_pt_0} up to any fixed time $T>0$. 
Hence, Theorem~\ref{th:abstract_fpt} and continuity of the reconstruction maps imply that $u_{\eps} \rightarrow u$ in probability in $\cC^{\eta}([0,T] \times \T)$ as required. 
\end{proof}

\begin{remark}
Note that for any fixed $\eps > 0$, there may be a set $\Omega_T$ of strictly positive probability such that 
for realisations $\xi_\eps \in \Omega_T$, the solution to \eqref{e:macro_eq_appro} explodes before time $T$. 
This is not a problem, convergence in probability should be interpreted as stating that the probability of
this event converges to $0$ and that the solution conditioned on survival converges in probability.
\end{remark}

\section{A general pointwise bound}

\label{sec:general_bound}

In this section, we state and prove two general bounds that control the correlation functions of 
trigonometric polynomials of Gaussian processes by those of suitable polynomials. These bounds, combined with 
the convergence of polynomial models in \cite{HQ}, are the main ingredient for the convergence result in Theorem~\ref{th:main_converge}.

\subsection{The statement}

\label{sec:general_statement}

For every finite set $\mathscr{A}$, let $\N^\sA$ be the set of multi-indices on $\sA$. We define the length 
of a multi-index by $|\n| = \sum_{\alpha \in \sA} n^{\alpha}$. For $\sB \subset \sA$, we write
$\N_{\sB}^{\sA} \subset \N^{\sA}$ for those multiindices that vanish on $\sB$. For every collection of joint Gaussian random variables $\X = (X_{\alpha})_{\alpha \in \sA}$ and every $\n \in \N^\sA$, we write $\X^{\diamond \n} = \bdiamond_{\alpha} X_{\alpha}^{\diamond n^{\alpha}}$. 

For $H\colon \R^{\sA} \to \R$ a continuous function with at most exponential growth, $\X = (X_\alpha)_{\alpha \in \sA}$ a collection of jointly Gaussian random variables, and $\mM \subset \N^\sA$ finite, we define
\begin{equ} [e:defExpansion]
	\mathscr{T}_{\mM} \big(H(\X)\big) = H(\X) - \sum_{\n \in \mM} C_{\n} \X^{\diamond \n}, 
\end{equ}
where $C_{\n} = \E(\d^\n H)(\X) / \n!$ is the coefficient of the term $\X^{\diamond \n}$ in the chaos expansion of $H(\X)$. 

Let $\tT$ be a finite set of ``types'', which comes equipped with a partition $\tT = \oO \sqcup \eE$, with
$\oO$ denoting ``odd'' points and $\eE$ standing for ``even'' points. We now consider $\tT$, $\oO$ and $\eE$
as being fixed. 
For every $\n \in \N^\tT$, we define the following sets: 
\begin{equs}[2] [e:def_subsets]
	\oO_{1}(\n) &= \{t \in \oO: n^t \phantom{1} \text{odd}\}, &\qquad \oO_{2}(\n) &= \{t \in \oO: n^t \phantom{1} \text{even}\}, \\
	\eE_{1}(\n) &= \{t \in \eE: n^t \phantom{1} \text{odd}\},  &\qquad 
	\eE_{2}(\n) &= \{t \in \eE: n^t \geq 2 \phantom{1} \text{even}\},\\
	\eE_{0}(\n) &= \{t \in \eE: n^t = 0\}. 
\end{equs}
Note that $\oO_2(\n)$ includes those $t \in \oO$ for which $n^t = 0$, while for $t \in \eE$ we separate $0$ and strictly positive even indices. We will sometimes omit the argument $\n$ from the sets $\oO_j$ and $\eE_j$ and
we use $|\cdot|$ to denote the cardinality of a set. We also define
\begin{equ} \label{e:branching}
	\bB(\n) = \big\{ 2\k+\n: \k \in \N_{\eE_0(\n)}^{\tT} \big\}\;, 
\end{equ}
and we call this the ``branching'' of $\n$. Note that in $\bB(\n)$, we add positive even integers only to the components $t$ which are \textit{not} in $\eE_{0}(\n)$. 

Consider a $\tT$-tuple of space-time points $(x_t)_{t \in \tT}$ in $\R^+ \times \T$, and recall the notation $X_t = \eps^{\frac{1}{2}} \Psi_{\eps}(x_t)$. Also let $\theta_t \in \R$ for every $t$. For convenience, we define $\trig_t = \sin$ if $t \in \oO$ and $\trig_t = \cos$ if $t \in \eE$. Since $\E \sin(\theta_t X_t) = 0$, we have
\begin{equ} \label{e:defPhi}
\Phi(\Btheta,\X) := \prod_{t \in \oO} \sin(\theta_t X_t) \prod_{t \in \eE} \cent{\cos(\theta_t X_t)} = \prod_{t \in \tT} \cent{\trig_t(\theta_t X_t)}, 
\end{equ}
where we have used the shorthand notation $\Btheta = (\theta_t)_{t \in \tT}$ and $\X = (X_t)_{t \in \tT}$. For every $\n \in \N^\tT$, we write $C_{\n}(\Btheta,\X)$ for the coefficient of $\X^{\diamond \n}$ in the chaos expansion of $\Phi(\Btheta,\X)$, defined as in \eqref{e:defExpansion}. In this section, we will be interested
in the correlation functions for quantities of the form $\sT_{\mM}(\Phi(\Btheta,\X))$. The set $\mM$ of removed 
chaos components will always be chosen such that it satisfies the following assumption. 

\begin{assumption} \label{as:removal_chaos}
	$\mM\subset \N^\tT$ is finite and there exists a finite set of ``roots'' $\rR \subset \mM^{c}$ such that $\bigcup_{\m \in \rR} \bB(\m) = \mM^c$.
\end{assumption}

For every $t \in \tT$ and $K \geq 1$, fix $K$ space-time points $\{x_{k,t}: k \in [K]\}$. Here and below, we use $[K]$ to denote the set $\{1, \dots, K\}$. For every $k \in [K]$, we write $\x_k = (x_{k,t})_{t \in \tT}$ and $\X_k = (X_{k,t})_{t \in \tT}$, where as before $X_{k,t} = \eps^{\frac{1}{2}} \Psi_{\eps}(x_{k,t})$. For every $N \geq 1$ and every Gaussian random variable $X$, we also write $P_N$ and $Q_N$ for the random
variables
\begin{equ}
P_{N}(X) = \sum_{k=1}^{N} X^{\diamond (2k-1)}, \qquad Q_{N}(X) = \sum_{k=1}^{N} X^{\diamond (2k)}. 
\end{equ}
Note that $Q_{N}$ has no constant term. The main statement is the following. 

\begin{thm} \label{th:general_bound}
	Let $\mM \subset \N^\tT$ satisfy Assumption~\ref{as:removal_chaos}. Let $\r \in \N^\tT$. Then, there exists $N$ depending on $K$, $|\tT|$ and $\mM$ only such that
	\begin{equ} \label{e:general_bound}
	\bigg| \E \prod_{k=1}^{K} \sT_{\mM} \Big( \d_{\Btheta}^{\r} \Phi(\Btheta,\X_k)\Big)\bigg| \lesssim \theta^{N} \E \prod_{k=1}^{K} \sT_{\mM} \bigg[ \Big(\prod_{t \in \oO}P_{N}(X_{k,t}) \Big) \Big( \prod_{t \in \eE}Q_{N}(X_{k,t}) \Big) \bigg], 
	\end{equ}
    where $\theta = 1 + \max_{t} |\theta_t|$, and the proportionality constant depends on $K$, $\r$ and $|\tT|$ only. In particular, the bound is uniform over all locations of the points $\{x_{k,t}\}$ and all $\eps \in (0,1)$. 
\end{thm}

\begin{remark}
We do not claim that this is a general statement about arbitrary collections of Gaussian random variables.
Instead, this only holds for collections of the type $X_t = \eps^{\frac{1}{2}} \Psi_{\eps}(x_t)$ as always
considered in this article. In particular, we exploit the fact that the $X_t$ are all positively
correlated and the fact that their covariances are related in a specific way to the distances
between the corresponding space-time locations, which themselves satisfy the triangle inequality.
Both of these ingredients are crucial in the proof.
\end{remark}

We will mainly use the above bound for arbitrary even $K$ and for $K=1$. In the case $K=1$, if $\0 \in \mM$, then both sides of \eqref{e:general_bound} vanish and the bound holds trivially. If $\0 \notin \mM$ (still with $K=1$), then both sides of \eqref{e:general_bound} are the same as if $\mM = \emptyset$, and hence one can remove $\mM$. In this case, we have an improved bound in terms of a better power of $\theta$. It can be stated in a more general form as follows. 

\begin{thm} \label{th:special_bound}
	Let $\tilde{\tT}$ be a finite set of types with even and odd categories $\tilde{\eE}$ and $\tilde{\oO}$. Then, we have
	\begin{equ} \label{e:special1}
	|\E \d_{\Btheta}^{\r} \Phi(\Btheta,\X)| \lesssim \theta^{|\tilde{\oO}| + 2|\tilde{\eE}|} \E \Big[ \Big(\prod_{t \in \tilde{\oO}} X_t\Big) \Big(\prod_{t \in \tilde{\eE}} X_{t}^{\diamond 2}\Big) \Big], 
	\end{equ}
	where $\theta = 1 + \max_{t} |\theta_t|$. If we do not recenter the factors of $\Phi$, then we have the bound
	\begin{equ} \label{e:special2}
	\Big|\E \d_{\Btheta}^{\r}  \Big(\prod_{t \in \tilde{\tT}} \trig_{t}(\theta_t X_t)  \Big)\Big| \lesssim
	 \theta^{|\tilde{\oO}|} \E \prod_{t \in \tilde{\oO}} X_t. 
	\end{equ}
	In particular, if $|\tilde{\oO}|$ is odd, then both sides of \eqref{e:special2} vanish. 
\end{thm}

\begin{rmk}
	Theorem~\ref{th:special_bound} will be applied in two places. We will first use it to get a fine control of the coefficients $C_{\n}(\Btheta,\X_k)$ (see the expression \eqref{e:coeff_expression}) right after we prove it in Section~\ref{sec:special_proof}. In this case we take $\tilde{\tT} = \tT$. Second, we apply it to the small ``remainder'' part of the nonlinearity $F$ to reduce the regularity assumption for the convergence in Section~\ref{sec:convergence}. In this case, we will take $\tilde{\tT} = [K] \times \tT$, and each point $x_{k,t}$ should be thought of as having a unique type in $\tilde{\tT}$. 
\end{rmk}

In the remaining part of this section, we will first set up a general method to prove Theorem~\ref{th:special_bound}. Since the bound \eqref{e:general_bound} is trivial when $K=1$ and $\0 \in \mM$, and is indifferent between all other $\mM$ that do not contain $\0$ (including the empty set), this already establishes \eqref{e:general_bound} for $K=1$. We will then use Theorem~\ref{th:special_bound} to prove Theorem~\ref{th:general_bound} for $K \geq 2$. Also note that when both $K$ and $|\oO|$ are odd (or just $|\tilde{\oO}|$ being odd in the case of Theorem~\ref{th:special_bound}), all of the above bounds are trivial since both sides of the inequality are $0$.

\subsection{Factorial decay of the coefficients}

In this subsection, we give a bound on the decay of $C_{\n}(\Btheta,\X)$, as defined in \eqref{e:defPhi}
and \eqref{e:defExpansion}. Recall \eqref{e:def_subsets} and in particular that 
$\oO_i$ and $\eE_i$ depend on $\n$. The definition of $C_{\n}(\Btheta,\X)$ immediately yields the identity
	\begin{equs} [e:coeff_expression]
	C_{\n}(\Btheta,\X) &= (-1)^{\frac{1}{2}(|\n|+|\eE_1|-|\oO_1|)} \cdot \frac{\Btheta^{\n}}{\n!}\\
	\cdot \E &\Big[ \Big(\prod_{t \in \oO_2 \cup \eE_1} \sin (\theta_t X_t) \Big) \Big( \prod_{t \in \oO_1 \cup \eE_2 }  \cos (\theta_t X_t) \Big) \Big( \prod_{t \in \eE_0} \cent{\cos (\theta_t X_t)} \Big) \Big]. 
	\end{equs}
The following proposition is then an easy consequence. 

\begin{prop} \label{pr:coeff_1}
	There exists $C>0$ depending on the multi-index $\r$ only, such that for every $\n \in \N^\tT$ with $|\n|$ having the same parity as $|\oO|$, every $(\theta_t)_{t \in \tT}$ and every $(x_t)_{t \in \tT}$, one has
	\begin{equ} \label{e:coeff_0}
	|\d_{\Btheta}^{\r} C_{\n}(\Btheta,\x)| \leq \frac{(C \theta)^{|\n|}}{\n!}, 
	\end{equ}
	where $\theta = 1 + \max_t |\theta_t|$. As a consequence, we have
	\begin{equ} \label{e:coeff_1}
	\sum_{|\n|=N} |\d_{\Btheta}^{\r} C_{\n}(\Btheta,\x)| \leq \frac{(C \theta)^{N}}{N!} 
	\end{equ}
	if $N$ has the same parity as $|\oO|$, and $0$ otherwise. The constant $C$ in \eqref{e:coeff_1} depends on $\r$ and $|\tT|$. 
\end{prop}
\begin{proof}
	The parity of $|\oO|$ determines whether $\Phi(\Btheta,\X)$ expands into an odd or even chaos series. Thus, $C_{\n}(\Btheta,\X)=0$ if $|\n|$ has a different parity as $|\oO|$. As for the bounds, the first claim is an immediate consequence of \eqref{e:coeff_expression}, and the second claim follows 
	from the multinomial theorem. 
\end{proof}

\begin{rmk}
Proposition~\ref{pr:coeff_1} gives the factorial decay of the coefficients, and (a simple variant of it) is enough to prove Theorem~\ref{th:special_bound}. On the other hand, the proof of Theorem~\ref{th:general_bound} requires finer control of the $C_{\n}$'s beyond factorial decay, namely the dependence on $\x$. We will obtain this control from Theorem~\ref{th:special_bound}. 
\end{rmk}

\subsection{Clustering}
\label{sec:clustering}

Let $\tT = \oO\sqcup \eE$ be the type space described in Section~\ref{sec:general_statement}, and denote the whole collection of $K |\tT|$ points by $(x_{u})_{u \in \Ind}$ with $\Ind = [K] \times \tT$. We also fix a finite set $\mM \subset \N^{\tT}$ satisfying Assumption~\ref{as:removal_chaos} as well as a ``root set'' $\rR$. 

Let $(\theta_t)_{t \in \tT}$ be a fixed collection of ``frequencies'', and $\theta = 1 + \max_{t} |\theta_t|$. Our aim is to obtain the bound in Theorem~\ref{th:general_bound} uniformly over all locations of the points as well as the $\theta_t$'s, except of course for the polynomial dependence in $\theta$ that is already explicit in the statement.

To do this, we divide the points into clusters of size of order $\theta^{2} \eps$ in the following way. Let $\sim$ be the transitive closure of the relation on $\Ind$ given by $u\sim u'$ if
$|x_u - x_{u'}| \le \theta^2 \eps$.
We write $\Clus$ for the partition of $\Ind$ into clusters obtained in this way. For $\uU \subset \Ind$, we also write $\Clus(\uU)$ for the collection of clusters in $\Clus$ that contain at least one point in $\uU$. 

Let $\sS_{\max}$ denote the set of indices $s$ such that for every cluster $\cC \in \Clus$, $\cC \cap \big(\{s\} \times \tT \big) = \cC$ or $\emptyset$. In other words, indices $s$ in $\sS_{\max}$ are such that, starting from
a point $x_{(s,t)}$ with any $t \in \tT$ and performing steps of size at most $\theta^2 \eps$, it is 
only possible to reach other points $x_{(s,t')}$.
Recall the definitions of $\sT_{\mM}$ and $\Phi$ from \eqref{e:defExpansion} and \eqref{e:defPhi}. For any 
$\sS \subset K$ (but we will always choose $\sS \subset \sS_{\max}$), we then have
\begin{equs} [e:prod_expansion]
\prod_{k=1}^{K} \sT_{\mM} \big(\d_{\Btheta}^{\r}\Phi (\Btheta,\X_k)\big) &= \sum_{\iI \sqcup \jJ = [K] \setminus \sS} (-1)^{|\iI|} \Bigg[ \Big(\prod_{s \in \sS} \sT_{\mM}\big( \d_{\Btheta}^{\r} \Phi(\Btheta,\X_s)\big)\Big)\\
&\quad \Big(\prod_{i \in \iI} \sum_{\p \in \mM} \d_{\Btheta}^{\r} C_{\p}^{(i)} \X_{i}^{\diamond \p} \Big) \Big( \prod_{j \in \jJ} \d_{\Btheta}^{\r} \Phi(\Btheta,\X_j) \Big) \Bigg], 
\end{equs}
where the sum is taken over all disjoint subsets $\iI$ and $\jJ$ such that $\iI \cup \jJ = [K] \setminus \sS$. Here, $C_{\p}^{(i)}$ denotes the coefficient of the term $\X_{i}^{\diamond \p}$ in the chaos expansion of $\Phi (\Btheta,\X_i)$, and hence $\d_{\Btheta}^{\r} C_{\p}^{(i)}$ denotes the coefficient for the same term in $\d_{\Btheta}^{\r} \Phi(\Btheta,\X_i)$. We keep the superscript $i$ to emphasise its dependence on the locations $\x_i = (x_{i,t})_{t}$. Note that each term on the right hand side above depends on $\sS$, while the left hand side does not.

\begin{rmk}
While \eqref{e:prod_expansion} holds for any subset $\sS \subset [K]$, only $\sS \subset \sS_{\max}$ gives us 
useful control on its right hand side. Such sets $\sS$ have the property that for every $s \in \sS$, every $k \neq s$ and all $t,t' \in \tT$, one has
\begin{equ} \label{e:singleton_distance}
|x_{s,t}-x_{k,t'}| \geq \theta^{2} \eps. 
\end{equ}
This is the only property that the analysis in Sections~\ref{sec:representative} and~\ref{sec:induction} is based on, and $\sS_{\max}$ is the maximal subset of $[K]$ with this property. On the other hand, the final proofs of Theorems~\ref{th:special_bound} and~\ref{th:general_bound} require more specific choice of $\sS$ (either $\emptyset$ or $\sS_{\max}$), but this happens only from Section~\ref{sec:special_proof}. Hence, for this moment, we assume $\sS \subset \sS_{\max}$. When we reach the final stage of the proof of the main theorems, we will make clear the specific choice of $\sS$ in those situations. 
\end{rmk}

We now proceed with an arbitrary (fixed) $\sS \subset \sS_{\max}$. Each term in the sum on the right hand side of \eqref{e:prod_expansion} is a product of three terms, each being a product over a (possibly empty) subset of $[K]$. For $s \in \sS$, considering the corresponding factor in the first term, 
let $\rR$ be any set of ``roots'' for $\mM$ described above, so we have the chaos expansion
\begin{equ} \label{e:expansion_s}
\sT_{\mM} \big( \d_{\Btheta}^{\r} \Phi(\Btheta,\X_s) \big) = \sum_{\m \in \rR} \sum_{\k \in \N_{\eE_0(\m)}^\tT} \d_{\Btheta}^{\r} \bar{C}_{2\k + \m}^{(s)} \X_{s}^{\diamond 2\k+\m}\;, 
\end{equ}
where we set
\begin{equ} \label{e:coeff_average}
\bar{C}_{\n}^{(s)} = C_{\n}^{(s)} / |\{(\k,\m) \,:\, \m \in \rR,\ \k \in \N_{\eE_0(\m)}^\tT,\ 2\k + \m = \n\}|\;.
\end{equ}
Now we turn to the term involving the product over $j \in \jJ$. For every $\jJ \subset [K] \setminus \sS$, let
\begin{equ}
\sU^{\jJ} := \big\{ \cC \cap (\jJ \times \tT): \cC \in \Clus, |\cC \cap (\jJ \times \tT)| \geq 1 \big\}. 
\end{equ}
We write $\sU^{\jJ} = \sU^{\jJ}_0 \sqcup \sU^{\jJ}_1 \sqcup \sU^{\jJ}_2$, where
\begin{equs} \label{e:divide_UJ}
\sU^{\jJ}_2 &= \{\set \in \sU^{\jJ}\,:\, |\set| = 1,\, \set \subset \jJ \times \eE\}\;,\\
\sU^{\jJ}_1 &= \{\set \in \sU^{\jJ}\,:\, |\set \cap (\jJ \times \oO)| \,\text{is odd}\}\;,\\
\sU^{\jJ}_0 &= \{\set \in \sU^{\jJ}\,:\, |\set| \ge 2,\, |\set \cap (\jJ \times \oO)| \,\text{is even}\}\;.
\end{equs}
For $\set \in \sU^{\jJ}$, we furthermore define subsets $\nN_i(\set) \subset \N^{\set}$ by
\begin{equ}
\nN_2 = \{\n: |\n| \geq 2 \, \text{is even}\}, \quad \nN_1 = \{\n: |\n| \, \text{is odd}\}, \quad 
\nN_0 = \{\n: |\n| \geq 0 \, \text{is even}\},
\end{equ}
and we set $\bar{\nN}(\set) = \nN_i$ for $\set \in \sU^{\jJ}_i$. These sets do depend on $\set$ since
we only consider $\n \in \N^{\set}$. Note that in $\nN_2$, the ``multi-indices'' $\n$ are really positive even integers since $|\set|=1$. We still write it in $\n$ in order to keep the same notation for all $\nN_i$. 

Introduce now variables $\Bbeta = (\beta_{j,t})_{j\in [K], t \in \tT}$ and write $\Bbeta_j = (\beta_{j,t})_{t \in \tT}$. Also let $\s = (\s_{j,t})_{j\in [K], t \in \tT}$ be given by $\s_j = \r$, so that the third product on the right hand side of \eqref{e:prod_expansion} can be written as
\begin{equ}
\prod_{j \in \jJ} \d_{\Btheta}^{\r} \Phi(\Btheta,\X_j) = 
\prod_{j \in \jJ} \d_{\Bbeta_j}^{\s_j} \Phi(\Bbeta_j,\X_j) \Big|_{\Bbeta_j = \Btheta}\;.
\end{equ}
For any $\set \subset [K]\times \tT$, write $\s_\set$ for the restriction of $\s$ to $\set$, and similarly for $\Bbeta$,
so that 
\begin{equ}
\prod_{j \in \jJ} \d_{\Bbeta_j}^{\s_j} \Phi(\Bbeta_j,\X_j) =
\prod_{\set \in \sU^\jJ} \d_{\Bbeta_\set}^{\s_\set} \Phi_\set(\Bbeta_\set,\X_\set) \;,
\end{equ}
where we further used the notation 
\begin{equ}
\Phi_\set(\Bbeta_\set,\X_\set) = \prod_{(j,t) \in \set}  \cent{\trig_t(\beta_{j,t} X_{j,t})}\;.
\end{equ}
Writing similarly to before $C_\n^{(\set)}(\Bbeta_\set)$ for the coefficient of $\X_{\set}^{\diamond \n}$ in the Wiener chaos expansion of $\Phi_\set(\Bbeta_\set,\X_\set)$, we conclude that
\begin{equ}[e:expansion_j']
\prod_{j \in \jJ} \d_{\Btheta}^{\r} \Phi(\Btheta,\X_j) = 
\prod_{\set \in \sU^\jJ} \sum_{\n \in \bar{\nN}(\set)} \d_{\Bbeta_\set}^{\s_\set}C_\n^{(\set)}(\Bbeta_\set) \X_\set^{\diamond \n} \Big|_{\Bbeta_j = \Btheta\,,\forall j\in \jJ} \;.
\end{equ}
In principle, one would think that the sum on the right hand side of \eqref{e:expansion_j'} should 
be taken over the whole set $\N^{\set}$, but since $C_{\n}^{(\set)} = 0$ for $\n \in \N^{\set} \setminus \bar{\nN}(\set)$, 
we can write the sum as above. We will also write $\d_{\Btheta_\set}^{\r_\set}C_\n^{(\set)}(\Btheta_\set) \X_\set^{\diamond \n}$ for each corresponding term on the right hand side of \eqref{e:expansion_j'}, so that we have
\begin{equ} \label{e:expansion_j}
\prod_{j \in \jJ} \d_{\Btheta}^{\r} \Phi(\Btheta,\X_j) = \prod_{\set \in \sU^\jJ} \sum_{\n \in \bar{\nN}(\set)} \d_{\Btheta_\set}^{\r_\set}C_\n^{(\set)}(\Btheta_\set) \X_\set^{\diamond \n}. 
\end{equ}
Note that the right hand side of \eqref{e:expansion_j} is a \textit{notation} for that of \eqref{e:expansion_j'}. Plugging \eqref{e:expansion_s} and \eqref{e:expansion_j} back into \eqref{e:prod_expansion} and expanding the product, we obtain the expression
\begin{equ} \label{e:linear_comb}
\E \prod_{k=1}^{K} \sT_{\mM} \big(\d_{\Btheta}^{\r}\Phi (\Btheta,\X_k)\big) = \sum_{\iI \sqcup \jJ = [K] \setminus \sS} (-1)^{|\iI|} \sum_{\stackrel{\m \in \rR^{\sS}}{\p \in \mM^{\iI}}} \gG_{\iI, \jJ, \m, \p}(\Btheta, \X), 
\end{equ}
where the range of the outer sum is as before and the terms in the inner sum are given by
\begin{equs} \label{e:termwise}
\begin{split}
\gG_{\iI, \jJ, \m,\p}(\Btheta,\X) = &\sum_{\k, \n} \Big(\prod_{s \in \sS} (\d_{\Btheta}^{\r} \bar{C}_{2\k_s+\m_s}^{(s)})\Big)\Big( \prod_{i \in \iI} (\d_{\Btheta}^{\r} C_{\p_i}^{(i)}) \Big)\Big(\prod_{\set \in \sU^\jJ} (\d_{\Btheta_{\set}}^{\r_{\set}} C_{\n_\set}^{(\set)})\Big)\\
&\cdot \E \bigg[\Big( \prod_{s \in \sS} \X_{s}^{\diamond (2\k_s + \m_s)} \Big)\Big(\prod_{i \in \iI} \X_{i}^{\diamond \p_i}\Big)\Big( \prod_{\set \in \sU^\jJ} \X_{\set}^{\diamond \n_{\set}}\Big) \bigg]\;.
\end{split}
\end{equs}
Here, the sum is taken over $\k = (\k_s)_{s \in \sS}$ where each $\k_s$ runs over $\N_{\eE_0(\m_s)}^\tT$, and $\n = (\n_\set)_{\set \in \sU^\jJ}$ where each $\n_\set$ runs over $\bar{\nN}(\set)$. We also use the usual convention that empty products equal $1$.

As for Theorem~\ref{th:general_bound}, for $K \geq 2$ and $\sS = \sS_{\max}$, we will show that each $\gG$ in \eqref{e:termwise} is bounded by the right hand side of \eqref{e:general_bound} (note that this term-wise bound is in general false for $\sS \neq \sS_{\max}$ unless $\mM = \emptyset$). Since $\sS_{\max}, \mM$ and $\rR$ are all finite, Theorem~\ref{th:general_bound} (for $K \geq 2$) then immediately follows from \eqref{e:linear_comb}. Also, as we will obtain the same bound for each choice of $\iI$, $\jJ$, $\m$ and $\p$, we drop these subscripts and simply write $\gG(\Btheta, \X)$ for \eqref{e:termwise}. We will also drop the dependence on these subsets in other situations whenever no confusion may arise (for example, in \eqref{e:divide_UJ}). 

In the case $K=1$, we always have $\sS_{\max} = \{1\}$ whatever the location of the points. Taking $\sS = \sS_{\max}$ in this case simply gives the chaos expansion of the left hand side of \eqref{e:linear_comb}, and no information of the clustering can be used with that expansion. As mentioned in Section~\ref{sec:general_statement}, the only interesting situation under $K=1$ is when $\0 \notin \mM$, which is exactly the same as $\mM = \emptyset$. Hence, it is reduced to Theorem~\ref{th:special_bound}, which can be shown by choosing $\sS = \emptyset$. 

Note that $\iI, \jJ, \m, \p$ (and also $\hH$) all depend on $\sS$. Since most of the intermediate bounds below are true for all $\sS \subset \sS_{\max}$, and there are only finitely many choices of $\sS$, varying $\sS$ only changes the proportionality constants in these bounds. Thus, we will still keep the notation and proceed with general $\sS \subset \sS_{\max}$. We will specify which $\sS$ we use only when it becomes necessary. Also, although $\sS_{\max}$ and hence any $\sS \subset \sS_{\max}$ depends on the location of the points $\{x_{k,t}\}$, the bounds below will be uniform over all locations.

\subsection*{Notations}

In what follows, we use the notations
\begin{equ}
\m_s = (m_s^t)_{t \in \tT}, \quad \m = (\m_s)_{s \in \sS}, \quad |\m_s| = \sum_{t \in \tT} |m_s^t|, \quad |\m| = \sum_{s \in \sS} |\m_s|, 
\end{equ}
and similarly for other multi-indices $\k=(\k_s)_{s \in \sS}$, $\p = (\p_i)_{i \in \iI}$ and $\n = (\n_h)_{h \in \hH}$. 
%We also use the notation
%\begin{equ}
%(2\k_s+\m_s) ! = \prod_{t \in \tT} (2k_s^t+m_s^t)!, \qquad (2\k+\m)! = \prod_{s \in \sS} (2\k_s + \m_s)!
%\end{equ}
%for product of factorials, and the same rule applies to other letters ($\p$ and $\n$). 

\subsection{The representative point}
\label{sec:representative}

We now start to develop ingredients that are needed to bound the right hand side of \eqref{e:termwise}. 
For this, we choose in an arbitrary way one representative point $u^{*}(\set)$ from each cluster $\set \in \sU^\jJ$ and we 
write $X_\set$ instead of $X_{u^{*}(\set)}$ for the corresponding random variable.
We then show that all of the Wick powers $\X_\set^{\diamond \n_\set}$ appearing in \eqref{e:termwise} can be replaced by 
$X_{\set}^{\diamond |\n_\set|}$, at the cost of a polynomial factor in $\theta$, independently of 
the specific choice of $u^{*}(\set)$. The precise statement is the following. 

\begin{prop} \label{pr:rep}
 	Let $\sS \subset \sS_{\max}$ and $\sU^{\jJ}$ be as described above. Let $|\n| = \sum_{\set} |\n_\set|$ and $|\p| = \sum_i |\p_i|$. Then, there exists $C>0$ such that
 	\begin{equs} \label{e:rep}
 	\begin{split}
 	&\phantom{111}\E \bigg[ \Big(\prod_{s \in \sS} \X_{s}^{\diamond (2\k_s + \m_s)}\Big) \Big( \prod_{i \in \iI} \X_{i}^{\diamond \p_i} \Big) \Big( \prod_{\set \in \sU^\jJ} \X_{\set}^{\diamond \n_\set} \Big) \bigg]\\ 
 	&\leq C^{|\n|} \theta^{2|\p|} \E \bigg[ \Big(\prod_{s \in \sS} \X_{s}^{\diamond (2\k_s + \m_s)}\Big) \Big(\prod_{i \in \iI} \X_{i}^{\diamond \p_i}\Big) \Big(\prod_{\set \in \sU^\jJ} X_{\set}^{\diamond |\n_\set|}\Big) \bigg]\;,
 	\end{split}
\end{equs}
uniformly over the exponents $(\k_s), (\p_i), (\n_\set)$, and all locations of points satisfying the constraints
enforced by the definitions of $\sS$, $\iI$ and $\jJ$. 
The constant $C$ can be taken $C = 2 \Lambda^{2} \max_{\set} |\set|$, where $\Lambda$ is as in \eqref{e:corr_ff}. 
\end{prop}

\begin{rmk}
	Note that on the left hand side of \eqref{e:rep}, the boldface letter $\X_{\set}$ refers to the collection of random variables $(X_u)_{u \in \set}$, so
	\begin{equ}
	\X_{\set}^{\diamond \n_\set} = \bdiamond_{u \in \set} X_{u}^{\diamond n^u}. 
	\end{equ}
	On the other hand, $X_{\set}$ with the normal capital $X$ on the right hand side refers to the random variable $X_{u^{*}(\set)}$. Hence, the Wick power $|\n_\set|$ is an integer instead of a multi-index. 
\end{rmk}

\begin{proof}
	When the sum of the exponents $|2\k + \m| + |\p| + |\n|$ is odd, then both sides of \eqref{e:rep} are $0$, so 
	we only consider the case when it is even. 
	
By Wick's formula, both sides of \eqref{e:rep} are a sum over products of pairwise expectations, and the number of pairings for the pairwise products are equal. Furthermore, there is a natural one-to-one correspondence between the pairings on the two sides in that every factor of $X_{u}$ on the left hand side with $u \in \set$ for some $\set \in \sU^\jJ$ is 
replaced by $X_{\set}$ on the right hand side. 
Since all correlations are positive, it thus
suffices to control the effect of such replacements. 
Consider $(i,t) \in [K]\times \tT$, $\set \in \sU^{\jJ}$, and $u, u^* \in \u$.
We then distinguish two cases. In the first case, one has
$|x_{(i,t)}-x_{u}| \ge \theta^2 \eps$, in which case the triangle inequality and the definition of 
a cluster imply that 
\begin{equ}
|x_{(i,t)}-x_{u^*}| \le  |x_{(i,t)}-x_{u}| + |x_{u}-x_{u^*}|
\le |\set| |x_{(i,t)}-x_{u}|\;.
\end{equ}
It then follows from Lemma~\ref{le:corr_change} that in this case
\begin{equ}
\E (X_{i,t} X_{u}) \le \Lambda^2 |\set| \E (X_{i,t} X_{u^*})\;.
\end{equ}
If instead $|x_{(i,t)}-x_{u}| < \theta^2 \eps$, then 
$|x_{(i,t)}-x_{u^*}| \le |\set| \theta^2 \eps$. We then conclude from Lemma~\ref{le:corr_ff} that
\begin{equ}
	\E (X_{i,t} X_{u^*}) \geq \frac{1}{\Lambda(1+|\set| \theta^2)} \geq \frac{1}{2 |\set| \Lambda^2 \theta^2} \E(X_{(i,t)}X_{u})\;.
\end{equ}
Note now that by the properties of the Wick product, correlations of the type
$\E(X_{(i,t)}X_{u})$ only ever show up with either
$(i,t) \in \set'$ for some $\set' \neq \set$ in $\sU^\jJ$, or $i \in\sS$, or $i \in \iI$.
In the first two cases, it follows from the definitions of $\sS$ and our clusters that we
are necessarily in the situation $|x_{(i,t)}-x_{u}| \ge \theta^2 \eps$, so the converse can only
arise for $i \in \iI$.

In order to conclude, it suffices to note that the number of factors with $i \in \iI$ is
precisely $|\p|$, while the total number of factors that require a substitution is 
$|\n| = \sum_{\set} |\n_\set|$.
\end{proof}

\begin{rmk}
	The proposition says that the replacement by a single representative point for all $\set \in \sU^\jJ$ costs $2|\p|$ powers of $\theta$. Since by \eqref{e:prod_expansion}, we are only interested in those $p$ with each $\p_i \in \mM$, the total cost of powers in $\theta$ is always finite and depends on $\mM$ and $K$ only. In particular, for fixed $\mM$, it is linear in $K$. \end{rmk}

Now, for every $\set \in \sU^\jJ$ and $i=0,1,2$, let $q_{\set} = i$ if $\set \in \sU^{\jJ}_{i}$. We have the following proposition. 

\begin{prop} \label{pr:replacement}
	Recall that $|\p| = \sum_{i} |\p_i|$. There exists $C>0$ depending on $\Lambda, \r$ and the total number of points such that for every $\gG$ in \eqref{e:termwise}, we have
	\begin{equs} \label{e:replacement}
	\begin{split}
	|\gG(\Btheta, \X)| &\leq \theta^{2|\p|} \sum_{\k, \Bell} \Big(\prod_{s \in \sS} |\d_{\Btheta}^{\r} \bar{C}_{2\k_s+\m_s}^{(s)}|\Big)  \Big(\prod_{i \in \iI} |\d_{\Btheta}^{\r} C_{\p_i}^{(i)}|\Big) \Big(\prod_{\set \in \sU^{\jJ}} \frac{(C \theta)^{2\ell_\set + q_\set}}{(2\ell_\set +q_\set)!}\Big)\\ 
	&\cdot \E \bigg[ \Big(\prod_{s \in \sS} \X_{s}^{\diamond (2\k_s + \m_s)}\Big) \Big(\prod_{i \in \iI} \X_i^{\diamond \p_i}\Big) \Big(\prod_{\set \in \sU^{\jJ}} X_{\set}^{\diamond (2\ell_\set + q_\set)}\Big) \bigg], 
	\end{split}
	\end{equs}
	where the sum is taken over $\k = (\k_s)_{s \in \sS}$ where every $\k_s$ runs through $\nN_{\eE_0(\m_s)}^\tT$, and $\Bell = (\ell_\set)_{\set \in \sU^\jJ} \in \N^{\sU^\jJ}$. The bound is uniform over all locations of points. 
\end{prop}
\begin{proof}
	Applying Proposition~\ref{pr:rep} to \eqref{e:termwise}, we get
	\begin{equs}
	|\gG(\Btheta,\X)| &\leq \theta^{2|\p|} \sum_{\k,\n} \Big(\prod_{s \in \sS} |\d_{\Btheta}^{\r} \bar{C}_{2\k_s+\m_s}^{(s)}|\Big) \Big( \prod_{i \in \iI} |\d_{\Btheta}^{\r} C_{\p_i}^{(i)}|\Big) \Big( \prod_{\set \in \sU^\jJ} C^{|\n_\set|} |\d_{\Btheta_{\set}}^{\r_{\set}} C_{\n_\set}^{(\set)}|\Big)\\
	&\cdot \E \bigg[ \Big(\prod_{s \in \sS} \X_{s}^{\diamond (2\k_s + \m_s)}\Big) \Big(\prod_{i \in \iI} \X_{i}^{\diamond \p_i}\Big) \Big(\prod_{\set \in \sU^\jJ} X_{\set}^{\diamond |\n_\set|}\Big) \bigg], 
	\end{equs}
	where we have used $|\n| = \sum_{\set} |\n_\set|$ to decompose the constant $C^{|\n|}$ from Proposition~\ref{pr:rep} into factors of $C^{|\n_\set|}$ and distributed them into the corresponding terms in the product. The sum over $\k$ has the same range as in \eqref{e:termwise}, while $\n = (\n_\set)_{\set \in \sU^\jJ}$ with each $\n_\set$ running through 
	\begin{equ}
	\bar{\nN}(\set) = \big\{ \n_\set: |\n_\set| = 2 \ell_\set + q_\set, \ell_\set \geq 0 \big\}. 
	\end{equ}
    By Proposition~\ref{pr:coeff_1}, the coefficient $\d_{\Btheta_\set}^{\r_\set} C_{\n_\set}^{(\set)}$ satisfies the bound
    \begin{equ}
    |\d_{\Btheta_\set}^{\r_\set} C_{\n_\set}^{(\set)}| \leq \frac{(C \theta)^{|\n_\set|}}{\n_\set!}. 
    \end{equ}
    We then sum these coefficients over the level sets $\{\n_\set: |\n_\set| = 2\ell_\set + q_\set\}$ for every fixed $\Bell$, and the claim follows from the multinomial theorem. 
\end{proof}

Propositions~\ref{pr:rep} and~\ref{pr:replacement} hold for all possible choices of $u^{*}(\set) \in \set$. But in the context below, it will be convenient to make more specific choices based on the cluster $\set$. More precisely, we let $u^*(\set)$ be any point in $\set \cap (\jJ \times \oO)$ if $\set \in \sU_1^{\jJ}$, and arbitrary otherwise. Note that for $\set \in \sU_2^{\jJ}$, since the only point there belongs to $\jJ \times \eE$, so $u^*$ has to be even. We fix this choice throughout the rest of this section. 

The parities of the chosen points will only be used in Sections~\ref{sec:special_proof} and~\ref{sec:enhance} below. Hence, we will still use the notation $x_\set$ or $X_\set$ when we do not use those properties.

\subsection{The graphic representation}

Our aim now is to bound the right hand side of \eqref{e:replacement}, which contains the expectation of products of arbitrarily high Wick powers of Gaussians. Such an expectation can be written as a sum over products of pairwise expectations. In order to describe our objects and bounds in a convenient way, we introduce graphical notations to describe products of pairwise expectations, and we perform most operations at the graphical level. 

Given a set $\VV$, write $\VV_2$ for the set of all subsets $\{u,v\} \subset \VV$ with exactly two elements. A (generalised) graph is a pair $\Gamma = (\VV, \EE)$, where $\VV$ is a set of vertices and $\EE \colon \VV_2 \to \N$ is the set of edges with multiplicities. More precisely, each edge $\{u,v\} \in \VV_2$ has a multiplicity $\EE(u,v) = \EE(v,u)$. We do not allow self-loops, so $\EE(u,u) = 0$ for all $u \in \VV$. 

Given a graph $\Gamma = (\VV, \EE)$, for every $u \in \VV$, we define the degree of $u$ by
\begin{equ}
	\deg(u) := \sum_{v \in \VV} \EE(u,v)\;, 
\end{equ}
and the total degree of the graph $\Gamma$ is defined by
\begin{equ}
	\deg(\Gamma) := \sum_{u \in \VV} \deg(u) = 2 \sum_{\{u,v\} \subset \VV_2} \EE(u,v)\;. 
\end{equ}
In this article, the vertex set $\VV$ will always be a subset of $[K] \times \tT$, and we identify vertices of $\Gamma$ with a finite collection of space-time points $\{x_u\}$. Recall that $X_u = \eps^{\frac{1}{2}} \Psi_{\eps}(x_u)$. Defining $R: \VV_2 \rightarrow \R^{+}$ by $R(u,v) = \E X_u X_v$, we assign to $\Gamma$ a (positive) value $|\Gamma| = \prod_{e \in \VV_2} \big(R(e)\big)^{\EE(e)}$. It is then clear that $|\Gamma|$ is one of the terms appearing in the expectation $\E \big( \prod_{u} X_{u}^{\diamond \ell_u} \big)$. On the other hand, in order to encode Wick products between $X_u$ and $X_v$ for $u \neq v$, we introduce the notion of admissible graphs. 

\begin{defn} \label{de:ad_graph}
Let $\aA, \bB$ be finite index sets, let $\pi \colon \aA \to \bB$, and let $\x = \{x_a\}_{a\in \aA}$ be a 
collection of points. For every $\m \in \N^{\aA}$ and every $k \in \bB$, let $\m_k$ denote the restriction of 
$\m$ to $\pi^{-1}(k)$. Then a graph $\Gamma$ with vertex set $\aA$ is admissible with respect to $(\pi,\m)$
if $\deg a = m_a$ and if $\EE(a,b) = 0$ as soon as $\pi(a) = \pi(b)$.
\end{defn}

\begin{remark}\label{rem:admissible}
We will use Definition~\ref{de:ad_graph} in the context of bounding products of the type
\begin{equ}[e:genericWick]
\prod_{b \in \bB}  \Big(\bdiamond_{a \in \pi^{-1}(b)} \X_{a}^{\diamond m_a}\Big)\;.
\end{equ}
In this case, we also say that ``$\Gamma$ is admissible with respect to \eqref{e:genericWick}'',
with $\pi$ and $\m$ implied from the expression appearing on the right hand side.
\end{remark}

Given a graph of total degree $N$, we would like to bound its value by that of a new graph with smaller degree,
obtained by an explicit operation on the edges of the original graph. We introduce a few more notions 
to better describe these operations. 

Since we will always ignore isolated vertices, we say that $\Gamma' = (\VV', \EE')$ is a subgraph of $\Gamma = (\VV, \EE)$ if $\VV' = \VV$ and $\EE' \leq \EE$. We say that the collection $\{\Gamma_k = (\VV, \EE_k)\}$ of subgraphs of $\Gamma$ adds up to $\Gamma = (\VV, \EE)$ if $\sum_{k} \EE_k = \EE$. The definition of $|\Gamma|$ then implies that
$|\Gamma| = \prod_{k} |\Gamma_k|$ if $\{\Gamma_k\}$ adds up to $\Gamma$. Lemma~\ref{le:corr_ff}
also simply translates into the following lemma. 

\begin{lem} \label{le:reduction}
	For $a, b, c \geq 0$, we have the bound
	\begin{equ}
		\left|
		\begin{tikzpicture}[scale=1,baseline=-0.65cm]
		\node at (-0.8,0) [dot] (left){};
		\node at (0.8,0) [dot] (right) {};
		\node at (0,-1) [dot] (below) {}; 
		\node at (0,-1.3) {\scriptsize $x$}; 
		\node at (-1,0.2) {\scriptsize $y$}; 
		\node at (1,0.2) {\scriptsize $z$}; 
		\draw (left) to node[labl]{\tiny $a+1$} (below); 
		\draw (right) to node[labl]{\tiny $b+1$} (below); 
		\draw (left) to node[labl]{\tiny $c$} (right); 
		\end{tikzpicture}
		\right|
		\leq \frac{2 \Lambda^{3} \eps}{\min\big\{|x-y|,|x-z|\big\} + \eps} \cdot
		\left|
		\begin{tikzpicture}[scale=1,baseline=-0.65cm]
		\node at (-0.8,0) [dot] (left){};
		\node at (0.8,0) [dot] (right) {};
		\node at (0,-1) [dot] (below) {}; 
		\node at (0,-1.3) {\scriptsize $x$}; 
		\node at (-1,0.2) {\scriptsize $y$}; 
		\node at (1,0.2) {\scriptsize $z$}; 
		\draw (left) to node[labl]{\tiny $a$} (below); 
		\draw (right) to node[labl]{\tiny $b$} (below); 
		\draw (left) to node[labl]{\tiny $c+1$} (right); 
		\end{tikzpicture}
		\right|, 
	\end{equ}
	where $\Lambda$ is the same as in \eqref{e:corr_ff}. 
\end{lem}

Note that there is a factor $\eps$ in the numerator since an edge between two points $u$ and $v$ stands for the correlation $\E(X_u X_v) = \eps \E \big( \Psi_{\eps}(u) \Psi_{\eps}(v) \big)$.

\begin{rmk}
	Strictly speaking, the value of a graph depends on $\eps$ since the random field itself does, so a more proper notation is $|\Gamma|_{\eps}$ instead of $|\Gamma|$. However, as we shall see later, all the bounds we obtain in this section are uniform in $\eps$, so we omit it in notation for simplicity. 
\end{rmk}

\subsection{Backward induction}
\label{sec:induction}

Proposition~\ref{pr:replacement} reduces proving Theorem~\ref{th:general_bound} to bounding the right hand side of \eqref{e:replacement}. It involves a sum over indices $\k_s$ and $\ell_\set$ all the way to infinity. In order to control it by polynomials up to a fixed degree to match the right hand side of \eqref{e:general_bound}, we control correlation functions of products of Wick polynomials with high degrees by those with lower degrees via a backward induction argument. 

Recall the index sets $\sS, \iI, \jJ$ and collection of sub-clusters $\sU^\jJ$ described above. We fix 
as before an arbitrary $u^*(\set) \in \set$ for every $\set \in \sU^\jJ$ and write 
$x_{\set} = x_{u^*(\set)}$, $X_{\set} = X_{u^*(\set)}$, etc. Also recall the notations 
$\x_s = (x_{s,t})_{t \in \tT}$ and $\x_i = (x_{i,t})_{t \in \tT}$. In this subsection, we fix the vertex set to be
\begin{equ}[e:vertex_backward]
\VV = ((\sS \cup \iI) \times \tT) \cup \{ u^*(\set)\,:\, \set \in \sU^\jJ\}\;.
\end{equ}
Henceforth, for $i \in \sS \cup \jJ$, $\x_i$ may denote either a $\tT$-tuple $(x_{i,t})_{t \in \tT}$ or a collection of $|\tT|$ points $\{x_{i,t}\}_{t \in \tT}$. Fix $\m = (\m_s)_{s \in \sS} \in \rR^{\sS}$, $\p = (\p_i)_{i \in \iI} \in \mM^{\iI}$ and $\q=(q_\set)_{\set \in \sU^\jJ} \in \N^{\sU^\jJ}$. Here in this subsection, we do not use any of the properties of $\mM$, $\rR$ or $\q$ described before, so they can be any subset of $\N^{\tT}$ (or any integer for $q_\set$). In particular, Proposition~\ref{pr:induction} below does not depend on $\mM$ satisfying Assumption~\ref{as:removal_chaos}. 

For every pair $(\k,\Bell)$ such that
\begin{equ}
\k = (\k_s)_{s \in \sS} \in \bigtimes_{s \in \sS} \N_{\eE_0(\m_s)}^\tT \quad \text{and} \quad \Bell = (\ell_\set)_{\set \in \sU^\jJ} \in \N^{\sU^\jJ}, 
\end{equ}
let $\Omega_{\k,\Bell}$ be the set of admissible graphs, in the sense of Remark~\ref{rem:admissible}, for the product
\begin{equ}[e:defOkl]
\Big(\prod_{s \in \sS} \X_{s}^{\diamond (2\k_s + \m_s)}\Big) \Big(\prod_{i \in \iI} \X_{i}^{\diamond \p_i}\Big) \Big(\prod_{\set \in \sU^\jJ} X_{\set}^{\diamond (2\ell_\set +q_\set)}\Big). 
\end{equ}
Our aim is to control the value of graphs in $\Omega_{\k,\Bell}$ for large $(\k,\Bell)$ by those in $\Omega_{\k',\Bell'}$ with smaller $(\k',\Bell')$. To be more precise about the type of graphs which control those in $\Omega_{\k,\Bell}$, we introduce the following definition.

\begin{defn} \label{de:min_graph}
Let $\Omega^{*}$ be the set of graphs $\Gamma$ with vertex set $\VV$ given in \eqref{e:vertex_backward} such that all of the following hold: 
\begin{enumerate}
	\item $\Gamma \in \Omega_{\k,\Bell}$ for some $(\k,\Bell)$. 
	
	\item If $k_s^t \geq 1$ for some $(s,t) \in \sS \times \tT$, then precisely one of the following is true for the corresponding point $x_{s,t}$: 
	\begin{itemize}
		\item There exists $s' \in \sS \setminus \{s\}$ such that $\EE(x_{s,t}, x_v)=0$ whenever $v \notin \{s'\} \times \tT$. Furthermore, if $\EE(x_{s,t}, x_{s',t'}) \geq 2$ for that $s'$ and for some $t'$, then $k_{s'}^{t'}=0$. 
		
		\item There exists $\set \in \sU^\jJ$ with $\ell_\set=0$ such that $\EE(x_{s,t}, x_v)=0$ whenever $v \neq \set$. 
		
		\item There exists $i \in \iI$ such that $\EE(x_{s,t}, x_v)=0$ whenever $v \notin \{i\} \times \tT$. 
	\end{itemize}

    \item If $\ell_\set \geq 1$ for some $\set \in \sU^\jJ$, then precisely one of the following is true for the corresponding point $x_{\set}$: 
    \begin{itemize}
    	\item There exists $s \in \sS$ such that $\EE(x_\set, x_{v})=0$ whenever $v \notin (\{s\} \cup \iI) \times \tT$. Furthermore, if $\EE(x_\set, x_{s,t}) \geq 1$ for that $s$ and some $t \in \tT$, then $k_s^t=0$. 
    	
    	\item There exists $\set' \in \sU^\jJ \setminus \{\set\}$ such that $\EE(x_{\set}, x_v)=0$ whenever $v \notin \{\set'\} \cup (\iI  \times \tT)$. Furthermore, if $\EE(x_{\set}, x_{\set'}) \geq 2$, then $\ell_{\set'}=0$. 
    \end{itemize}
\end{enumerate}
\end{defn}

\begin{rmk}
The second and third conditions above impose some constraints on the pair $(\k,\Bell)$ in the first condition. In particular, $\Omega^{*} \cap \Omega_{\k,\Bell} \neq \emptyset$ only when
\begin{equ} \label{e:constraint_deg}
2k_s^t + m_s^t \leq \max_{\tilde{s},i,\tilde{\set}} \big\{ |\m_{\tilde{s}}| + |\tT|, |\p_i|, q_{\tilde{\set}} \big\} \phantom{1} \text{and} \phantom{1} 2 \ell_\set + q_\set \leq \max_{\tilde{s},\tilde{\set}} \big\{ |\m_{\tilde{s}}|, q_{\tilde{\set}}+1 \big\} + |\p|\;,
\end{equ}
for all $(s,t) \in \sS \times \tT$ and all $\set \in \sU^\jJ$, where the maximum over $\tilde{s}$ and $\tilde{\set}$ are taken over $\sS$ and $\sU^\jJ$ respectively. As a consequence, there can be only finitely many graphs (depending on $|\m|$, $|\p|$, $|\q|$ and $|\tT|$) in $\Omega^{*}$. 
\end{rmk}

\begin{rmk}
	The main difference between the second and third constraints is that in the third one, we do not remove ``extra'' edges between $x_{\set}$ and $x_{i,t}$. The main reason is that there is no assumption on the distance between $\x_i$ and the clusters $\{\set \in \sU^\jJ\}$. Reducing the situation to one with the third constraint analogous to the second one 
would cause some additional powers of $\theta$ in the bound, which compensates the decrease of the degrees. It also 
complicates the argument, so we leave it as it is for simplicity. This will not make a difference in the final 
statement since the regularity we require for $F$ follows from Theorem~\ref{th:special_bound}, which does not involve any point in $\iI \times \tT$. 
\end{rmk}

\begin{prop} \label{pr:induction}
	There exists $C\ge 1$ depending only on the value $\Lambda$ in \eqref{e:corr_ff} such that
	\begin{equ} \label{e:induction_0}
	\max_{\Gamma \in \Omega_{\k, \Bell}} \theta^{\deg(\Gamma)} |\Gamma| \leq C^{|\k| + |\Bell|} 
	\max_{\Gamma^* \in \Omega^{*}} \theta^{\deg( \Gamma^*)} | \Gamma^*|\;,
	\end{equ}
	for every pair $(\k, \Bell)$. 
\end{prop}
\begin{proof}
We  claim that 
whenever $\Gamma \in \Omega_{\k,\Bell} \setminus \Omega^*$, there exists a $\bar \Gamma \in \Omega_{\bar \k,\bar \Bell}$ 
for some $\bar \k \le \k$ and $\bar \Bell\le \Bell$ with $\bar \k+\bar \Bell < \k+ \Bell$ such that
\begin{equ} \label{e:induction}
	 |\Gamma| \leq C^{|\k-\bar \k| + |\Bell-\bar \Bell|} \theta^{\deg(\bar \Gamma)-\deg(\Gamma)}  |\bar \Gamma|\;.
\end{equ}
This bound can then be iterated until one reaches $\bar \Gamma \in \Omega^*$. 
Since furthermore $\Omega_{\0,\0} \subset \Omega^*$, this necessarily happens after at most 
$|\k + \Bell|$ steps, which then concludes the proof. It remains to exhibit a $\bar \Gamma$ as above
for every $\Gamma$. We distinguish four different cases which cover the set $\Omega_{\k,\Bell} \setminus \Omega^*$.

\medskip\noindent\textit{Case 1:} There exists $(s,t) \in \sS\times \tT$ such that 
$k_s^t \ge 1$ and such that we can find two other vertices $u$ and $v$ with $\EE((s,t),u) \wedge \EE((s,t),v) \ge 1$
and such that $\bar \Gamma \in \Omega_{\k,\Bell}$ does \textit{not} imply $\EE(u,v) = 0$. 
In this case, we define $\bar \Gamma$ so that it differs from $\Gamma$ solely by decreasing $\EE((s,t),u)$
and $\EE((s,t),v)$ by $1$, while at the same time increasing $\EE(u,v)$ by $1$.
It is then immediate that $\bar \Gamma \in \Omega_{\bar \k,\bar\Bell}$ with $\bar \k = \k - \one_{(s,t)}$
and $\bar\Bell = \Bell$. By Lemma~\ref{le:reduction} and the fact that $\min\{ |x_{s,t}-x_u|, |x_{s,t}-x_v| \} \geq \theta^2 \eps$ by the definition of $\sS$ (and the fact that neither $u$ nor $v$ belong to $\{s\}\times \tT$
by the definition of $\Omega_{\k,\Bell}$), one has the bound
\begin{equ}\left|
	\begin{tikzpicture}[scale=1,baseline=-0.8cm]
	\node at (-0.8,0) [dot] (left){};
	\node at (0.8,0) [dot] (right) {};
	\node at (0,-1) [dot] (below) {}; 
	\node at (0,-1.3) {\scriptsize $(s,t)$}; 
	\node at (-1,0.2) {\scriptsize $u$}; 
	\node at (1,0.2) {\scriptsize $v$}; 
	\draw (left) to node[labl]{\tiny $a+1$} (below); 
	\draw (right) to node[labl]{\tiny $b+1$} (below); 
	\draw (left) to node[labl]{\tiny $c$} (right); 
	\end{tikzpicture}\right|
	\leq 2 \Lambda^{3} \theta^{-2}
	\left|
	\begin{tikzpicture}[scale=1,baseline=-0.8cm]
	\node at (-0.8,0) [dot] (left){};
	\node at (0.8,0) [dot] (right) {};
	\node at (0,-1) [dot] (below) {}; 
	\node at (0,-1.3) {\scriptsize $(s,t)$}; 
	\node at (-1,0.2) {\scriptsize $u$}; 
	\node at (1,0.2) {\scriptsize $v$}; 
	\draw (left) to node[labl]{\tiny $a$} (below); 
	\draw (right) to node[labl]{\tiny $b$} (below); 
	\draw (left) to node[labl]{\tiny $c+1$} (right); 
	\end{tikzpicture}\right|\;.
\end{equ}
Since $\deg \bar \Gamma = \deg\Gamma - 2$, this bound is indeed of the from \eqref{e:induction} as required.

\medskip\noindent\textit{Case 2:} There exists $(s,t) \in \sS\times \tT$ such that 
$k_s^t \ge 1$ and one of the following two conditions hold:
\begin{itemize}
\item There exists $u = (s',t') \in \sS \times \tT$ with $s' \neq s$ and $k_{s'}^{t'} \geq 1$ such that $\EE((s,t), u) \geq 2$. 		
\item There exists $u = \set \in \sU^\jJ$ with $\ell_{\set} \geq 1$ and $\EE((s,t),u) \geq 2$.
\end{itemize}
In this case, we define $\bar \Gamma$ so that it differs from $\Gamma$ solely by decreasing $\EE((s,t),u)$ by $2$,
so that $\bar \Gamma \in \Omega_{\bar \k,\bar\Bell}$ with $\bar \k = \k - \one_{(s,t)} - \one_{(s',t')}$
and $\bar \Bell = \Bell$ in the first case, while $\bar \k = \k - \one_{(s,t)}$
and $\bar \Bell = \Bell- \one_{\set}$ in the second case.
By Lemma~\ref{le:corr_ff} and the fact that $|x_{s,t}-x_u| \geq \theta^2 \eps$, we then obtain the bound
\begin{equ}[e:boundSingle]
	\left|
    \begin{tikzpicture}[scale=1,baseline=-0cm]
    \node at (-0.8,0) [dot] (left){};
    \node at (0.8,0) [dot] (right) {};
    \node at (-1,0.2) {\scriptsize $(s,t)$}; 
    \node at (1,0.2) {\scriptsize $u$}; 
    \draw (left) to node[labl]{\tiny $a+2$} (right); 
    \end{tikzpicture}
    \right|
    \phantom{1} \leq \Lambda^{2} \theta^{-4} \phantom{1}
    \left|
    \begin{tikzpicture}[scale=1,baseline=-0cm]
    \node at (-0.8,0) [dot] (left){};
    \node at (0.8,0) [dot] (right) {};
    \node at (-1,0.2) {\scriptsize $(s,t)$}; 
    \node at (1,0.2) {\scriptsize $u$}; 
    \draw (left) to node[labl]{\tiny $a$} (right); 
    \end{tikzpicture}\right| , 
\end{equ}
Since this time $\deg \bar \Gamma = \deg\Gamma - 4$, this is again of the from \eqref{e:induction} as required.

\medskip\noindent\textit{Case 3:}    
There exists $\set \in \sU^\jJ$ with $\ell_\set \geq 1$ and  such that we can find 
two other vertices $u$ and $v$ with $\EE(\set,u) \wedge \EE(\set,v) \ge 1$ such that $u,v \not\in \iI\times \tT$ 
and such that $\bar \Gamma \in \Omega_{\k,\Bell}$ does \textit{not} imply $\EE(u,v) = 0$.
In this case, we proceed exactly as in Case~1, with $(s,t)$ replaced by $\set$.
Here, the fact that $u,v \not\in \iI\times \tT$ is crucial to guarantee that both points are at distance
at least $\theta^2\eps$ from $x_{s,t}$.

\medskip\noindent\textit{Case 4:}    
There exist $\set, \set' \in \sU^\jJ$ such that $\ell_\set \geq 1$ and 
$\EE(x_{\set}, x_{\set'}) \geq 2$.
In this case we proceed as in Case~2,
noting that $\bar \Gamma \in \Omega_{\bar \k,\bar\Bell}$ with $\bar \k = \k$
and $\bar \Bell = \Bell - \one_\set - \one_{\set'}$
and that we have again the bound \eqref{e:boundSingle}, but with $(s,t)$ replaced by $\set$,
which is again of the from \eqref{e:induction} as required.

\medskip\noindent
It remains to show that the four cases above do indeed cover all of $\Omega_{\k,\Bell} \setminus \Omega^*$.
Comparing these cases to Definition~\ref{de:min_graph}, the only way in which $\Gamma$ could possibly fail 
to belong to $\Omega^*$
which is not obviously covered by these cases is to have points $\set \in \sU^\jJ$ and 
$(s,t) \in \sS \times \tT$ such that $\EE((s,t),\set) = 1$, $\ell_\set \ge 1$, and $k_s^t \ge 1$. 
However, this case must either be covered by Case~1, or $(s,t)$ is only connected to $\set$, in which case
one must have $\EE((s,t),\set) = 2 k_s^t + m_s^t \ge 2$ by the definition of $\Omega_{\k,\Bell}$, implying that it is
covered by Case~4.
\end{proof}

\subsection{Proof of Theorem~\ref{th:special_bound}}
\label{sec:special_proof}

We are now ready to prove Theorem~\ref{th:special_bound}. We will show it for our fixed type space $\tT = \oO \sqcup \eE$, and the statement of the theorem for $\tilde{\tT}$ is just a change of notation. We will mainly focus on the bound \eqref{e:special1}, and briefly explain how one can obtain \eqref{e:special2} by slightly modifying the argument. 

We assume $|\oO|$ is even, for otherwise both sides of \eqref{e:special1} and \eqref{e:special2} vanish and there is nothing to prove. Let $\Btheta = (\theta_t)_{t \in \tT}$ be a collection of frequencies and $\theta = 1 + \max_{t}|\theta_t|$ as before. The left hand side of \eqref{e:special1} corresponds to the case $K=1$ and $\mM = \emptyset$ in the identity \eqref{e:linear_comb}. We choose $\sS = \emptyset$, so there are two terms on its right hand side. If $\iI$ is not empty, then the corresponding term vanishes since $\mM$ is empty. Hence, we only need to control \eqref{e:termwise} in the case $\iI = \emptyset$ and $\jJ = [K] = \{1\}$. We will then simply drop the notation involving $\jJ$ or elements from it. For example, we will write $\sU$ instead of $\sU^{\jJ}$, and $u \in \eE$ instead of $u \in \{1\} \times \eE$, etc.

For $j\in\{0,1,2\}$, recall the definition of $\sU_j$ from \eqref{e:divide_UJ}. 
By Proposition~\ref{pr:replacement}, we then have the bound
\begin{equ} \label{e:special_1st}
|\E \d_{\Btheta}^{\r} \Phi(\Btheta,\X)| \leq \sum_{\Bell \in \N^{\sU}} \E \Big[ \prod_{\set \in \sU} \frac{(C \theta)^{2\ell_\set + q_\set}}{(2\ell_\set+q_\set)!} X_{\set}^{\diamond (2\ell_\set+q_\set)} \Big], 
\end{equ}
where $C$ depends on $\Lambda$, $|\tT|$ and $\r$ only, and $q_\set = j$ if $\set \in \sU_j$ ($j=0,1,2$). As before, we 
write $|\Bell|=\sum_{\set}\ell_\set$ and $|\q|=\sum_{\set}q_\set$. For $\Bell \in \N^{\sU}$, let $\Omega_{\Bell}$ denote 
the collection of admissible graphs for the product $\prod_{\set \in \sU} X_{\set}^{\diamond (2\ell_\set+q_\set)}$, which 
is consistent with the notation in the previous section since $\sS = \iI = \emptyset$. Since there are at most 
$(2 |\Bell| + |\q| - 1)!!$ graphs in $\Omega_{\Bell}$, we can further control \eqref{e:special_1st} by
\begin{equ} \label{e:special_2nd}
|\E \d_{\Btheta}^{\r}(\Btheta,\X)| \leq \sum_{\Bell \in \N^{\sU}} \bigg( \frac{C^{2|\Bell|+|\q|}}{(2\Bell+\q)!} \cdot (2|\Bell|+|\q|-1)!! \cdot \max_{\Gamma \in \Omega_{\Bell}} \big( \theta^{2|\Bell|+|\q|} |\Gamma| \big)\bigg), 
\end{equ}
where we have used the shorthand notation $(2\Bell+\q)! = \prod_{\set} (2\ell_\set+q_\set)!$. Applying Proposition~\ref{pr:induction} to control the term $\theta^{2|\Bell|+|\q|}|\Gamma|$ and using the multinomial theorem for the sum over $\Bell$, we get
\begin{equ} \label{e:special_3rd}
|\E \d_{\Btheta}^{\r}\Phi(\Btheta,\X)| \leq C \max_{\Gamma \in \Omega^{*}} \big( \theta^{\deg(\Gamma)} |\Gamma| \big)\;. 
\end{equ}
We want to show that the right hand side above can be controlled by the right hand side of \eqref{e:special1}. 
In order to show this, we note first that since we are in the case $\sS = \iI = \emptyset$, the nodes of the
graphs $\Gamma \in \Omega^*$ are indexed by elements of $\sU$. Furthermore, these graphs $\Gamma$ are such that
nodes indexed by $\set \in \sU_j$ have degree $j$, with the exception of nodes in $\sU_0$ that can have degree either $0$ or $2$.
In other words, there exists a (unique) $A \subset \sU_0$ such that $\Gamma$ is admissible for the product
\begin{equ} \label{eq:special_admissible}
\Big( \prod_{\set \in \sU_1} X_{\set} \Big) \cdot \Big( \prod_{\set \in \sU_{2} \cup A} X_{\set}^{\diamond 2} \Big)\;. 
\end{equ}
We then show that for any $\Gamma \in \Omega^*$ there exists a graph $\Enh(\Gamma)$ with nodes indexed by $\tT$
such that $\Enh(\Gamma)$ is admissible for the product appearing on the right hand side of \eqref{e:special1}
and such that 
\begin{equ}[e:wanted]
|\Gamma| \lesssim  \theta^{\deg(\Enh(\Gamma)) - \deg \Gamma}|\Enh(\Gamma)|\;.
\end{equ}
Since $\deg(\Enh(\Gamma)) = |\oO| + 2|\eE|$ by the definition of admissibility, the claimed bound then follows 
immediately from the fact that the covariances of the random variables $X_t$ are all positive. 

As previously, $\Enh(\Gamma)$ is built iteratively and it suffices to show that 
\eqref{e:wanted} holds at each step of the iteration.
At each step, one of the nodes of $\Gamma$ indexed by $\set$ is replaced by a collection of nodes
indexed by the elements $u \in \set$.
We again describe our enhancement procedure in graphic notations. We use \tikz[baseline=-3] \node [odot] {}; to represent a point in $\oO$, and \tikz[baseline=-3] \node [edot] {}; to represent a point in $\eE$. A point \tikz[baseline=-3] \node [var] {}; in a grey area represents a point outside the cluster under consideration, and the parity of that point does not matter. Also, a grey area is not necessarily a cluster -- it just means an area outside the cluster in consideration. In particular, two \tikz[baseline=-3] \node [var] {};'s drawn in the same grey area may belong to different clusters. 

For any cluster $\set \in \sU$, we write $\oO_{\set} = \oO \cap \set$ and $\eE_{\set} = \eE \cap \set$\protect\footnote{These are different notations from $\oO_j$ and $\eE_j$ defined at the beginning of the section. }. We also let $u^{*} = u^{*}(\set)$ denote the representative point of $\set$. Fix an arbitrary $\Gamma \in \Omega^{*}$, and let $A \subset \sU_0$ be the set in \eqref{eq:special_admissible}. For this $\Gamma$, we have $\deg(u^{*}(\set)) = 1$ for every $\set \in \sU_1$, $\deg(u^{*}(\set)) = 2$ for every $\set \in \sU_2 \cup A$, and $0$ otherwise. We now iterate over all clusters and construct $\Enh(\Gamma)$ in the following way.
\begin{enumerate}
	\item If $\set \in \sU_2$, then $\set$ consists of a single point $u=u^{*}$ with $\deg(u) = \deg(\set)=2$, and we do not change anything, except for relabeling $\set$ by $u$. 
	
	\item For $\set \in \sU_1$, $u^*$ belongs to $\oO$. By definition of $\Omega^{*}$, $\deg(u^*) = 1$ and this edge connects to a point outside $\set$. Then, depending on whether $|\eE_\set|=1$ or not, we do the following operations:
	\begin{equs} \label{e:op_h1}
	\begin{split}
	\begin{tikzpicture} [scale=0.7,baseline=2]
	\draw (0,0) ellipse (40pt and 25pt); 
	\node[cloud, cloud puffs=7.7, cloud ignores aspect, minimum width=1cm, minimum height=1.1cm,  draw=lightgray, fill=lightgray]  at (0,2) {};
	\node at (0,2) [var] (up) {}; 
	\node at (-0.7,0) [odot] (left) {}; 
	\node at (0.5,0) [] (right) {$\dots$}; 
	\node at (-0.7,-0.3) {\tiny $u^*$};
	\draw (up) to (left); 
	\end{tikzpicture}
	\quad &\leq C \theta^{|\oO_\set|-1+2 |\eE_\set|} \phantom{1}
	\begin{tikzpicture} [scale=0.6,baseline=2]
	\draw (0,0) ellipse (65pt and 40pt); 
	\node[cloud, cloud puffs=7.7, cloud ignores aspect, minimum width=1cm, minimum height=1.1cm,  draw=lightgray, fill=lightgray]  at (0,2.5) {};
	\node at (0,2.5) [var] (up) {}; 
	\node at (-1.5,0) [odot] (left) {}; 
	\node at (-1.5,-0.3) {\tiny $u^*$};
	\node at (-0.5,-0.4) [odot] (oddleftup) {}; 
	\node at (-0.5,-0.9) [odot] (oddleftdown) {}; 
	\node at (0.2,-0.65) (oddpts) {$\dots$}; 
	\node at (0.9,-0.4) [odot] (oddrightup) {}; 
	\node at (0.9,-0.9) [odot] (oddrightdown) {}; 
	\node at (-1,0.2) [edot] (evenleft1) {}; 
	\node at (-0.3,0.2) [edot] (evenleft2) {}; 
	\node at (0.2,0.2) (evenpts) {\scriptsize $\dots$}; 
	\node at (0.7,0.2) [edot] (evenright1) {}; 
	\node at (1.4,0.2) [edot] (evenright2) {}; 
	\draw [bend right=50] (up) to (left); 
	\draw (oddleftup) to (oddleftdown); 
	\draw (oddrightup) to (oddrightdown); 
	\draw (evenleft1) to (evenleft2); 
	\draw (evenright1) to (evenright2); 
	\draw [bend left=60] (evenleft1) to (evenright2); 
	\end{tikzpicture}\;, 
	\quad |\eE_\set| \neq 1, \\
	\begin{tikzpicture} [scale=0.7,baseline=2]
	\draw (0,0) ellipse (40pt and 25pt); 
	\node[cloud, cloud puffs=7.7, cloud ignores aspect, minimum width=1cm, minimum height=1.1cm,  draw=lightgray, fill=lightgray]  at (0,2) {};
	\node at (0,2) [var] (up) {}; 
	\node at (-0.7,0) [odot] (left) {}; 
	\node at (0.5,0) [] (right) {$\dots$}; 
	\node at (-0.7,-0.3) {\tiny $u^*$};
	\draw (up) to (left); 
	\end{tikzpicture}
	\quad &\leq C \theta^{|\oO_\set|-1+2 |\eE_\set|} \phantom{1}
	\begin{tikzpicture} [scale=0.7,baseline=2]
	\draw (0,0) ellipse (55pt and 35pt); 
	\node[cloud, cloud puffs=7.7, cloud ignores aspect, minimum width=1cm, minimum height=1.1cm,  draw=lightgray, fill=lightgray]  at (0,2.5) {};
	\node at (0,2.5) [var] (up) {}; 
	\node at (-0.8,0.7) [odot] (left) {}; 
	\node at (-0.8,0.45) {\tiny $u^*$};
	\node at (-0.9,0) [odot] (oddleftup) {}; 
	\node at (-0.9,-0.7) [odot] (oddleftdown) {}; 
	\node at (0,-0.35) (oddpts) {$\dots$}; 
	\node at (0.9,0) [odot] (oddrightup) {}; 
	\node at (0.9,-0.7) [odot] (oddrightdown) {}; 
	\node at (0.8,0.7) [edot] (even) {}; 
	\draw (up) to (even);
	\draw (left) to (even);  
	\draw (oddleftup) to (oddleftdown); 
	\draw (oddrightup) to (oddrightdown); 
	\end{tikzpicture}\;, 
	\qquad |\eE_\set|=1. 
	\end{split}
	\end{equs}
When $|\eE_\set| \neq 1$, we pair all points in $\set \cap \oO$ except $u^*$ (this is possible since there are an even number of them), and link all points in $\set \cap \eE$ cyclically. When $|\eE_\set|=1$, we connect $u^*$ with the unique even point, then move the previously existing edge attached to $u^*$ to that even point 
as well, and finally pair all the remaining ``odd'' points. The corresponding bounds are immediate from the fact that the covariance between points belonging to the same cluster is greater than $\theta^{-2}$ by Lemma~\ref{le:corr_ff}, as well as Lemma~\ref{le:corr_change} which allows us to move endpoints of edges within a given cluster at the cost of some fixed multiplicative constant.
	
	\item For $\set \in A$, the situation is more complicated. By \eqref{eq:special_admissible}, $u^*$ has two edges and both of them are connected to points outside $\set$. There are four possibilities depending on whether $u^* \in \oO$ or $\eE$, and whether $|\eE_\set|=1$ or not. If $u^* \in \oO$ and $|\eE_\set|=1$, we perform the operation
	\begin{equ} \label{e:op_h00}
	\begin{tikzpicture} [scale=0.7,baseline=2]
	\draw (0,0) ellipse (35pt and 25pt); 
	\node[cloud, cloud puffs=7.7, cloud ignores aspect, minimum width=1.5cm, minimum height=1.1cm,  draw=lightgray, fill=lightgray]  at (0,2) {};
	\node at (-0.5,2) [var] (upleft) {}; 
	\node at (0.5,2) [var] (upright) {}; 
	\node at (-0.7,0) [odot] (left) {}; 
	\node at (0.5,0) [] (right) {$\dots$}; 
	\node at (-0.7,-0.3) {\tiny $u^*$};
	\draw (upleft) to (left); 
	\draw (upright) to (left); 
	\end{tikzpicture}
	\quad \leq C \theta^{|\oO_\set|+2(|\eE_\set|-1)} \phantom{1}
	\begin{tikzpicture} [scale=0.7,baseline=2]
	\draw (0,0) ellipse (45pt and 35pt); 
	\node[cloud, cloud puffs=7.7, cloud ignores aspect, minimum width=1.5cm, minimum height=1.1cm,  draw=lightgray, fill=lightgray]  at (0,2.5) {};
	\node at (-0.5,2.5) [var] (upleft) {}; 
	\node at (0.5,2.5) [var] (upright) {}; 
	\node at (-0.7,-0.7) {\tiny $u^*$}; 
	\node at (-0.8,0.3) [odot] (oddleftup) {}; 
	\node at (-0.8,-0.4) [odot] (oddleftdown) {}; 
	\node at (0,-0.05) (oddpts) {$\dots$}; 
	\node at (0.8,0.3) [odot] (oddrightup) {}; 
	\node at (0.8,-0.4) [odot] (oddrightdown) {}; 
	\node at (0,0.7) [edot] (even) {}; 
	\draw (upleft) to (even);
	\draw (upright) to (even); 
	\draw (oddleftup) to (oddleftdown); 
	\draw (oddrightup) to (oddrightdown); 
	\end{tikzpicture}\;, 
	\quad u^* \in \oO, \; |\eE_\set|=1. 
	\end{equ}
	Here we have moved the two edges from $u$ to the ``even'' point, and paired all the odd points (since $A \subset \sU_0$, so every cluster in $A$ contains an even number of them). For the other three situations, we can ``move edges around'' in a similar way, pair all the odd points, and cyclically connect the even ones as in one of the situations from \eqref{e:op_h1}. In this way, we get bounds with the same power of $\theta$ as \eqref{e:op_h00}. 
	
	\item For $\set \in \sU_0 \setminus A$, all the points in this cluster have degree $0$ and $|\oO_\set|$ is even, so we have
	\begin{equ} \label{e:op_h0}
	1 \leq C \theta^{|\oO_\set|+2|\eE_\set|} \phantom{1}
	\begin{tikzpicture} [scale=0.7,baseline=1]
	\draw (0,0) ellipse (55pt and 35pt); 
	\node at (-0.8,0.7) [odot] (left) {}; 
	\node at (0.8,0.7) [odot] (right) {}; 
	\node at (-0.9,0) [odot] (oddleftup) {}; 
	\node at (-0.9,-0.7) [odot] (oddleftdown) {}; 
	\node at (0,-0.35) (oddpts) {$\dots$}; 
	\node at (0.9,0) [odot] (oddrightup) {}; 
	\node at (0.9,-0.7) [odot] (oddrightdown) {}; 
	\node at (0,0.7) [edot] (even) {}; 
	\draw (left) to (even);  
	\draw (right) to (even); 
	\draw (oddleftup) to (oddleftdown); 
	\draw (oddrightup) to (oddrightdown); 
	\end{tikzpicture}\;, 
	\qquad |\eE_\set| = 1. 
	\end{equ}
	The bound for the case $|\eE_\set| \neq 1$ is the same. 
\end{enumerate}
Iterating over all $\set \in \sU$, we obtain indeed a graph with nodes indexed by $\tT$ and such that all nodes in $\oO$ have degree $1$, while nodes in $\eE$ have degree $2$, so that it is admissible for the right hand side of \eqref{e:special1} as required, and such that \eqref{e:wanted} holds, thus concluding the proof of \eqref{e:special1}.

We now briefly explain how one can prove \eqref{e:special2}. The argument would be a simple modification of that for \eqref{e:special1}. We first note that after the same clustering and backward reduction procedure, we arrive at the bound of type \eqref{e:special_3rd}, with its left hand side replaced by that in \eqref{e:special2}. However, the difference here is that $\Omega^*$ in this case contains different set of graphs than those admissible for \eqref{eq:special_admissible}. 

In fact, since the cosines are not re-centered, the Wick power series for those points in $\sU_2$ starts at $0$th order rather than $2$nd. As a consequence, for the $q_\set$ in \eqref{e:replacement}, we will have $q_\set=0$ for $\set \in \sU_2$. Hence, it is natural in this case to define $\Omega_{\k,\Bell}$ as before, but setting now $q_\set=0$ for $\set \in \sU_2$ in \eqref{e:defOkl}. Besides this change in the definition of $\Omega_{\k,\Bell}$, we 
keep Definition~\ref{de:min_graph} for $\Omega^*$. By this definition, we see that $\Omega^*$ in this case 
simply consists of graphs admissible for the product $\prod_{\set \in \sU_1} X_{\set}$. 

Then, in the enhancement procedure for $\Gamma \in \Omega^*$, we only pair the remaining points in $\oO$ within each cluster but do not add any edges to points in $\eE$. This will produce a graph admissible for the right hand side of \eqref{e:special2}, together with the correct power of $\theta$. 

\subsection{Simplifying the interim bound for \texorpdfstring{$\gG$}{G}}
\label{sec:simplification}

The bound \eqref{e:special1} provides us with finer controls on the coefficients  $\d_{\Btheta}^{\r} \bar{C}_{2\k_s+\m_s}^{(s)}$ and $\d_{\Btheta}^{\r} C_{\p_i}^{(i)}$ in \eqref{e:replacement}. Recall the definitions of $\oO_j$ and $\eE_j$ from \eqref{e:def_subsets}. For $s \in \sS_{\max}$, $i \in \iI$, and $j=1,2$, we let $\oO_j^s = \oO_j(\m_s)$ and $\oO_j^i = \oO_j(\p_i)$. We define $\eE_j^s$ and $\eE_j^i$ similarly for $j=0,1,2$. We have the following lemma on the controls of the coefficients. 

\begin{lem} \label{le:coeff_fine}
	The coefficients $\d_{\Btheta}^{\r} \bar{C}_{2\k_s+\m_s}^{(s)}$ and $\d_{\Btheta}^{\r} C_{\p_i}^{(i)}$ satisfy the bound
	\begin{equs}
	|\d_{\Btheta}^{\r} \bar{C}_{2\k_s+\m_s}^{(s)}| &\leq C \theta^{2|\tT|} \cdot \frac{(C \theta)^{|2\k_s+\m_s|}}{(2\k_s+\m_s)!} \cdot \E \bigg[ \Big(\prod_{t \in \oO_2^s \cup \eE_1^s} X_t\Big) \Big(\prod_{t \in \oO_1^s \cup \eE_2^s} (1 + X_{t}^{\diamond 2})\Big) \Big(\prod_{t \in \eE_0^s} X_t^{\diamond 2}\Big) \bigg], \\
	|\d_{\Btheta}^{\r} C_{\p_i}^{(i)}| &\leq C \theta^{2|\tT|} \theta^{|\p_i|} \cdot \E \bigg[ \Big(\prod_{t \in \oO_2^i \cup \eE_1^i} X_t\Big) \Big(\prod_{t \in \oO_1^i \cup \eE_2^i} (1 + X_{t}^{\diamond 2})\Big) \Big(\prod_{t \in \eE_0^i} X_t^{\diamond 2}\Big) \bigg]. 
	\end{equs}
\end{lem}
\begin{proof}
	In the expression \eqref{e:coeff_expression}, for $k \in \sS \cup \iI$ and $t \in \oO_1^{k} \cup \eE_2^k$, we write
	\begin{equ}
	\cos(\theta_t X_t) = \cent{\cos(\theta_t X_t)} + \E \cos(\theta_t X_t). 
	\end{equ}
	We then expand the product and apply \eqref{e:special1} to each term in the sum. The lemma then follows immediately. Note that for the sum in the first term, we have $\oO_j^s = \oO_j(2\k_s + \m_s)$ and the same for $\eE_j^s$. 
\end{proof}

We can now further simplify the bound for $\gG$ given in Proposition~\ref{pr:replacement}. The statement is as follows.

\begin{prop} \label{pr:simplification}
	Fix $\mM \subset \N^{\tT}$ as well as a root set $\rR$ satisfying Assumption~\ref{as:removal_chaos}. Let $\sS, \iI, \jJ$ be any disjoint subsets of $[K]$ such that their union equals $[K]$. Fix $\m \in \mM^{\sS}$ and $\p \in \mM^{\iI}$. Let $\Omega^{*}$ be the collection of graphs characterised by Definition~\ref{de:min_graph} with the vertex set $\VV$ in \eqref{e:vertex_backward}. Then, we have
	\begin{equs} \label{e:simplification}
	\begin{split}
	|\gG(\Btheta,\X)| &\leq C \theta^{2|\p|+2|\tT| (|\sS|+|\iI|) + \deg(\Omega^{*})} \Big( \sum_{\Gamma^{*} \in \Omega^{*}} |\Gamma^{*}| \Big)\\
	&\cdot \bigg[\prod_{k \in \sS \cup \iI} \E \bigg( \Big(\prod_{t \in \oO_2^k \cup \eE_1^k} X_t\Big) \Big(\prod_{t \in \oO_1^k \cup \eE_2^k} (1 + X_{t}^{\diamond 2})\Big) \Big(\prod_{t \in \eE_0^k} X_t^{\diamond 2}\Big) \bigg) \bigg]
	\end{split}
	\end{equs}
	where $\deg(\Omega^{*}) = \max_{\Gamma^{*} \in \Omega^{*}} \deg(\Gamma^{*})$. The constant $C$ depends on $|\tT|$, $K$, $\r$ and $\Lambda$ only. 
\end{prop}

\begin{proof}
	Applying Lemma~\ref{le:coeff_fine} to \eqref{e:replacement}, we get
	\begin{equs}
	|\gG(\Btheta,\X)| &\leq C \theta^{2|\p| + 2|\tT|(|\sS|+|\iI|)} \bigg[\sum_{\k,\Bell} \bigg( \frac{C^{|2\k+\m|+|2\Bell+\q|}}{(2\k+\m)! (2\Bell+\q)!} \cdot \sum_{\Gamma \in \Omega_{\k,\Bell}} \theta^{\deg(\Gamma)} |\Gamma| \bigg)\bigg]\\ 
	&\cdot \bigg[\prod_{k \in \sS \cup \iI} \E \bigg( \Big(\prod_{t \in \oO_2^k \cup \eE_1^k} X_t\Big) \Big(\prod_{t \in \oO_1^k \cup \eE_2^k} (1 + X_{t}^{\diamond 2})\Big) \Big(\prod_{t \in \eE_0^k} X_t^{\diamond 2}\Big) \bigg) \bigg], 
	\end{equs}
	where the range of the sum of $(\k,\Bell)$ is the same as that in \eqref{e:replacement}. Note that the factor $\theta^{2|\p|}$ comes from the bound \eqref{e:replacement}, while the other $|2\k+\m| + |\p|$ powers of $\theta$ from Lemma~\ref{le:coeff_fine} are in the term $\theta^{\deg(\Gamma)}$. By Proposition~\ref{pr:induction} and the multinomial theorem, we can control the term in the bracket in the first line by $\sum_{\Gamma^{*} \in \Omega^{*}} \theta^{\deg(\Gamma^{*})} |\Gamma^{*}|$. The claim then follows. 
\end{proof}

\subsection{Enhancement and proof of Theorem~\ref{th:general_bound}}

\label{sec:enhance}

We now start to develop the final ingredients to prove Theorem~\ref{th:general_bound}. From now on, we take $\sS = \sS_{\max}$. We also assume $K \geq 2$ for otherwise the bound \eqref{e:general_bound} is either trivial ($\mM \neq \emptyset$) or implied by \eqref{e:special1} ($\mM = \emptyset$). 

Let the vertex set $\VV$ be the whole collection of $K|\tT|$ points $\{(k,t)\}_{k \in [K], t \in \tT}$. Fix any $\Gamma^{*} \in \Omega^{*}$ as given by Definition~\ref{de:min_graph}. For every $k \in \sS \cup \iI$, we also fix an arbitrary graph $\Gamma^{(k)}$ that is admissible for the product\protect\footnote{Here and below, when we say a graph is admissible for a product where some factors contain a sum of Wick powers, we mean that the graph is admissible for one of the terms in the expansion of the product. }
\begin{equ} \label{e:k_admissible}
\Big(\prod_{t \in \oO_2^k \cup \eE_1^k} X_t\Big) \Big(\prod_{t \in \oO_1^k \cup \eE_2^k} (1 + X_{t}^{\diamond 2})\Big) \Big(\prod_{t \in \eE_0^k} X_t^{\diamond 2}\Big). 
\end{equ}
Let
\begin{equ} \label{e:Gamma_bar}
\bar{\Gamma} = \Big( \bigcup_{k \in \sS_{\max} \cup \iI} \Gamma^{(k)} \Big) \cup \Gamma^{*}. 
\end{equ}
Note that some of the points in $\bar{\Gamma}$ may have degree $0$. In view of \eqref{e:simplification}, it suffices to control $|\bar{\Gamma}|$, possibly with a few extra powers of $\theta$, by the right hand side of \eqref{e:general_bound}. To do this, we perform an enhancement procedure to $\bar{\Gamma}$ similar to the one in Section~\ref{sec:special_proof}. The enhanced graph $\EEnh(\bar{\Gamma})$ will need to control $|\bar{\Gamma}|$ in its value and at the same time also match a subset of the terms on the right hand side of \eqref{e:general_bound}. This will immediately imply \eqref{e:general_bound}.

Before we do the enhancement, let us first check the parities of the degrees of vertices in $\bar{\Gamma}$. For every $(s,t) \in \sS_{\max} \times \tT$, the contribution to $\deg(s,t)$ comes from both $\Gamma^{(s)}$ and $\Gamma^{*}$. By \eqref{e:k_admissible} and the definition of $\Omega^{*}$, we have
\begin{equ}
\deg(s,t) = m_s^t + \one_{\{t \in \oO_2(\m_s) \cup \eE_1(\m_s) \}} \quad  (\text{mod} \phantom{1} 2). 
\end{equ}
It is straightforward to check that $\deg(s,t)$ is odd if and only if $t \in \oO$. The same is true for $(i,t) \in \iI \times \tT$ (just replace $\m_s$ by $\p_i$) so that, for every $u \in (\sS_{\max} \cup \iI) \times \tT$, $\deg(u)$ is odd if $u \in (\sS_{\max} \cup \iI) \times \oO$ and even if $u \in (\sS_{\max} \cup \iI) \times \eE$. 

We now turn to the points in $\jJ \times \tT$. The only contribution to the degrees of these points are from $\Gamma^{*}$. For every $\set \in \sU^\jJ$, by the choice of the representative point $u^* \in \set$ described in Section~\ref{sec:special_proof}, we see that $\deg(u^*)$ is odd if and only if $\set \in \sU^{\jJ}_1$. Then, $\deg(j,t)$ is even if $(j,t) \in \jJ \times \eE$, and in particular is $0$ if it is not a representative point for any $\set \in \sU^\jJ$. 

Hence, the only points whose parity of degree are inconsistent with their types are those in $(\jJ \times \oO) \setminus \{u(\set)\}_{\set \in \sU^{\jJ}_1}$. To keep track of these points, we let
\begin{equ}
\vV_{\set} = \big\{ v \in \set \cap (\jJ \times \oO): \deg(v) \; \text{is even} \big\}. 
\end{equ}
One can check that $\vV_\set = \set \cap (\jJ \times \oO)$ for $\set \in \sU^{\jJ}_0$, $\vV_\set = \big(\set \cap (\jJ \times \oO) \big) \setminus \{u^*(\set)\}$ for $\set \in \sU^{\jJ}_1$, and is empty for $\set \in \sU^{\jJ}_2$. In particular, $|\vV_\set|$ is even for every $\set \in \sU^{\jJ}$. Hence, we perform the following operation to $\bar{\Gamma}$: 
\begin{equ} \label{e:enhance_odd}
1 \leq C \theta^{|\vV_\set|} \phantom{1}
\begin{tikzpicture} [scale=0.7,baseline=2]
\draw (0,0) ellipse (55pt and 35pt); 
\node at (-0.9,0.5) [odot] (oddleftup) {}; 
\node at (-0.9,-0.5) [odot] (oddleftdown) {}; 
\node at (0,0) (oddpts) {$\dots$}; 
\node at (0.9,0.5) [odot] (oddrightup) {}; 
\node at (0.9,-0.5) [odot] (oddrightdown) {}; 
\draw (oddleftup) to (oddleftdown); 
\draw (oddrightup) to (oddrightdown); 
\end{tikzpicture}\;, 
\quad \forall\; \set \in \sU^{\jJ}, 
\end{equ}
where we paired the points in $\vV_\set$ for every $\set$ and linked each pair with an edge of multiplicity $1$. Note that this operation is performed only to points in $\vV_\set$ but not all of $\set$, so the notion \tikz[baseline=-3] \node [odot] {}; for odd points is justified. We do this operation for every $\set$. Hence, all the points in the set $\{v \in \vV_\set, \set \in \sU^\jJ\}$ are affected in this operation, which adds exactly one degree to each of these points. After this procedure, $\deg(v)$ is odd for every $v \in \jJ \times \oO$. 

Thus, at this stage, all points have parities consistent with their types: $\deg(u)$ is odd if $u \in [K] \times \oO$ and is even if $u \in [K] \times \eE$. Since we will not make use of the choice of representative points any more, we use $u$ to denote generic points rather than just the representative ones chosen before. 

We now turn to the clusters in $\Clus(\sS_{\max}^{c} \times \tT) = \Clus\big( (\iI \cup \jJ) \times \tT \big)$ (not restricting to points in $\jJ \times \tT$ any more). For every cluster $\wW$ in this collection, if $|\wW|=1$, then we do not change anything to that cluster. If $|\wW| \geq 2$, then for $\bar{N} > \max_{\p \in \mM} |\p|$, we cyclically connect \textit{all} the points in $\wW$ with edges of multiplicity $\bar{N}$ in addition to the edges and multiplicities they already have. The effect of the operation can be graphically described by
\begin{equ} \label{e:enhance_all}
1 \leq C \theta^{2 \bar{N} |\wW|} \quad
\begin{tikzpicture} [scale=0.9,baseline=2]
%\draw (0,0) ellipse (85pt and 85pt); 
\node at (-1,2) [svar] (up1) {}; 
\node at (1,2) [svar] (up2) {}; 
\node at (-1,-2) [svar] (down1) {}; 
\node at (1,-2) [svar] (down2) {}; 
\node at (-2,0.85) [svar] (left1) {}; 
\node at (-2,-0.85) [svar] (left2) {}; 
\node at (2,0.85) [svar] (right1) {}; 
\node at (2,-0.85) [svar] (right2) {}; 
\node at (-0.4,-2) () {$\cdots$}; 
\node at (0.4,-2) () {$\cdots$}; 
\draw (up1) to node[labl]{\tiny $+\bar{N}$} (up2); 
\draw (up2) to node[labl]{\tiny $+\bar{N}$} (right1); 
\draw (right1) to node[labl]{\tiny $+\bar{N}$} (right2); 
\draw (right2) to node[labl]{\tiny $+\bar{N}$} (down2); 
\draw (up1) to node[labl]{\tiny $+\bar{N}$} (left1); 
\draw (left1) to node[labl]{\tiny $+\bar{N}$} (left2); 
\draw (left2) to node[labl]{\tiny $+\bar{N}$} (down1); 
\end{tikzpicture}\;, 
\qquad |\wW| \geq 2,  
\end{equ}
where all the points on the right hand side are in the same cluster $\wW$, and ``$+\bar{N}$'' refers to that the 
integer $\bar{N}$ is added in addition to the multiplicity that might already exist. The exact order of the points in 
the cycle is arbitrary, and the notion \tikz[baseline=-3] \node [svar] {}; refers to a generic point in $\wW$ whose 
parity does not matter. We perform this operation for every $\wW \in \Clus(\sS_{\max}^{c} \times \tT)$ with 
$|\wW| \geq 2$. This operation adds $2\bar{N}$ degrees to every point in $(\iI \cup \jJ) \times \tT$ that is 
not a singleton, so the parities of the degrees do not change. The graph obtained at the end of this procedure 
is our desired enhanced graph, which we denote by $\EEnh(\bar{\Gamma})$. 

We now show that $\EEnh(\bar{\Gamma})$ does satisfy a number of properties that will allow us to deduce 
Theorem~\ref{th:general_bound} from this construction. In order to precisely describe those properties, 
for $(k,t) \in [K] \times \tT$, we define  the external degrees of $(s,t)$ and $k$ by
\begin{equ}
\deg_{\ex}(k,t) = \sum_{u \notin \{k\} \times \tT} \EE \big((k,t), u\big), \qquad \ddeg_{\ex}(k) = \big(\deg_{\ex}(k,t) \big)_{t \in \tT}, 
\end{equ}
with $|\ddeg(k)|$ being the sum of its components. We now summarise the properties of $\EEnh(\bar{\Gamma})$ in the following proposition. 

\begin{prop} \label{pr:enhance}
	Let $\EEnh(\bar{\Gamma})$ be the graph obtained from $\bar{\Gamma}$ as above. Then, we have the bound
	\begin{equ} \label{e:bound_enhance}
	|\bar{\Gamma}| \leq C \theta^{(2\bar{N}+1)K|\tT|} |\EEnh(\bar{\Gamma})|. 
	\end{equ}
	Furthermore, $\EEnh(\bar{\Gamma})$ has the following properties: 
	\begin{enumerate}\itemsep0em
		\item One has $\deg(u) \geq 1$ for all $u \in [K] \times \tT$. Furthermore, $\deg(u)$ is odd if $u \in [K] \times \oO$ 
		and even if $u \in [K] \times \eE$. 
		\item For every $s \in \sS_{\max}$, there exists $\k_s \in \N_{\eE_0(\m_s)}^\tT$ such that $\deg_{\ex}(s) = 2\k_s + \m_s$. 
		\item For every $k \in \sS_{\max}^{c}$, there exist $k' \in \sS_{\max}^{c} \setminus \{k\}$ and $(t,t') \in \tT \times \tT$ such that $\EE((k,t), (k',t')) \geq \bar{N}$. 
	\end{enumerate}
\end{prop}

\begin{rmk}
The definition of $\sS_{\max}$ excludes the possibility that $|\sS_{\max}| = K-1$. Hence, if $\sS_{\max}^{c}$ is not empty, then it has at least two points, so the third property above at least makes sense. If $\sS_{\max}=\emptyset$, then Property 2 is automatically true. On the other hand, if $\sS_{\max} = [K]$, then Property 3 is automatically true. 
\end{rmk}

\begin{proof} [Proof of Proposition~\ref{pr:enhance}]
	The bound \eqref{e:bound_enhance} follows from the bounds \eqref{e:enhance_odd} and \eqref{e:enhance_all} and that the total degree of $\theta$ is bounded by
	\begin{equ}
	\sum_{\set \in \sU^\jJ} |\vV_\set| + 2\bar{N} \sum_{\wW \in \Clus(\sS_{\max}^{c} \times \tT)} |\wW| \leq (2\bar{N}+1)K|\tT|. 
	\end{equ}
	We now turn to the three properties. For the first one, the positivity of the degrees can be checked as follows. For $s \in \sS_{\max}$, since the enhancement procedure does not change anything to points in $\sS_{\max} \times \tT$, $\deg(s,t)$ consists precisely the contributions from the original graphs $\Gamma^{(s)}$ and $\Gamma^{*}$. By \eqref{e:k_admissible} and the definition of $\Omega^{*}$, we have
	\begin{equ}
	\deg(s,t) \geq 2k_s^t + m_s^t + \one_{\{t \in \oO_2^s \cup \eE_1^s\}} + 2 \cdot \one_{\{t \in \eE_0^s\}}. 
	\end{equ}
	If $m_s^t \geq 1$, then there is nothing to prove. If $m_s^t=0$, then we have $t \in \eE_0^s = \eE_0(\m_s)$ and hence $\deg(s,t) \geq 2$. The same is true for points in $\iI \times \tT$ by noting that the enhancement procedure \eqref{e:enhance_all} can only add $2\bar{N}$ degrees to those points. For $u \in \jJ \times \tT$, there are two possibilities. If $\{u\}$ is a singleton (that is, itself alone is a cluster), then $u$ is the representative point for some $\set \in \sU^{\jJ}_1 \cup \sU^{\jJ}_2$. Hence, $\deg(u) \geq 1$. If $\{u\}$ is not singleton, the $u$ belongs to some $\wW \in \Clus(\sS_{\max}^{c} \times \tT)$ with $|\wW| \geq 2$.  By the enhancement \eqref{e:enhance_all}, we have $\deg(u) \geq \bar{N}$. This shows that all points in $\EEnh(\bar{\Gamma})$ have strictly positive degrees. 
	
	As for the consistency of the parities, we have shown this property after Procedure \eqref{e:enhance_odd}, and that it remains unchanged under Procedure \eqref{e:enhance_all}. Thus, Property 1 is true for $\EEnh(\bar{\Gamma})$. 
	
	For Property 2, note that the enhancement procedures above do not affect any point in $\sS_{\max} \times \tT$, and the graph $\Gamma^{(s)}$ only contributes to internal degrees of $\{(s,t)\}_{t \in \tT}$ (total multiplicities of edges linking to the collection of points in $\{(s,t)\}_{t \in \tT}$). Thus, all contributions to $\ddeg_{\ex} (s)$ come from $\Gamma^{*}$. The property then follows immediately from the first constraint in Definition~\ref{de:min_graph}. 
	
	Property 3 is the only one which we use $\sS = \sS_{\max}$. By definition of $\sS_{\max}$, for every $k \in \sS_{\max}^{c}$, there exists a cluster $\wW \in \Clus (\sS_{\max}^{c} \times \tT)$ which contains at least one point from both $\{k\} \times \tT$ and $(\sS_{\max}^{c} \setminus \{k\}) \times \tT$. For that $\wW$, since it contains at least two points, so it has gone through the operation~\ref{e:enhance_all}. In the cyclic ordering in that operation, there exists $t \in \tT$ such that $(k,t) \in \wW$ and is adjacent to some $(k',t')$ in the cycle for some $k' \neq k$. This implies $\EE\big((k,t), (k',t')\big) \geq \bar{N}$. 
\end{proof}

\begin{proof} [Proof of Theorem~\ref{th:general_bound}]
	We are now ready to prove the main theorem. By the bounds \eqref{e:simplification} and \eqref{e:bound_enhance}, we know that up to some power of $\theta$, $|\gG(\Btheta,\X)|$ and hence the left hand side of \eqref{e:general_bound} is controlled by $\sum |\EEnh(\bar{\Gamma})|$ where the sum is taken over all possible graphs $\bar{\Gamma}$ obtained from \eqref{e:Gamma_bar}. It then suffices to show that each $\EEnh(\bar{\Gamma})$ is admissible for the product on the right hand side of \eqref{e:general_bound} for large enough $N$. This will imply $|\EEnh(\bar{\Gamma})|$ is controlled by that right hand side, and hence completes the proof of the theorem. 
	
	To show the admissibility, we match the three properties of $\EEnh(\bar{\Gamma})$ in Proposition~\ref{pr:enhance} with those of the admissible graphs to the right hand side of \eqref{e:general_bound}. The collection of graphs admissible for the product
	\begin{equ}
	\prod_{k=1}^{K} \sT_{\mM} \Big( \prod_{t \in \oO} P_{N}(X_{k,t}) \prod_{t \in \eE} Q_{N}(X_{k,t}) \Big)
	\end{equ}
	can be completely characterised as follows: 
	\begin{enumerate}
		\item $\deg(u) \in \{1, 3, \dots, 2 N-1\}$ if $u \in [K] \times \oO$, and $\deg(u) \in \{2, 4, \dots, 2 N\}$ if $u \in [K] \times \eE$. 
		
		\item For every $k \in [K]$, $\ddeg_{\ex}(k) \notin \mM$. 
	\end{enumerate}
    For the first criterion, the consistency of parity and strict positivity of the degrees is given directly by Property 1 in Proposition~\ref{pr:enhance}. Also, the degrees are bounded by $2N-1$ and $2N$ if $N$ is large enough. 
    
    We now turn to the second one. For $s \in \sS_{\max}$, by Property 2, we have $\ddeg_{\ex}(s) = 2\k_s + \m_s \in \bB(\m_s)$. Since $\m_s \in \rR \subset \mM^{c}$, the assumption on $\mM$ explicitly states that $\bB(\m_s) \subset \mM^c$ and hence $\ddeg_{\ex}(s) \notin \mM$. For $k \in \sS_{\max}^{c}$, Property 3 implies that $|\ddeg_{\ex}(k)| \geq \bar{N} > \max_{\p \in \mM} |\p|$. Hence, we also have $\ddeg_{\ex}(k) \notin \mM$ for $k \in \sS_{\max}^{c}$. 
    
    Finally, the degree of the power in $\theta$ is also bounded by large enough $N$. This completes the proof. 
\end{proof}

\begin{rmk}
	If we carefully keep track of the powers in $\theta$ in all of the above arguments, we see that we can take $N = CK$, where $C$ depends on $\tT, \mM$ and $\rR$ only. 
\end{rmk}

\appendix

\section{The standard KPZ model} \label{sec:KPZ_model}

For the sake of completeness, we briefly describe the standard KPZ model and its approximations. These approximations correspond to the renormalised model $\hPi^{\eps}$ defined in Section~\ref{sec:rs} in the special case when $F(u) = au^2$. We denote it by $\Pi^{\KPZ;\eps}$. Note that since those models are normalised, every $a \neq 0$ gives rise to the \textit{same} model.

One of the main results in \cite{HairerKPZ} and \cite{HQ} says that $\Pi^{\KPZ;\eps}$ converges to a limit model, which we denote by $\Pi^{\KPZ}$. It is called the standard KPZ model. By stationarity, we can take the base point to be $0$ in all cases. Recall the collection of symbols in Table~\ref{e:symbols}. For \<0'> and \<2'2'0'>, we set
\begin{equ}
\Pi^{\KPZ;\eps}_{0} \<0'> = \Pi^{\KPZ;\eps}_{0} \<2'2'0'> \equiv 0. 
\end{equ}
For the symbols \<1'>, \<2'>, \<1'1'>, \<2'1'>, \<2'2'0> and \<2'1'1'> and using the same graphical notations as
in \cite{HQ}, we decompose them into different homogeneous chaos so that
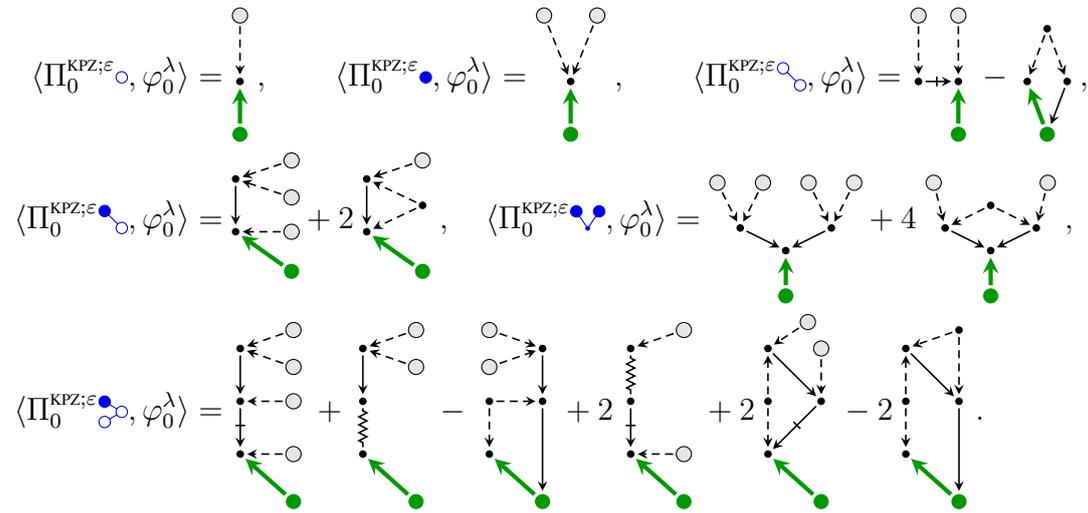
\begin{equs}
\scal{\Pi^{\KPZ;\eps}_{0} \<1'>, \varphi_{0}^{\lambda}} &= 
\begin{tikzpicture}[scale=0.35,baseline=0.3cm]
\node at (0,-1)  [root] (root) {};
\node at (0,1)  [dot] (int) {};
\node at (0,3.5)  [var] (up) {};
\draw[testfcn] (int) to (root); 
\draw[kepsilon] (up) to (int); 
\end{tikzpicture}
\;, \qquad
\scal{\Pi^{\KPZ;\eps}_{0} \<2'>, \varphi_{0}^{\lambda}} = 
\begin{tikzpicture}[scale=0.35,baseline=0.3cm]
\node at (0,-1)  [root] (root) {};
\node at (0,1)  [dot] (int) {};
\node at (-1,3.5)  [var] (left) {};
\node at (1,3.5)  [var] (right) {};
\draw[testfcn] (int) to  (root);
\draw[kepsilon] (left) to (int);
\draw[kepsilon] (right) to (int);
\end{tikzpicture}
\;, \qquad
\scal{\Pi^{\KPZ;\eps}_{0}\<1'1'>, \varphi_{0}^{\lambda}} =
\begin{tikzpicture}[scale=0.35,baseline=0.3cm]
\node at (0,-1)  [root] (root) {};
\node at (0,1)  [dot] (int) {};
\node at (-1.5,1)  [dot] (centre) {};
\node at (-1.5,3.5)  [var] (left) {};
\node at (0,3.5)  [var] (right) {};
\draw[testfcn] (int) to  (root);
\draw[kepsilon] (left) to (centre);
\draw[kernel1] (centre) to (int);
\draw[kepsilon] (right) to (int);
\end{tikzpicture}\;
- \;
\begin{tikzpicture}[scale=0.35,baseline=0.3cm]
\node at (0.75,-1)  [root] (root) {};
\node at (0,1)  [dot] (int) {};
\node at (1.5,1) [dot] (cent) {};
\node at (0.75,3) [dot] (top) {};
\draw[testfcn] (int) to  (root);
\draw[kepsilon] (top) to (int); 
\draw[kepsilon] (top) to (cent); 
\draw[kernel] (cent) to (root); 
\end{tikzpicture}\;, \\
\scal{\Pi^{\KPZ;\eps}_{0}\<2'1'>, \varphi_{0}^{\lambda}} &= 
\begin{tikzpicture}[scale=0.35,baseline=1.1cm]
\node at (0.8,1.5)  [root] (root) {};
\node at (-1.3,3)  [dot] (left1) {};
\node at (-1.3,5)  [dot] (left2) {};
\node at (0.8,3) [var] (variable2) {};
\node at (0.8,4.3) [var] (variable3) {};
\node at (0.8,5.7) [var] (variable4) {};
\draw[testfcn] (left1) to (root);
\draw[kernel] (left2) to  (left1);
\draw[kepsilon] (variable2) to  (left1); 
\draw[kepsilon] (variable3) to  (left2); 
\draw[kepsilon] (variable4) to (left2);
\end{tikzpicture}
+ 2\;
\begin{tikzpicture}[scale=0.35,baseline=1.1cm]
\node at (0.8,1.5)  [root] (root) {};
\node at (-1.3,3)  [dot] (left1) {};
\node at (-1.3,5)  [dot] (left2) {};
\node at (0.8,4) [dot] (variable2) {};
\node at (0.8,5.7) [var] (variable3) {};
\draw[testfcn] (left1) to (root);
\draw[kernel] (left2) to  (left1);
\draw[kepsilon] (variable2) to  (left1); 
\draw[kepsilon] (variable3) to  (left2); 
\draw[kepsilon] (variable2) to (left2);
\end{tikzpicture}\;, \quad
\scal{\Pi^{\KPZ;\eps}_{0}\<2'2'0>, \varphi_{0}^{\lambda}} = 
\begin{tikzpicture}[scale=0.3,baseline=.9cm]
\node at (0,0)  [root] (root) {};
\node at (0,2)  [dot] (int) {};
\node at (-2,3)  [dot] (left) {};
\node at (-1,5) [var] (variable3) {};
\node at (-3,5) [var] (variable4) {};
\node at (2,3)  [dot] (right) {};
\node at (1,5) [var] (variable1) {};
\node at (3,5) [var] (variable2) {};
\draw[testfcn] (int) to (root);
\draw[kernel] (left) to  (int);
\draw[kernel] (right) to  (int);
\draw[kepsilon] (variable2) to  (right); 
\draw[kepsilon] (variable1) to  (right); 
\draw[kepsilon] (variable3) to  (left); 
\draw[kepsilon] (variable4) to (left);
\end{tikzpicture}
+ 4\;
\begin{tikzpicture}[scale=0.3,baseline=.9cm]
\node at (0,0)  [root] (root) {};
\node at (0,2)  [dot] (int) {};
\node at (-2,3)  [dot] (left) {};
\node at (-2.5,5) [var] (variable4) {};
\node at (2,3)  [dot] (right) {};
\node at (0,4) [dot] (top) {};
\node at (2.5,5) [var] (variable2) {};
\draw[testfcn] (int) to (root);
\draw[kernel] (left) to  (int);
\draw[kernel] (right) to  (int);
\draw[kepsilon] (variable2) to  (right); 
\draw[kepsilon] (top) to  (right); 
\draw[kepsilon] (top) to  (left); 
\draw[kepsilon] (variable4) to (left);
\end{tikzpicture}\;, \\
\scal{\Pi^{\KPZ;\eps}_{0} \<2'1'1'>, \varphi_{0}^{\lambda}} &=
\begin{tikzpicture}[scale=0.35,baseline=0.8cm]
\node at (0,-0.8)  [root] (root) {};
\node at (-2,1)  [dot] (left) {};
\node at (-2,3)  [dot] (left1) {};
\node at (-2,5)  [dot] (left2) {};
\node at (0,1) [var] (variable1) {};
\node at (0,3) [var] (variable2) {};
\node at (0,4.3) [var] (variable3) {};
\node at (0,5.7) [var] (variable4) {};
\draw[testfcn] (left) to (root);
\draw[kernel1] (left1) to   (left);
\draw[kernel] (left2) to  (left1);
\draw[kepsilon] (variable2) to  (left1); 
\draw[kepsilon] (variable1) to   (left); 
\draw[kepsilon] (variable3) to   (left2); 
\draw[kepsilon] (variable4) to   (left2);
\end{tikzpicture}
\;+\;
\begin{tikzpicture}[scale=0.35,baseline=0.8cm]
\node at (0,-0.8)  [root] (root) {};
\node at (-2,1)  [dot] (left) {};
\node at (-2,3)  [dot] (left1) {};
\node at (-2,5)  [dot] (left2) {};
\node at (0,4.3) [var] (variable3) {};
\node at (0,5.7) [var] (variable4) {};
\draw[testfcn] (left) to (root);
\draw[kernelBig] (left1) to   (left);
\draw[kernel] (left2) to  (left1);
\draw[kepsilon] (variable3) to   (left2); 
\draw[kepsilon] (variable4) to   (left2);
\end{tikzpicture}
\;-\;
\begin{tikzpicture}[scale=0.35,baseline=0.8cm]
\node at (0,-0.8)  [root] (root) {};
\node at (-2,1)  [dot] (left) {};
\node at (0,3)  [dot] (left1) {};
\node at (0,5)  [dot] (left2) {};
\node at (-2,3) [dot] (variable1) {};
\node at (-2,4.3) [var] (variable3) {};
\node at (-2,5.7) [var] (variable4) {};
\draw[testfcn] (left) to (root);
\draw[kernel] (left1) to   (root);
\draw[kernel] (left2) to  (left1);
\draw[kepsilon] (variable1) to  (left1); 
\draw[kepsilon] (variable1) to   (left); 
\draw[kepsilon] (variable3) to   (left2); 
\draw[kepsilon] (variable4) to   (left2);
\end{tikzpicture}
\;+2\;
\begin{tikzpicture}[scale=0.35,baseline=0.8cm]
\node at (0,-0.8)  [root] (root) {};
\node at (-2,1)  [dot] (left) {};
\node at (-2,3)  [dot] (left1) {};
\node at (-2,5)  [dot] (left2) {};
\node at (0,1) [var] (variable1) {};
\node at (0,5.7) [var] (variable4) {};
\draw[testfcn] (left) to (root);
\draw[kernel1] (left1) to   (left);
\draw[kernelBig] (left2) to  (left1);
\draw[kepsilon] (variable1) to   (left); 
\draw[kepsilon] (variable4) to   (left2);
\end{tikzpicture}
\;+2\;
\begin{tikzpicture}[scale=0.35,baseline=0.8cm]
\node at (0,-0.8)  [root] (root) {};
\node at (-2,1)  [dot] (left) {};
\node at (0,3)  [dot] (left1) {};
\node at (-2,5)  [dot] (left2) {};
\node at (0,5) [var] (variable1) {};
\node at (-2,3) [dot] (variable3) {};
\node at (-0.5,6) [var] (variable4) {};
\draw[testfcn] (left) to (root);
\draw[kernel1] (left1) to   (left);
\draw[kernel] (left2) to  (left1);
\draw[kepsilon] (variable3) to  (left); 
\draw[kepsilon] (variable1) to   (left1); 
\draw[kepsilon] (variable3) to   (left2); 
\draw[kepsilon] (variable4) to   (left2);
\end{tikzpicture}
\;-2\;
\begin{tikzpicture}[scale=0.35,baseline=0.8cm]
\node at (0,-0.8)  [root] (root) {};
\node at (-2,1)  [dot] (left) {};
\node at (0,3)  [dot] (left1) {};
\node at (-2,5)  [dot] (left2) {};
\node at (-2,3) [dot] (variable3) {};
\node at (0,5.7) [dot] (variable4) {};
\draw[testfcn] (left) to (root);
\draw[kernel] (left1) to   (root);
\draw[kernel] (left2) to  (left1);
\draw[kepsilon] (variable3) to  (left); 
\draw[kepsilon] (variable4) to   (left1); 
\draw[kepsilon] (variable3) to   (left2); 
\draw[kepsilon] (variable4) to   (left2);
\end{tikzpicture}\;.
\end{equs}
%\martinText{
%Why do we actually need this? Is the point that these actually converge separately and that each of the terms
%individually satisfies the correct bounds?
%}\weijunText{It is mainly for reference in the convergence proof and identifying limit in Section~\ref{sec:convergence}. For each of the symbols under our model, we can write it as the sum of a main part and a remainder. The main part is exactly the same as these KPZ trees, so we can refer to here and cite directly. See for example \eqref{e:tau2_main} and \eqref{e:tau5_main}. We then show the remainder vanishes.}
The following result shown in \cite{HQ,Peter} (but see also \cite{HairerKPZ} for an earlier similar result
in a slightly different context) is the main convergence result for the KPZ model / equation. 

\begin{thm}
	The model $\Pi^{\KPZ;\eps}$ converges to a limiting model $\Pi^{\KPZ}$, which is the KPZ model. For
	\begin{equ}
	\tau \in \big\{ \<1'>, \<2'>, \<1'1'>, \<2'1'>, \<2'2'0>, \<2'1'1'> \big\}, 
	\end{equ}
	the object $\scal{\Pi_{0}^{\KPZ}\tau, \varphi_{0}^{\lambda}}$ is described by the same combination of the trees in $\scal{\Pi_{0}^{\KPZ;\eps}\tau, \varphi_{0}^{\lambda}}$ except that each decorated arrow \tikz[baseline=-0.1cm] \draw[kepsilon] (0,0) to (1,0); is replaced by the plain kernel \tikz[baseline=-0.1cm] \draw[kernel] (0,0) to (1,0);. In addition, there exists some $\kappa'>0$ such that one has the bound
	\begin{equ} \label{e:converge_error}
	\big(\E |\scal{\Pi^{\KPZ;\eps}_{0}\tau - \Pi^{\KPZ}_{0}\tau, \varphi_{0}^{\lambda}}|^{2n} \big)^{\frac{1}{2n}} \lesssim_{n} \eps^{\kappa'} \lambda^{|\tau|+\kappa'}. 
	\end{equ}
\end{thm}

\endappendix

\bibliographystyle{Martin}
\bibliography{Refs}

\end{document}